\documentclass[3p,10pt,a4paper,twoside,fleqn,sort&compress]{amsart}
\usepackage{amssymb,amsmath,latexsym}
\usepackage[varg]{pxfonts}
\usepackage[margin=1.9cm]{geometry}

\newtheorem{theorem}{Theorem}[section]

\newtheorem{corollary}[theorem]{Corollary}

\newtheorem{definition}[theorem]{Definition}
\newtheorem{example}[theorem]{Example}

\newtheorem{lemma}[theorem]{Lemma}

\newtheorem{proposition}[theorem]{Proposition}
\newtheorem{remark}[theorem]{Remark}

\newcommand{\core}{\mathrel{\text{\textcircled{$\#$}}}}

\begin{document}

\title{Generalized inverses, ideals, and projectors in rings}

\author[P. M. Morillas]{Patricia Mariela Morillas$^{*}$\\\\\textit{$^{*}$I\lowercase{nstituto de} M\lowercase{atem\'{a}tica} A\lowercase{plicada} S\lowercase{an} L\lowercase{uis} (UNSL-CONICET). E\lowercase{j\'{e}rcito de los }A\lowercase{ndes 950, 5700 }S\lowercase{an} L\lowercase{uis},
A\lowercase{rgentina}\\E\lowercase{-mail address:
morillas.unsl@gmail.com}}}


\begin{abstract}
The theory of generalized inverses of matrices and operators is
closely connected with projections, i.e., idempotent (bounded)
linear transformations. We show that a similar situation occurs in
any associative ring $\mathcal{R}$ with a unit $1 \neq 0$. We prove
that generalized inverses in $\mathcal{R}$ are related to idempotent
group endomorphisms $\rho: \mathcal{R} \rightarrow \mathcal{R}$,
called projectors. We use these relations to give characterizations
and existence conditions for $\{1\}$, $\{2\}$, and $\{1,2\}$-inverses with any given principal/annihilator ideals. As a
consequence, we obtain sufficient conditions for any right/left
ideal of $\mathcal{R}$ to be a principal or an annihilator ideal of
an idempotent element of $\mathcal{R}$. We also study some
particular generalized inverses: Drazin and $(b,c)$ inverses, and $(e,f)$ Moore-Penrose, $e$-core, $f$-dual
core, $w$-core, dual $v$-core, right $w$-core, left dual $v$-core,
and $(p,q)$ inverses in
rings with involution.

\bigskip
\bigskip


\medskip


\end{abstract}

\maketitle

\makeatletter
\renewcommand\@makefnmark%
{\mbox{\textsuperscript{\normalfont\@thefnmark)}}}
\makeatother

\section{Introduction}\label{S Introduction}

For non-invertible operators and matrices, and more generally, for
non-invertible elements of semigroups and rings, several generalized
inverses were defined and studied. Each generalized inverse is used
to study specific types of problems. They are useful for solving
matrix and operator equations (including integral and differential
equations), in probability theory, in the study of algebras, rings,
and semigroups, among others. See, e.g., \cite{Ben-Israel-Greville
(2003), Bhaskara Rao (2002), Campbell-Meyer (2009),
CvetkovicIlic-Wei (2017), Rao-Mitra (1971), Wang-Wei-Qiao (2018)}
and references therein.

Throughout this paper, $\mathcal{R}$ will be an associative ring
with a unit $1 \neq 0$. An involution $\ast$ of $\mathcal{R}$ is an
involutory anti-automorphism $a \mapsto a^{\ast}$, i.e.,
$(a^{\ast})^{\ast} = a$, $(a + b)^{\ast} = a^{\ast} + b^{\ast}$,
$(ab)^{\ast}= b^{\ast}a^{\ast}$ for all $a, b \in \mathcal{R}$. For
$a \in \mathcal{R}$, consider the \emph{principal right (resp. left)
ideal of $\mathcal{R}$ with generator $a$}, $a\mathcal{R}=\{ar : r
\in \mathcal{R}\}$ (resp. $\mathcal{R}a=\{ra : r \in
\mathcal{R}\}$), and the \emph{right} (resp. \emph{left})
\emph{annihilator of $a$}, ${\rm rann}(a)=\{r \in \mathcal{R} :
ar=0\}$ (resp. ${\rm lann}(a)=\{r \in \mathcal{R} : ra=0\}$). For
$a, x \in \mathcal{R}$, consider the following equalities:

\begin{flushleft}
\setlength{\abovedisplayskip}{0pt}
\setlength{\belowdisplayskip}{0pt}
\begin{minipage}{.24\textwidth}
\begin{equation}
axa = a,
\end{equation}
\end{minipage}
\begin{minipage}{.24\textwidth}
\begin{equation}
xax=x,
\end{equation}
\end{minipage}
\begin{minipage}{.24\textwidth}
\begin{equation}
(xa)^{\ast} = xa,
\end{equation}
\end{minipage}
\begin{minipage}{.24\textwidth}
\begin{equation}
(ax)^{\ast} = ax,
\end{equation}
\end{minipage}
\end{flushleft}

\begin{flushleft}
\setlength{\abovedisplayskip}{0pt}
\setlength{\belowdisplayskip}{0pt}
\begin{minipage}{.19\textwidth}
\begin{equation}
ax = xa,
\end{equation}
\end{minipage}
\begin{minipage}{.19\textwidth}
\begin{equation}
xa^{2} = a,
\end{equation}
\end{minipage}
\begin{minipage}{.19\textwidth}
\begin{equation}
ax^{2} = x,
\end{equation}
\end{minipage}
\begin{minipage}{.19\textwidth}
\begin{equation}
a^{2}x = a,
\end{equation}
\end{minipage}
\begin{minipage}{.19\textwidth}
\begin{equation}
x^{2}a = x,
\end{equation}
\end{minipage}
\end{flushleft}
\[(1^{k})\,\,\,\,\,\,\,\,\,\,\,\,\,\,\,\,\,\,\,\,xa^{k+1} = a^{k} \text{ for some } k \in \{1, 2, \ldots\}\,\,\,\,\,\,\,\,\,\,\,\,\,\,\,\,\,\,\,\,\,\,\,\,\,\,\,\,\,\,\,\,\,\,\,\,\,\,\,(\,^{k}1)\,\,\,\,\,\,\,\,\,\,\,\,\,\,\,\,\,\,\,\,a^{k+1}x = a^{k} \text{ for some } k \in \{1, 2,
\ldots\}\] where (3) and (4) require $\mathcal{R}$ to be a
$\ast$-ring.
\begin{definition}
For any $a \in \mathcal{R}$, let $a\{i, j, \ldots , l\}$ denote
the set of elements $x \in \mathcal{R}$ which satisfy equations (i),
(j), . . . , (l) from among equations (1)--(9),($1^{k}$) and
($\,^{k}1$). An element $x \in a\{i, j, \ldots , l\}$ is called an
$\{i, j, \ldots , l\}$-inverse of $a$, and denoted by $a^{(i,j,...
,l)}$.
\end{definition}

The relation of generalized inverses of matrices and operators to
oblique and orthogonal projections is one of the most important
properties for their study and applications. Before describing the approach and the results for generalized inverses in rings that we
present in this paper, we now recall some of these relations.

Let $\mathbb{C}^{m \times n}$ denote the set of matrices of order $m
\times n$. Let $A \in \mathbb{C}^{m \times n}$ and $X \in
\mathbb{C}^{n \times m}$. Then \[X\in A\{1\} \Leftrightarrow
AX=P_{{\rm R}(A),S} \text{ and } XA=P_{T,{\rm N}(A)},\] where $S$ is
some subspace of $\mathbb{C}^{m}$ complementary to the range ${\rm
R}(A)$ of $A$, $P_{{\rm R}(A),S}$ is the oblique projection onto
${\rm R}(A)$ along $S$, $T$ is some subspace of $\mathbb{C}^{n}$
complementary to the null space ${\rm N}(A)$ of $A$, and $P_{T,{\rm
N}(A)}$ is the oblique projection onto $T$ along ${\rm N}(A)$. We
also have
\[X\in A\{1,3\} \Leftrightarrow AX=P_{{\rm R}(A)}\] and \[X\in A\{1,4\} \Leftrightarrow
XA=P_{{\rm R}(A^{\ast})},\] where $A^{\ast}$ is the conjugate
transpose of $A$, and $P_{{\rm R}(A)}$ and $P_{{\rm R}(A^{\ast})}$
denote the orthogonal projections onto the ${\rm R}(A)$ and ${\rm
R}(A^{\ast})$, respectively. See, e.g., \cite{Ben-Israel-Greville
(1976)} for more details.

Let $\mathcal{H}_{1}$ and $\mathcal{H}_{2}$ be Hilbert spaces over
$\mathbb{F}=\mathbb{R}$ or $\mathbb{F}=\mathbb{C}$. Let
$\mathcal{BC}(\mathcal{H}_{1}, \mathcal{H}_{2})$ denote the set of
bounded linear operators from $\mathcal{H}_{1}$ to $\mathcal{H}_{2}$
with closed range. If $A \in \mathcal{BC}(\mathcal{H}_{1},
\mathcal{H}_{2})$, then the Moore-Penrose inverse
$A^{\dag}=A^{(1,2,3,4)}$ exists and
\[AA^{\dag} = P_{{\rm R}(A)} \text{ and } A^{\dag}A = P_{{\rm
R}(A^{\ast})},\]  where $A^{\ast}$ is the adjoint of $A$ (see, e.g.,
\cite[Theorem 1]{Petryshyn (1967)}).

Let now $\mathcal{V}$ be a complex Banach space. If $A \in
\mathcal{B}(\mathcal{V})$ has finite index $k$, then the Drazin
inverse $A^{D}=A^{(1^{k},2,5)}$ of $A$ exists and satisfies
\[AA^{D}=A^{D}A=P_{{\rm R}(A^{k}),{\rm N}(A^{k})},\] where $P_{{\rm
R}(A^{k}),{\rm N}(A^{k})}$ is the oblique projection onto ${\rm
R}(A^{k})$ along ${\rm N}(A^{k})$ (see, e.g., \cite[Theorem 4 and
its proof]{King (1977)}).

The core $A^{\core}=A^{(1,2,3,6,7)}$ and the dual core
$A_{\core}=A^{(1,2,4,8,9)}$ inverses of $A \in \mathbb{C}^{n \times
n}$ are defined by the conditions
\begin{equation}\label{E D CI matrix}
AA^{\core}=P_{{\rm R}(A)} \text{ and } {\rm R}(A^{\core}) \subseteq
{\rm R}(A)
\end{equation}
and
\begin{equation}\label{E D DCI matrix}
A_{\core}A=P_{{\rm R}(A^{\ast})} \text{ and } {\rm R}(A_{\core})
\subseteq {\rm R}(A^{\ast}),
\end{equation}
respectively (see \cite[Definition 1, (i) and (ii) on page
693]{Baksalary-Trenkler (2010)}).

Since the theory of generalized inverses of matrices and operators
is closely connected with projections, i.e., idempotent (bounded)
linear transformations, it is natural to think about a similar
situation when working in any ring. In this paper, we show that we can similarly relate
generalized inverses in a ring $\mathcal{R}$ to idempotent group
endomorphisms $\rho: \mathcal{R} \rightarrow \mathcal{R}$, called
projectors, which are linked to direct sum decompositions of
$\mathcal{R}$, and use these relations in their study.

The organization of the paper is as follows:

In Section~\ref{S Preliminaries}, we present some properties of
generalized inverses, elements, direct sums and projectors in rings
that we use throughout the article.

In Section~\ref{S Basic relations}, we relate $\{1\}$, $\{2\}$,
$\{1,2\}$, $\{1,5\}$, and Drazin inverses to projectors.

Numerous particular generalized inverses are defined or studied by
means of their associated principal or annihilator ideals. In
Sections~\ref{S 1I with prescribed principal and annihilator
ideals}, \ref{S 2I with prescribed principal and annihilator
ideals}, and \ref{S 12I with prescribed principal and annihilator
ideals}, we address this topic with generality using
projectors. Let $a \in \mathcal{R}$, $\mathcal{S}$, $\mathcal{T}$ be
right ideals of $\mathcal{R}$, and $\mathcal{S}'$, $\mathcal{T}'$ be
left ideals of $\mathcal{R}$. We study $\{1\}$, $\{2\}$, and
$\{1,2\}$-inverses $x$ of $a$ such that one of the following
conditions holds: $xa\mathcal{R}=\mathcal{S}$ and ${\rm
rann}(xa)=\mathcal{T}$; $\mathcal{R}xa=\mathcal{S}'$ and ${\rm
lann}(xa)=\mathcal{T}'$; $xa\mathcal{R}=\mathcal{S}$ and
$\mathcal{R}xa=\mathcal{S}'$; or ${\rm rann}(xa)=\mathcal{T}$ and
${\rm lann}(xa)=\mathcal{T}'$. We also consider $\{1\}$ and
$\{1,2\}$-inverses such that one of the following conditions holds:
$xa\mathcal{R}=\mathcal{S}$; ${\rm rann}(xa)=\mathcal{T}$;
$\mathcal{R}xa=\mathcal{S}'$; or ${\rm lann}(xa)=\mathcal{T}'$. We
give several characterizations and existence conditions for these
generalized inverses. We establish the connection of the Mitsch
partial order \cite{Mitsch (1986)} with the $\{2\}$-inverses
considered in Section~\ref{S 2I with prescribed principal and
annihilator ideals}. Since in these sections, we do not make any
additional a priori assumption for the ideals $\mathcal{S}$,
$\mathcal{T}$, $\mathcal{S}'$, and $\mathcal{T}'$, we can consider a
large variety of particular cases (as we do in  Section~\ref{S
Particular classes of 1 2 and 12 inverses}). Moreover, as a
consequence, we obtain sufficient conditions for any right/left
ideal of $\mathcal{R}$ to be a principal or an annihilator ideal of
an idempotent element of $\mathcal{R}$.

In Section~\ref{S Particular classes of 1 2 and 12 inverses}, we
apply results of the previous sections to study some classes of
$\{1\}$, $\{2\}$, and $\{1,2\}$-inverses. First, we consider
$\{1,3\}$, $\{1,4\}$, $\{1,3,4\}$, $\{1,3,6\}$, $\{1,4,8\}$,
$\{1,3,7\}$, and $\{1,4,9\}$-inverses. Then, we analyze the $(e,f)$
Moore-Penrose inverse \cite{Mosic-Djordjevic (2011)}, the $e$-core
and $f$-dual core inverses \cite{Mosic-Deng-Ma (2018)}, the $w$-core
and the dual $v$-core inverses \cite{Zhu-Wu-Chen (2023)}, and the
right $w$-core and left dual $v$-core inverses \cite{Zhu-Wu-Mosic
(2023)}. Finally, we obtain some properties of two types of
$\{2\}$-inverses: the $(b,c)$ inverses \cite{Drazin (2012)} and the
$(p,q)$ inverse \cite{Djordjevic-Wei (2005)}. We end Section~\ref{S
Particular classes of 1 2 and 12 inverses} by giving an illustrative
example with a matrix over a field.

\section{Preliminaries}\label{S Preliminaries}

This section presents notations, definitions and results about
generalized inverses, elements, direct sums and projectors in rings,
that will be used later.

\subsection{Generalized inverses}

Let $a \in \mathcal{R}$. If $a\{1\}\neq\emptyset$, then $a$ is
called \emph{regular} (in the sense of von Neumann) and an $x\in
a\{1\}$ is called an \emph{inner inverse} of $a$. If
$a\{2\}\neq\emptyset$, then $a$ is called \emph{anti-regular} and an
$x\in a\{2\}$ is called an \emph{outer inverse} of $a$.

Note that if $x\in a\{5\}$, then ($1^{k}$) is equivalent to
($\,^{k}1$). If there exists $k \in \{1, 2, \ldots\}$ such that
$a\{2, 5, 1^{k}\}\neq\emptyset$, then $a$ is called \emph{Drazin
invertible}. The smallest of these positive integers $k$ is called
the index of $a$. The set $a\{2,5,1^{k}\}$ has a unique element
called the \emph{Drazin inverse} of $a$ and denoted by $a^{D}$. In
particular, if $a\{1, 2, 5\}\neq\emptyset$, then $a$ is called
\emph{group invertible} and the \emph{group inverse} of $a$ is
denoted by $a^{\#}$. For more details about these inverses in rings
see, e.g., \cite{Drazin (1958), Lebtahi-Patrício-Thome (2013), Patrício-Puystjens (2004)}.

Let $\mathcal{R}$ be a $\ast$-ring and $a \in \mathcal{R}$. If
$a\{1,2,3,4\}$ is not empty, then $a$ is called \emph{Moore-Penrose
invertible}. In this case, $a\{1,2,3,4\}$ has a unique element
called the \emph{Moore-Penrose inverse} of $a$ and denoted by
$a^{\dag}$. See, e.g., \cite{Hartwig (1976),
Koliha-Djordjevic-Cvetkovic (2007), Koliha-RakoCevic (2007), Xu-Chen
(2019)} for properties of the Moore-Penrose inverse in $\ast$-rings.

Baksalary and Trenkler \cite{Baksalary-Trenkler (2010)} introduced
two generalized inverses for complex matrices (see (\ref{E D CI matrix}) and (\ref{E D DCI matrix})). Later, Raki\'{c},
Din\v{c}i\'{c}, and Djordjevi\'{c} \cite{Rakic-Dincic-Djordjevic
(2014b)} generalized these notions to an arbitrary $\ast$-ring. Let $\mathcal{R}$ be a $\ast$-ring and $a, x \in \mathcal{R}$. Then $x$
is a \emph{core} (resp. \emph{dual core}) \emph{inverse} of $a$ if
$x \in a\{1\}$ and $x\mathcal{R}=x^{\ast}\mathcal{R}=a\mathcal{R}$
(resp. $x\mathcal{R}=x^{\ast}\mathcal{R}=a^{\ast}\mathcal{R}$). (see
\cite[Definitions 2.3 and 2.4]{Rakic-Dincic-Djordjevic (2014b)}). If
they exist, the core and the dual core inverses of an element $a \in
\mathcal{R}$ are unique and are denoted by $a^{\core}$ and
$a_{\core}$, respectively. By \cite[Lemmas 2.1 and
2.2]{Rakic-Dincic-Djordjevic (2014b)}, the conditions of
\cite[Definitions 2.3 and 2.4]{Rakic-Dincic-Djordjevic (2014b)} for
the core and dual core inverses are equivalent to the conditions
(\ref{E D CI matrix}) and (\ref{E D DCI matrix}) for finite complex
matrices. In \cite{Rakic-Dincic-Djordjevic (2014b)}, it is proved
that $a^{\core}=a^{(1,2,3,6,7)}$ and $a_{\core}=a^{(1,2,4,8,9)}$. In
\cite{Xu-Chen-Zhang (2017)}, it is proved that
$a^{\core}=a^{(3,6,7)}$, moreover, by the proof of \cite[Theorem
3.1]{Xu-Chen-Zhang (2017)}, $a\{6,7\} \subseteq a\{1,2\}$.
Similarly, $a_{\core}=a^{(4,8,9)}$ and $a\{8,9\} \subseteq
a\{1,2\}$. Properties of the core and dual core inverse in
$\ast$-rings can be found in, e.g., \cite{Li-Chen (2018), Rakic-Dincic-Djordjevic
(2014b), Xu-Chen-Zhang (2017)}.

Let us denote the sets of all Drazin invertible, group invertible,
Moore-Penrose invertible, core invertible and dual core invertible
elements in $\mathcal{R}$ by $\mathcal{R}^{D}$, $\mathcal{R}^{\#}$,
$\mathcal{R}^{\dag}$, $\mathcal{R}^{\core}$ and
$\mathcal{R}_{\core}$, respectively. We will use the following
well-known equalities:
\begin{enumerate}
  \item Let $a \in \mathcal{R}$ and $a^{(1)} \in a\{1\}$. Then $aa^{(1)}\mathcal{R}=a\mathcal{R}$, ${\rm rann}(a^{(1)}a)={\rm rann}(a)$,
  ${\rm lann}(aa^{(1)})={\rm lann}(a)$, and
$\mathcal{R}a^{(1)}a=\mathcal{R}a$.
  \item Let $a \in \mathcal{R}$ and $a^{(1)} \in a\{2\}$. Then ${\rm rann}(aa^{(2)})={\rm rann}(a^{(2)})$, $a^{(2)}a\mathcal{R}=a^{(2)}\mathcal{R}$,
$\mathcal{R}aa^{(2)}=\mathcal{R}a^{(2)}$, and ${\rm
lann}(a^{(2)}a)={\rm lann}(a^{(2)})$.
  \item Let $a \in \mathcal{R}^{D}$ with index $k$ and $l \geq k$. Then
$a^{D}a\mathcal{R}=aa^{D}\mathcal{R}=a^{D}\mathcal{R}=a^{l}\mathcal{R}$,
${\rm rann}(aa^{D})={\rm rann}(a^{D}a)={\rm rann}(a^{D})={\rm
rann}(a^{l})$,
$\mathcal{R}aa^{D}=\mathcal{R}a^{D}a=\mathcal{R}a^{D}=\mathcal{R}a^{l}$
and ${\rm lann}(aa^{D})={\rm lann}(a^{D}a)={\rm lann}(a^{D})={\rm
lann}(a^{l})$.
\end{enumerate}

\subsection{Elements in rings}\label{S Preliminaries Sub Elements in rings}

An element $p \in \mathcal{R}$ is an \emph{idempotent} if $p=p^{2}$.
If $\mathcal{R}$ is a $\ast$-ring and $p=p^{\ast}$, then $p$ is said
to be \emph{symmetric}. If $p=p^{2}=p^{\ast}$, then $p$ is called a
\emph{projection}. The set of invertible and idempotent elements in
$\mathcal{R}$ are denoted with $\mathcal{R}^{-1}$ and
$\mathcal{R}^{\bullet}$, respectively. If $\mathcal{R}$ is a
$\ast$-ring, the set of symmetric elements in $\mathcal{R}$ is
denoted with $\mathcal{R}^{\rm sym}$.

For $a \in \mathcal{R}$, we consider the group endomorphisms
$\varphi_{a}: \mathcal{R} \rightarrow \mathcal{R}$ given by
$\varphi_{a}(x)=ax$ and $\,_{a}\varphi: \mathcal{R} \rightarrow
\mathcal{R}$ given by $\,_{a}\varphi(x)=xa$. We have, ${\rm
im}(\varphi_{a})=a\mathcal{R}$, ${\rm ker}(\varphi_{a})={\rm
rann}(a)$, ${\rm im}(_{a}\varphi)=\mathcal{R}a$, and ${\rm
ker}(_{a}\varphi)={\rm lann}(a)$.

\begin{lemma}\label{L invertible}
Let $a \in \mathcal{R}$. Then the following assertions are
equivalent:
\begin{enumerate}
  \item $a \in \mathcal{R}^{-1}$.
  \item $a\mathcal{R}=\mathcal{R}$ and ${\rm rann}(a)=\{0\}$.
  \item $\mathcal{R}a=\mathcal{R}$ and ${\rm lann}(a)=\{0\}$.
\end{enumerate}
\end{lemma}
\begin{proof}
(1) $\Rightarrow$ (2)(3): It is immediate.

(2) $\Rightarrow$ (1): Assume that $a\mathcal{R}=\mathcal{R}$ and
${\rm rann}(a)=\{0\}$. Then $\varphi_{a}$ is a group automorphism.
Let $\psi$ be the group automorphism such that
$\psi=\varphi_{a}^{-1}$. Since $\varphi_{a}(1)=a$, we have
$\psi(a)=1$. For each $s \in \mathcal{R}$ there exists a unique $r
\in \mathcal{R}$ such that $ar=s$. In particular, there exists a
unique $b \in \mathcal{R}$ such that $ab=1$. Then, $\varphi_{a}(r)=s
\Leftrightarrow ar=s \Leftrightarrow ar=abs \Leftrightarrow r=bs$.
Hence, $\psi(s)=r=bs$. From here, $\psi=\varphi_{b}$ and
$ba=\varphi_{b}(a)=\psi(a)=1$. This shows that $b=a^{-1}$.
Therefore, $a \in \mathcal{R}^{-1}$.

(3) $\Rightarrow$ (1): It is similar to the proof of (2)
$\Rightarrow$ (1).
\end{proof}
\begin{lemma}\label{L igualdad idempotentes}
Let $p, q \in \mathcal{R}^{\bullet}$. Then:
\begin{enumerate}
  \item $p\mathcal{R}={\rm rann}(1-p)$, $\mathcal{R}p={\rm lann}(1-p)$.
  \item $p\mathcal{R} \subseteq q\mathcal{R} \Leftrightarrow {\rm lann}(q) \subseteq {\rm lann}(p)$, $\mathcal{R}p \subseteq \mathcal{R}q \Leftrightarrow {\rm rann}(q) \subseteq
{\rm rann}(p)$.
  \item $q=p \Leftrightarrow \{q\mathcal{R} \subseteq p\mathcal{R}$ and ${\rm rann}(q) \subseteq {\rm rann}(p)\}$.
\end{enumerate}
\end{lemma}
\begin{lemma}\label{L L 2.5 RDD (2014b)}\cite[Lemmas 2.5 and
2.6]{Rakic-Dincic-Djordjevic (2014b)} Let $a, b \in \mathcal{R}$.
Then:
\begin{enumerate}
  \item If $a\mathcal{R} \subseteq b\mathcal{R}$, then ${\rm lann}(b) \subseteq
{\rm lann}(a)$.
  \item If ${\rm lann}(b) \subseteq {\rm lann}(a)$ and $b\{1\} \neq \emptyset$, then
$a\mathcal{R} \subseteq b\mathcal{R}$.
  \item If $\mathcal{R}a \subseteq \mathcal{R}b$, then ${\rm rann}(b) \subseteq
{\rm rann}(a)$.
  \item If ${\rm rann}(b) \subseteq {\rm rann}(a)$ and
$b\{1\} \neq \emptyset$, then $\mathcal{R}a \subseteq \mathcal{R}b$.
\end{enumerate}
\end{lemma}
\begin{lemma}\label{L a1novac aR subset bR b1novac}
Let $a, b \in \mathcal{R}$ be such that $b\{1\} \neq \emptyset$.
Then:
\begin{enumerate}
  \item If ${\rm rann}(b) \subseteq {\rm rann}(a)$ and $\mathcal{R}b \subseteq \mathcal{R}a$, then $a\{1\} \neq \emptyset$.
  \item If ${\rm lann}(b) \subseteq {\rm lann}(a)$ and $b\mathcal{R} \subseteq
a\mathcal{R}$, then $a\{1\} \neq \emptyset$.
\end{enumerate}
\end{lemma}
\begin{proof}
(1): Let $x\in b\{1\}$. Since ${\rm rann}(b) \subseteq {\rm
rann}(a)$ and $1-xb \in {\rm rann}(b)$, we have $a(1-xb)=0$. From
here, $a=axb\in aRb\subseteq aRa$. Therefore, $a\{1\} \neq
\emptyset$.

(2): The proof is similar to the proof of (1).
\end{proof}
As a consequence of Lemmas~\ref{L L 2.5 RDD (2014b)}(2)(4) and
\ref{L a1novac aR subset bR b1novac} we get the following result.
\begin{lemma}\label{L a1neqemptyset}
Let $a, b \in \mathcal{R}$ be such that $b\{1\} \neq \emptyset$.
\begin{enumerate}
  \item If ${\rm rann}(a) = {\rm rann}(b)$, then $\mathcal{R}b \subseteq
  \mathcal{R}a$ if and only if $a\{1\} \neq \emptyset$.
  \item If ${\rm lann}(a) = {\rm lann}(b)$, then $b\mathcal{R} \subseteq
a\mathcal{R}$ if and only if $a\{1\} \neq \emptyset$.
\end{enumerate}
\end{lemma}
If $\mathcal{R}$ is a $\ast$-ring, then $a\mathcal{R} \subseteq
b\mathcal{R} \Leftrightarrow \mathcal{R}a^{\ast} \subseteq
\mathcal{R}b^{\ast}$ and ${\rm rann}(a) \subseteq {\rm rann}(b)
\Leftrightarrow {\rm lann}(a^{\ast}) \subseteq {\rm
lann}(b^{\ast})$.
\begin{definition}
Let $\mathcal{R}$ be a $\ast$-ring.
\begin{enumerate}
  \item Let $a, b \in \mathcal{R}$. Then $a$ and $b$ are \emph{right orthogonal} (resp. \emph{left orthogonal}), written $a \perp_{r} b$ or $b \perp_{r}
a$ (resp. $a \perp_{l} b$ or $b \perp_{l} a$), if $a^{\ast}b=0$
(resp. $ab^{\ast}=0$).
  \item Let $\mathcal{S}, \mathcal{T} \subseteq \mathcal{R}$. Then $\mathcal{S}$ and $\mathcal{T}$ are
\emph{right orthogonal} (resp. \emph{left orthogonal}), written
$\mathcal{S} \perp_{r} \mathcal{T}$ or $\mathcal{T} \perp_{r}
\mathcal{S}$ (resp. $\mathcal{S} \perp_{l} \mathcal{T}$ or
$\mathcal{T} \perp_{l} \mathcal{S}$), if $a \perp_{r} b$ (resp. $a
\perp_{l} b$) for each $a \in \mathcal{S}$ and $b \in \mathcal{T}$.
\end{enumerate}
\end{definition}
\begin{lemma}\label{L aR a0 ortogonales}
Let $\mathcal{R}$ be a $\ast$-ring and $a \in \mathcal{R}$. Then:
\begin{enumerate}
  \item $a\mathcal{R} \perp_{r} {\rm rann}(a^{\ast})$.
  \item If $a \in \mathcal{R}^{\rm sym}$ then $a\mathcal{R} \perp_{r} {\rm rann}(a)$.
  \item If $a \in \mathcal{R}^{\bullet}$ and $a\mathcal{R} \perp_{r} {\rm rann}(a)$, then $a \in \mathcal{R}^{\rm sym}$.
\end{enumerate}
\end{lemma}
\begin{proof}
(1): Let $x \in a\mathcal{R}$ and $y \in {\rm rann}(a^{\ast})$.
Hence, there exists $r \in \mathcal{R}$ such that $x=ar$ and
$a^{\ast}y=0$. Then, $x^{\ast}y=r^{\ast}a^{\ast}y=0$. This proves
that $a\mathcal{R} \perp_{r} {\rm rann}(a^{\ast})$.

(2): It follows from (1).

(3): Assume that $a \in \mathcal{R}^{\bullet}$ and $a\mathcal{R}
\perp_{r} {\rm rann}(a)$. Since $a \in \mathcal{R}^{\bullet}$, it
follows that  $1-a \in {\rm rann}(a)$. Then, since $a\mathcal{R}
\perp_{r} {\rm rann}(a)$, we have $a^{\ast}(1-a)=0$. This last
equality is equivalent to $a^{\ast}=a^{\ast}a$. Then, $a^{\ast}=a$.
\end{proof}
Analogously to Lemma~\ref{L aR a0 ortogonales}, we have:
\begin{lemma}\label{L aR a0 ortogonales b}
Let $\mathcal{R}$ be a $\ast$-ring and $a \in \mathcal{R}$. Then:
\begin{enumerate}
  \item $\mathcal{R}a \perp_{l} {\rm lann}(a^{\ast})$.
  \item If $a \in \mathcal{R}^{\rm sym}$ then $\mathcal{R}a \perp_{l} {\rm lann}(a)$.
  \item If $a \in \mathcal{R}^{\bullet}$ and $\mathcal{R}a \perp_{l} {\rm lann}(a)$, then $a \in \mathcal{R}^{\rm sym}$.
\end{enumerate}
\end{lemma}

\subsection{Projectors}\label{S Preliminaries Sub Projectors}

In vector spaces and modules, idempotent linear transformations are
well known as (oblique) projections or projectors (see, e.g.,
\cite{Anderson-Fuller (1973), Roman (2008)}). In rings, we consider
idempotent group endomorphisms, called projectors, to use them to
study generalized inverses. Let $\mathcal{S}$ and $\mathcal{T}$ be
subgroups of $\mathcal{R}$. Let

\centerline{$\mathcal{S} + \mathcal{T}=\{s+t : s \in \mathcal{S}
\text{ and } t \in \mathcal{T}\}$.}

\begin{definition}
Let $\mathcal{S}$ and $\mathcal{T}$ be subgroups of
$\mathcal{R}$. Then $\mathcal{R}$ is the (\emph{internal})
\emph{direct sum} of $\mathcal{S}$ and $\mathcal{T}$, written
$\mathcal{R} = \mathcal{S} \oplus \mathcal{T}$, if $\mathcal{R} =
\mathcal{S} + \mathcal{T}$ and $\mathcal{S}\cap\mathcal{T}=\{0\}$.
In this case, $\mathcal{S}$ and $\mathcal{T}$ are called
\emph{direct summands} of $\mathcal{R}$, and $\mathcal{T}$ (resp.
$\mathcal{S}$) is called a \emph{complement} of $\mathcal{S}$ (resp.
$\mathcal{T}$) in $\mathcal{R}$.
\end{definition}
Associated with a direct sum decomposition of $\mathcal{R}$ we have
a group endomorphism:
\begin{definition}\label{D projector}
Let $\mathcal{S}$ and $\mathcal{T}$ be subgroups of
$\mathcal{R}$ such that $\mathcal{R} = \mathcal{S} \oplus
\mathcal{T}$. The group endomorphism $\rho_{\mathcal{S},\mathcal{T}}
: \mathcal{R} \rightarrow \mathcal{R}$ defined by
$\rho_{\mathcal{S},\mathcal{T}}(s+t)=s$, where $s \in \mathcal{S}$
and $t \in \mathcal{T}$, is called the \emph{oblique projector onto}
$\mathcal{S}$ \emph{along} $\mathcal{T}$. If $\mathcal{R}$ is a
$\ast$-ring and $\mathcal{S} \perp_{r} \mathcal{T}$ (resp.
$\mathcal{S} \perp_{l} \mathcal{T}$), we say that
$\rho_{\mathcal{S},\mathcal{T}}$ is a \emph{right (resp. left)
orthogonal projector}.
\end{definition}
We usually say projector and orthogonal projector instead of oblique
projector and right (left) orthogonal projector, respectively. From
Definition~\ref{D projector} we obtain:
\begin{lemma}\label{L proy im ker}
Let $\mathcal{S}$ and $\mathcal{T}$ be subgroups of $\mathcal{R}$
such that $\mathcal{R} = \mathcal{S} \oplus \mathcal{T}$. Then:
\begin{enumerate}
  \item $\rho_{\mathcal{S},\mathcal{T}}+\rho_{\mathcal{T},
\mathcal{S}}={\rm id}_{\mathcal{R}}$.
  \item ${\rm im}(\rho_{\mathcal{S},\mathcal{T}})=\mathcal{S}$ and ${\rm ker}(\rho_{\mathcal{S},\mathcal{T}})=\mathcal{T}$.
  \item $r \in {\rm im}(\rho_{\mathcal{S},\mathcal{T}})
\Leftrightarrow \rho_{\mathcal{S},\mathcal{T}}(r)=r$.
\end{enumerate}
\end{lemma}
Another property of projectors derived easily from Definition~\ref{D
projector} is the following:
\begin{lemma}\label{L oblique projector sii direct sum}
If $\varphi : \mathcal{R} \rightarrow \mathcal{R}$ is a group
endomorphism such that $\mathcal{R} = {\rm im}(\varphi) \oplus {\rm
ker}(\varphi)$ and $\varphi_{|{\rm im}(\varphi)}={\rm id}_{{\rm
im}(\varphi)}$, then $\varphi=\rho_{{\rm im}(\varphi), {\rm
ker}(\varphi)}$.
\end{lemma}
The next lemma asserts that projectors are precisely idempotent
group endomorphisms.
\begin{lemma}\label{L oblique projector sii idempotent}
A group endomorphism $\varphi : \mathcal{R} \rightarrow \mathcal{R}$
is a projector if and only if $\varphi$ is idempotent, and in this
case, $\mathcal{R} = {\rm im}(\varphi) \oplus {\rm ker}(\varphi)$
and $\varphi=\rho_{{\rm im}(\varphi), {\rm ker}(\varphi)}$.
\end{lemma}
\begin{proof}
Assume that there exist $\mathcal{S}$ and $\mathcal{T}$ subgroups of
$\mathcal{R}$ such that $\mathcal{R} = \mathcal{S} \oplus
\mathcal{T}$ and $\varphi=\rho_{\mathcal{S},\mathcal{T}}$. Let $r
\in \mathcal{R}$.  There exist $s \in \mathcal{S}$ and $t \in
\mathcal{T}$ such that $r=s+t$. Then
$\varphi^{2}(r)=\rho_{\mathcal{S},\mathcal{T}}(\rho_{\mathcal{S},\mathcal{T}}(s+t))=\rho_{\mathcal{S},\mathcal{T}}(s)=s=\rho_{\mathcal{S},\mathcal{T}}(r)=\varphi(r)$.
Hence, $\varphi^{2}=\varphi$.

Conversely, suppose that $\varphi$ is idempotent. Let $r \in
\mathcal{R}$, $s=\varphi(r)$ and $t=r-s$. Then
$\varphi(t)=\varphi(r-s)=\varphi(r)-\varphi(s)=\varphi(r)-\varphi^{2}(r)=0$.
Thus, $r=s+t$ with $s \in {\rm im}(\varphi)$ and $t \in {\rm
ker}(\varphi)$. This shows that $\mathcal{R} = {\rm im}(\varphi) +
{\rm ker}(\varphi)$. If $r \in {\rm im}(\varphi)$, then there exists
$s \in \mathcal{R}$ such that
\begin{equation}\label{E L oblique projector sii idempotent}
r=\varphi(s)=\varphi^{2}(s)=\varphi(r).
\end{equation}
Hence, $\varphi_{|{\rm im}(\varphi)}={\rm id}_{{\rm im}(\varphi)}$.
By (\ref{E L oblique projector sii idempotent}), if $r \in {\rm
im}(\varphi) \cap {\rm ker}(\varphi)$, then $r=0$. Thus, ${\rm
im}(\varphi) \cap {\rm ker}(\varphi) = \{0\}$. Therefore,
$\mathcal{R} = {\rm im}(\varphi) \oplus {\rm ker}(\varphi)$.
Applying now Lemma~\ref{L oblique projector sii direct sum}, we get
$\varphi=\rho_{{\rm im}(\varphi), {\rm ker}(\varphi)}$.
\end{proof}
So far, we only required $\mathcal{S}$ and $\mathcal{T}$ to be
subgroups of $\mathcal{R}$. The next two lemmas give properties of
the projectors when $\mathcal{S}$ and $\mathcal{T}$ are right (left)
ideals of $\mathcal{R}$.
\begin{lemma}\label{L S T right left ideals sii proy}
Let $\mathcal{S}$ and $\mathcal{T}$ be subgroups of $\mathcal{R}$
such that $\mathcal{R} = \mathcal{S} \oplus \mathcal{T}$. Then:
\begin{enumerate}
  \item $\mathcal{S}$ and $\mathcal{T}$ are right ideals of
$\mathcal{R}$ if and only if
$\rho_{\mathcal{S},\mathcal{T}}(r_{1}r_{2})=\rho_{\mathcal{S},\mathcal{T}}(r_{1})r_{2}$
for each $r_{1}, r_{2} \in \mathcal{R}$.
  \item $\mathcal{S}$ and $\mathcal{T}$ are left ideals of
$\mathcal{R}$ if and only if
$\rho_{\mathcal{S},\mathcal{T}}(r_{1}r_{2})=r_{1}\rho_{\mathcal{S},\mathcal{T}}(r_{2})
$ for each $r_{1}, r_{2} \in \mathcal{R}$.
\end{enumerate}
\end{lemma}
\begin{proof}
We prove (1). The proof of (2) is similar.

Assume that $\mathcal{S}$ and $\mathcal{T}$ are right ideals of
$\mathcal{R}$. Let $r_{1}, r_{2} \in \mathcal{R}$. Since
$\rho_{\mathcal{S},\mathcal{T}}(r_{1})r_{2} \in \mathcal{S}$ and
$\rho_{\mathcal{T}, \mathcal{S}}(r_{1})r_{2} \in \mathcal{T}$, we
have
$\rho_{\mathcal{S},\mathcal{T}}(r_{1}r_{2})=\rho_{\mathcal{S},\mathcal{T}}(\rho_{\mathcal{S},\mathcal{T}}(r_{1})r_{2}+\rho_{\mathcal{T},
\mathcal{S}}(r_{1})r_{2})=\rho_{\mathcal{S},\mathcal{T}}(r_{1})r_{2}$.

Conversely, assume that
$\rho_{\mathcal{S},\mathcal{T}}(r_{1}r_{2})=\rho_{\mathcal{S},\mathcal{T}}(r_{1})r_{2}$
for each $r_{1}, r_{2} \in \mathcal{R}$. Take $r_{1} \in
\mathcal{S}$ and $r_{2} \in \mathcal{R}$. Then
$\rho_{\mathcal{S},\mathcal{T}}(r_{1}r_{2})=r_{1}r_{2}$ and by
Lemma~\ref{L proy im ker}(2), $r_{1}r_{2} \in \mathcal{S}$. This
shows that $\mathcal{S}$ is a right ideal of $\mathcal{R}$. Take now
$r_{1} \in \mathcal{T}$ and $r_{2} \in \mathcal{R}$. Then
$\rho_{\mathcal{S},\mathcal{T}}(r_{1}r_{2})=0$ and by Lemma~\ref{L
proy im ker}(2), $r_{1}r_{2} \in \mathcal{T}$. This shows that
$\mathcal{T}$ is a right ideal of $\mathcal{R}$.
\end{proof}

As a consequence of Lemmas~\ref{L proy im ker} and~\ref{L S T right
left ideals sii proy} we get:
\begin{lemma}\label{L ro a = a and a ro = a}
Let $\mathcal{S}$ and $\mathcal{T}$ be subgroups of $\mathcal{R}$
such that $\mathcal{R} = \mathcal{S} \oplus \mathcal{T}$ and $a \in
\mathcal{R}$. Then the following assertions hold:
\begin{enumerate}
  \item If $\mathcal{S}$ and $\mathcal{T}$ are right ideals of $\mathcal{R}$, then $\rho_{\mathcal{S},\mathcal{T}}(1)a=a \Leftrightarrow a\mathcal{R} \subseteq \mathcal{S}$ and $a\rho_{\mathcal{S},\mathcal{T}}(1)=a \Leftrightarrow
\mathcal{T} \subseteq {\rm rann}(a)$.
  \item If $\mathcal{S}$ and $\mathcal{T}$ are left ideals of $\mathcal{R}$, then $a\rho_{\mathcal{S},\mathcal{T}}(1)=a \Leftrightarrow \mathcal{R}a \subseteq \mathcal{S}$ and $\rho_{\mathcal{S},\mathcal{T}}(1)a=a \Leftrightarrow
\mathcal{T} \subseteq {\rm lann}(a)$.
\end{enumerate}
\end{lemma}
\begin{proof}
Assume that $\mathcal{S}$ and $\mathcal{T}$ are right ideals of
$\mathcal{R}$. Then $\rho_{\mathcal{S},\mathcal{T}}(1)a=a
\Leftrightarrow \rho_{\mathcal{S},\mathcal{T}}(a)=a \Leftrightarrow
a \in \mathcal{S} \Leftrightarrow a\mathcal{R} \subseteq
\mathcal{S}$ and $a\rho_{\mathcal{S},\mathcal{T}}(1)=a
\Leftrightarrow
a\rho_{\mathcal{S},\mathcal{T}}(1)=a(\rho_{\mathcal{S},\mathcal{T}}(1)+\rho_{\mathcal{T},
\mathcal{S}}(1)) \Leftrightarrow a\rho_{\mathcal{T},
\mathcal{S}}(1)=0 \Leftrightarrow \forall r \in \mathcal{R}:
a\rho_{\mathcal{T}, \mathcal{S}}(1)r=0 \Leftrightarrow \forall r \in
\mathcal{R}: a\rho_{\mathcal{T}, \mathcal{S}}(r)=0 \Leftrightarrow
\mathcal{T} \subseteq {\rm rann}(a)$. This proves (1). The proof of
(2) is similar.
\end{proof}
In what follows, whenever we write $\rho_{\mathcal{S},\mathcal{T}}$,
we are implicity asserting that $\mathcal{R} = \mathcal{S} \oplus
\mathcal{T}$. Part (1) of the following corollary is a consequence
of Lemma~\ref{L oblique projector sii idempotent} whereas part (2)
follows from Lemma~\ref{L S T right left ideals sii proy}.
\begin{corollary}\label{C a idempotent varphia proy}
The following assertions hold:
\begin{enumerate}
  \item If $a \in \mathcal{R}^{\bullet}$, then
$\varphi_{a}=\rho_{a\mathcal{R},{\rm rann}(a)}$ and
$\,_{a}\varphi=\rho_{\mathcal{R}a,{\rm lann}(a)}$.
  \item Let $\mathcal{S}$, $\mathcal{T}$ be right (resp. left) ideals of $\mathcal{R}$ and $a=\rho_{\mathcal{S},\mathcal{T}}(1)$, then $a \in \mathcal{R}^{\bullet}$ and $\varphi_{a}=\rho_{\mathcal{S},\mathcal{T}}$ (resp. $\,_{a}\varphi=\rho_{\mathcal{S},\mathcal{T}}$).
\end{enumerate}
\end{corollary}

\section{Relations of
$\{1\}$, $\{2\}$, $\{1,2\}$, $\{1,5\}$, and Drazin inverses to
projectors}\label{S Basic relations}

Let $a\{1\} \neq \emptyset$ and $a^{(1)} \in a\{1\}$. Since
$aa^{(1)}, a^{(1)}a \in \mathcal{R}^{\bullet}$, Corollary~\ref{C a
idempotent varphia proy}(1) yields the next theorem that relates
$\{1\}$-inverses to projectors.
\begin{theorem}\label{T 1I projectors}
Let $a \in \mathcal{R}$. Then the following assertions are
equivalent:
\begin{enumerate}
  \item $x \in a\{1\}$.
  \item $\varphi_{ax}=\rho_{a\mathcal{R},{\rm rann}(ax)}$.
  \item $\varphi_{xa}=\rho_{xa\mathcal{R},{\rm
rann}(a)}$.
  \item $\,_{ax}\varphi=\rho_{\mathcal{R}ax,{\rm lann}(a)}$.
  \item $\,_{xa}\varphi=\rho_{\mathcal{R}a,{\rm lann}(xa)}$.
\end{enumerate}
\end{theorem}
Since $a^{(2)} \in a\{2\}$ if and only if $a \in a^{(2)}\{1\}$,
Theorem~\ref{T 1I projectors} has an analogous version for
$\{2\}$-inverses:
\begin{theorem}\label{T 2I projectors}
Let $a \in \mathcal{R}$. Then the following assertions are
equivalent:
\begin{enumerate}
  \item $x \in a\{2\}$.
  \item $\varphi_{ax}=\rho_{ax\mathcal{R},{\rm rann}(x)}$.
  \item $\varphi_{xa}=\rho_{x\mathcal{R},{\rm rann}(xa)}$.
  \item $\,_{ax}\varphi=\rho_{\mathcal{R}x,{\rm
lann}(ax)}$.
  \item $\,_{xa}\varphi=\rho_{\mathcal{R}xa,{\rm lann}(x)}$.
\end{enumerate}
\end{theorem}
The following lemma will be used later.
\begin{lemma}\label{L ab ab1 a y b ab1 ab}
Let $a, b \in \mathcal{R}$ be such that $(ab)\{1\} \neq \emptyset$.
Let $(ab)^{(1)} \in (ab)\{1\}$. Then:
\begin{enumerate}
  \item $ab(ab)^{(1)}a=a \Leftrightarrow ab\mathcal{R}=a\mathcal{R} \Leftrightarrow {\rm lann}(ab)={\rm lann}(a)$.
  \item $b(ab)^{(1)}ab=b \Leftrightarrow {\rm rann}(ab)={\rm rann}(b) \Leftrightarrow \mathcal{R}ab=\mathcal{R}b$.
\end{enumerate}
\end{lemma}
\begin{proof}
We always have $ab\mathcal{R} \subseteq a\mathcal{R}$, ${\rm
lann}(a) \subseteq {\rm lann}(ab)$, ${\rm rann}(b) \subseteq {\rm
rann}(ab)$ and $\mathcal{R}ab \subseteq \mathcal{R}b$. Using
Theorem~\ref{T 1I projectors} and Lemma~\ref{L ro a = a and a ro =
a} we get:
\[ab(ab)^{(1)}a=a \Leftrightarrow
\rho_{ab\mathcal{R},{\rm rann}(ab(ab)^{(1)})}(1)a=a \Leftrightarrow
a\mathcal{R} \subseteq ab\mathcal{R},\]
\[ab(ab)^{(1)}a=a \Leftrightarrow
\rho_{\mathcal{R}ab(ab)^{(1)},{\rm lann}(ab)}(1)a=a \Leftrightarrow
{\rm lann}(ab) \subseteq {\rm lann}(a),\]
\[b(ab)^{(1)}ab = b \Leftrightarrow
b\rho_{(ab)^{(1)}ab\mathcal{R},{\rm rann}(ab)}(1)=b \Leftrightarrow
{\rm rann}(ab) \subseteq {\rm rann}(b)\] and
\[b(ab)^{(1)}ab = b \Leftrightarrow b\rho_{\mathcal{R}ab,{\rm lann}((ab)^{(1)}ab)}(1) \Leftrightarrow \mathcal{R}b \subseteq \mathcal{R}ab.\]
This proves (1) and (2).
\end{proof}
The next theorem is a consequence of Theorems~\ref{T 1I projectors}
and \ref{T 2I projectors}. It relates $\{1,2\}$-inverses to
projectors.
\begin{theorem}\label{T 12I projectors}
Let $a, x \in \mathcal{R}$. Then the following assertions are
equivalent:
\begin{enumerate}
  \item $x \in a\{1,2\}$.
  \item $\varphi_{ax}=\rho_{a\mathcal{R},{\rm rann}(x)}$.
  \item $\varphi_{xa} = \rho_{x\mathcal{R},
  {\rm rann}(a)}$.
  \item $\,_{ax}\varphi=\rho_{\mathcal{R}x,{\rm
lann}(a)}$.
  \item $\,_{xa}\varphi=\rho_{\mathcal{R}a,{\rm lann}(x)}$.
\end{enumerate}
\end{theorem}
We note that Theorem~\ref{T 1I projectors}(1)(2)(3) is related to
\cite[Lemma 1.1(f) and (2.28)]{Ben-Israel-Greville (2003)}, whereas
Lemma~\ref{L ab ab1 a y b ab1 ab} and Theorem~\ref{T 12I projectors}
generalize \cite[Lemma 1.2]{Ben-Israel-Greville (2003)}, \cite[Ex.
2.21]{Ben-Israel-Greville (2003)} and \cite[Corollary
2.7]{Ben-Israel-Greville (2003)} (see also \cite[Corollary
1.3.2]{Wang-Wei-Qiao (2018)}), respectively.
\begin{remark}
Let $a \in \mathcal{R}$ and $x \in a\{1,2\}$. Using an argument
similar to the one used in the proof of the \emph{if} part of
\cite[Theorem 1.2]{Ben-Israel-Greville (2003)} we get:
\begin{enumerate}
  \item If $a \in \mathcal{R}$, $x \in a\{1\}$ and $x\mathcal{R}= xa\mathcal{R}$
(or $\mathcal{R}x=\mathcal{R}ax$), then $x \in a\{1,2\}$.
  \item If $a \in \mathcal{R}$, $x \in a\{2\}$ and $a\mathcal{R}= ax\mathcal{R}$
  (or $\mathcal{R}a=\mathcal{R}xa$), then $x \in a\{1,2\}$.
\end{enumerate}
\end{remark}
The characterizations of Moore-Penrose, core, and dual core
inverses using $\varphi_{ax}$ (resp. $\,_{ax}\varphi$) and
$\varphi_{xa}$ (resp. $\,_{xa}\varphi$) appear in Section~\ref{S
Moore-Penrose core and dual core inverses} as a consequence of some
general results. Now, we present an example to show that only
these group endomorphisms are not sufficient to characterize $x$ as
the Drazin, the Moore-Penrose, the core, or the dual core inverse.
\begin{example}
Consider the real matrices $X=\left(
         \begin{array}{cc}
           x_{1,1} & x_{1,2} \\
           x_{2,1} & x_{2,2} \\
         \end{array}
       \right)$ and $A=\left(
         \begin{array}{cc}
           2 & -2 \\
           0 & 0 \\
         \end{array}
       \right)
$. Then $A^{\#}=\left(
         \begin{array}{cc}
           1/2 & -1/2 \\
           0 & 0 \\
         \end{array}
       \right)$, $A^{\dag}=\left(
         \begin{array}{cc}
           1/4 & 0 \\
           -1/4 & 0 \\
         \end{array}
       \right)$, $A^{\core}=\left(
         \begin{array}{cc}
           1/2 & 0 \\
           0 & 0 \\
         \end{array}
       \right)$, and $A_{\core}=\left(
         \begin{array}{cc}
           1/4 & -1/4 \\
           -1/4 & 1/4 \\
         \end{array}
       \right)$. We have:
\begin{enumerate}
  \item $AX=XA=AA^{\#}=A^{\#}A$ if and only
       if $x_{1,1}=1/2$, $x_{2,1}=0$ and $x_{2,2}=x_{1,2}+1/2$ if and only if $X\in A\{1,5\}$.

  \item $AX=AA^{\dag}$ and $XA=A^{\dag}A$ if and only
       if $x_{1,1}=-x_{2,1}=1/4$, and $x_{1,2}=x_{2,2}$ if and only if $X\in A\{1,3,4\}$.

  \item $AX=AA^{\core}$ and $XA=A^{\core}A$ if and only
       if $x_{1,1}=1/2$, $x_{2,1}=0$ and $x_{1,2}=x_{2,2}$ if and only $X\in A\{3,6\}$.

  \item $AX=AA_{\core}$ and $XA=A_{\core}A$ if and only
       if $x_{1,1}=-x_{2,1}=1/4$ and $x_{2,2}=x_{1,2}+1/2$ if and only $X\in A\{4,8\}$.
\end{enumerate}
\end{example}
We have the following immediate result.
\begin{theorem}\label{T 15 projector}
Let $a \in \mathcal{R}$. Then the following assertions are
equivalent:
\begin{enumerate}
  \item $x \in a\{1,5\}$.
  \item $\varphi_{ax} =
\varphi_{xa} = \rho_{a\mathcal{R},{\rm rann}(a)}$.
  \item $\,_{ax}\varphi = \,_{xa}\varphi = \rho_{\mathcal{R}a,{\rm
lann}(a)}$.
\end{enumerate}
\end{theorem}
The next theorem relates the Drazin inverse to projectors.
\begin{theorem}\label{T DI oblique projector}
Let $a, x \in \mathcal{R}$. Then the following assertions are
equivalent:
\begin{enumerate}
  \item $a \in \mathcal{R}^{D}$ with index $k \leq l$ and $x=a^{D}$.
  \item $\varphi_{xa}=\varphi_{ax}=\rho_{a^{l}\mathcal{R},{\rm rann}(a^{l})}$ and $x\mathcal{R} \subseteq a^{l}\mathcal{R}$.
  \item $\varphi_{xa}=\varphi_{ax}=\rho_{a^{l}\mathcal{R},{\rm rann}(a^{l})}$ and ${\rm rann}(a^{l}) \subseteq {\rm rann}(x)$.
  \item $\,_{xa}\varphi=\,_{xa}\varphi=\rho_{\mathcal{R}a^{l},{\rm lann}(a^{l})}$ and $\mathcal{R}x \subseteq \mathcal{R}a^{l}$.
  \item $\,_{xa}\varphi=\,_{xa}\varphi=\rho_{\mathcal{R}a^{l},{\rm lann}(a^{l})}$ and ${\rm lann}(a^{l}) \subseteq {\rm lann}(x)$.
\end{enumerate}
\end{theorem}
\begin{proof}
The implications (1) $\Rightarrow$ (2)--(5) follow from the
definition of the Drazin inverse and Corollary~\ref{C a idempotent
varphia proy}.

If $\varphi_{xa}=\varphi_{ax}=\rho_{a^{l}\mathcal{R},{\rm
rann}(a^{l})}$, then $ax=xa$ and
$xa^{l+1}=\rho_{a^{l}\mathcal{R},{\rm rann}(a^{l})}(a^{l})=a^{l}$.
Thus, $x \in a\{5,1^{l}\}$. Similarly, if
$\,_{xa}\varphi=\,_{xa}\varphi=\rho_{\mathcal{R}a^{l},{\rm
lann}(a^{l})}$, then $x \in a\{5,1^{l}\}$.

(2) $\Rightarrow$ (1): Assume that (2) holds. Then $x \in
a\{5,1^{l}\}$ and $xax=\rho_{a^{l}\mathcal{R},{\rm
rann}(a^{l})}(x)=x$. So, $a \in \mathcal{R}^{D}$ with index
$k \leq l$ and $x=a^{D}$.

The proofs of the other implications are similar.
\end{proof}

\section{$\{1\}$-inverses with prescribed principal and annihilator ideals}\label{S 1I with prescribed principal and annihilator ideals}

As a consequence of Theorem~\ref{T 1I xaR ax0 expression} below, we
can assert that if we choose arbitrary right ideals $\mathcal{S}$
and $\mathcal{T}$ of $\mathcal{R}$ complementary to ${\rm rann}(a)$
and $a\mathcal{R}$, respectively, then there exists $x \in a\{1\}$
such that $xa\mathcal{R}=\mathcal{S}$ and ${\rm
rann}(ax)=\mathcal{T}$. Similar considerations can be made for
Theorems~\ref{T 1I Rax 0xa expression}-\ref{T 1I 0xa expression}.
Before we enunciate these theorems, we observe that if
$\mathcal{R}=a\mathcal{R}\oplus\mathcal{T}$, then there exists $z
\in \mathcal{R}$ such that $az=\rho_{a\mathcal{R},\mathcal{T}}(1)$
and $aza=\rho_{a\mathcal{R},\mathcal{T}}(a)=a$. Hence,
$a\{1\}\neq\emptyset$. Analogously, if $\mathcal{T}'$ is a left
ideal of $\mathcal{R}$ and
$\mathcal{R}=\mathcal{R}a\oplus\mathcal{T}'$, then
$a\{1\}\neq\emptyset$.
\begin{theorem}\label{T 1I xaR ax0 expression}
Let $a, x \in \mathcal{R}$ and $\mathcal{S}$, $\mathcal{T}$ be right
ideals of $\mathcal{R}$. Then the following assertions are
equivalent:
\begin{enumerate}
  \item $x \in a\{1\}$, $xa\mathcal{R}=\mathcal{S}$, and ${\rm rann}(ax)=\mathcal{T}$.
  \item $\varphi_{ax}=\rho_{a\mathcal{R},\mathcal{T}}$ and $\varphi_{xa}=\rho_{\mathcal{S},{\rm rann}(a)}$.
  \item $x=\rho_{\mathcal{S},{\rm rann}(a)}(1)a^{(1)}\rho_{a\mathcal{R},\mathcal{T}}(1)+(1-a^{(1)}a)y(1-aa^{(1)})$ where $a^{(1)} \in a\{1\}$ and $y \in \mathcal{R}$.
\end{enumerate}
\end{theorem}
\begin{proof}
(1) $\Rightarrow$ (2): It follows from Theorem~\ref{T 1I
projectors}.

(2) $\Rightarrow$ (3): Assume that (2) holds. Then, $x \in a\{1\}$
and \[\rho_{\mathcal{S},{\rm
rann}(a)}(x\rho_{a\mathcal{R},\mathcal{T}}(1))+(1-xa)x(1-ax) =
xa(x(ax))+x-xax-xax+xaxax = xax+x-xax-xax+xax=x.\] Thus, (3) holds
with $a^{(1)}=y=x$.

(3) $\Rightarrow$ (1): Assume that (3) holds. Using Lemma~\ref{L ro
a = a and a ro = a}(1) we obtain,
$ax=aa^{(1)}\rho_{a\mathcal{R},\mathcal{T}}(1)$,
$xa=\rho_{\mathcal{S},{\rm rann}(a)}(a^{(1)}a)$ and  $axa=a$. Hence,
$x \in a\{1\}$, and by Theorem~\ref{T 1I projectors},
$\varphi_{ax}=\rho_{a\mathcal{R},{\rm rann}(ax)}$ and
  $\varphi_{xa}=\rho_{xa\mathcal{R},{\rm rann}(a)}$.

Let $r \in \mathcal{R}$. Then
\begin{align*}
axr =0 \Leftrightarrow aa^{(1)}\rho_{a\mathcal{R},\mathcal{T}}(r)=0
&\Leftrightarrow \{\exists s \in \mathcal{R},\,
\rho_{a\mathcal{R},\mathcal{T}}(r)=as \text{ and } aa^{(1)}as=0\}\\
&\Leftrightarrow \{\exists s \in \mathcal{R},\,
\rho_{a\mathcal{R},\mathcal{T}}(r)=as  \text{ and } as=0\}
\Leftrightarrow r \in \mathcal{T}.
\end{align*}
Thus, ${\rm rann}(ax)=\mathcal{T}$. By Theorem~\ref{T 1I
projectors}, $\varphi_{1-a^{(1)}a}=\rho_{{\rm
rann}(a),a^{(1)}a\mathcal{R}}$. Then, $\rho_{\mathcal{S},{\rm
rann}(a)}((1-a^{(1)}a)\mathcal{R}) = \{0\}$ and
\[\rho_{\mathcal{S},{\rm rann}(a)}(r)=\rho_{\mathcal{S},{\rm rann}(a)}(a^{(1)}ar)+\rho_{\mathcal{S},{\rm rann}(a)}((1-a^{(1)}a)r)=\rho_{\mathcal{S},{\rm rann}(a)}(a^{(1)}ar)=xar.\]
From here, $xa\mathcal{R}=\mathcal{S}$. Consequently, (1) holds.
\end{proof}
Theorem~\ref{T 1I xaR ax0 expression} generalizes \cite[Theorem
2.12(a)(b)]{Ben-Israel-Greville (2003)}. Using Theorem~\ref{T 1I
projectors} and Lemma~\ref{L ro a = a and a ro = a}(2), we
analogously prove the next result about $\{1\}$-inverses with given
left principal and annihilator ideals.
\begin{theorem}\label{T 1I Rax 0xa expression}
Let $a, x \in \mathcal{R}$ and $\mathcal{S}$, $\mathcal{T}$ be left
ideals of $\mathcal{R}$. Then the following assertions are
equivalent:
\begin{enumerate}
  \item $x \in a\{1\}$, $\mathcal{R}ax=\mathcal{S}$, and ${\rm lann}(xa)=\mathcal{T}$.
  \item $\,_{ax}\varphi=\rho_{\mathcal{S},{\rm lann}(a)}$ and $\,_{xa}\varphi=\rho_{\mathcal{R}a,\mathcal{T}}$.
  \item $x=\rho_{\mathcal{R}a,\mathcal{T}}(1)a^{(1)}\rho_{\mathcal{S},{\rm lann}(a)}(1)+(1-a^{(1)}a)y(1-aa^{(1)})$ where $a^{(1)} \in a\{1\}$ and $y \in \mathcal{R}$.
\end{enumerate}
\end{theorem}
With similar proofs, we obtain the next six theorems. The first is
about $\{1\}$-inverses with given right and left principal ideals.
\begin{theorem}\label{T 1I xaR Rax expression}
Let $a, x \in \mathcal{R}$, $\mathcal{S}$ be a right ideal of
$\mathcal{R}$, and $\mathcal{S}'$ be a left ideal of $\mathcal{R}$.
Then the following assertions are equivalent:
\begin{enumerate}
  \item $x \in a\{1\}$, $xa\mathcal{R}=\mathcal{S}$, and $\mathcal{R}ax=\mathcal{S}'$.
  \item $\varphi_{xa}=\rho_{\mathcal{S},{\rm rann}(a)}$ and $\,_{ax}\varphi=\rho_{\mathcal{S}',{\rm lann}(a)}$.
  \item $a\{1\}\neq\emptyset$ and $x=\rho_{\mathcal{S},{\rm rann}(a)}(1)a^{(1)}\rho_{\mathcal{S}',{\rm lann}(a)}(1)+(1-a^{(1)}a)y(1-aa^{(1)})$ where $a^{(1)} \in a\{1\}$ and $y \in \mathcal{R}$.
\end{enumerate}
\end{theorem}
Now we give a theorem for $\{1\}$-inverses with given right and left
annihilator ideals.
\begin{theorem}\label{T 1I 0xa ax0 expression}
Let $a, x \in \mathcal{R}$, $\mathcal{T}$ be a right ideal of
$\mathcal{R}$, and $\mathcal{T}'$ be a left ideal of $\mathcal{R}$.
Then the following assertions are equivalent:
\begin{enumerate}
  \item $x \in a\{1\}$, ${\rm rann}(ax)=\mathcal{T}$, and ${\rm lann}(xa)=\mathcal{T}'$.
  \item $\varphi_{ax}=\rho_{a\mathcal{R},\mathcal{T}}$ and $\,_{xa}\varphi=\rho_{\mathcal{R}a,\mathcal{T}'}$.
  \item $x=\rho_{\mathcal{R}a,\mathcal{T}'}(1)a^{(1)}\rho_{a\mathcal{R},\mathcal{T}}(1)+(1-a^{(1)}a)y(1-aa^{(1)})$ where $a^{(1)} \in a\{1\}$ and $y \in \mathcal{R}$.
\end{enumerate}
\end{theorem}
The following two theorems are about $\{1\}$-inverses for which only
a right principal or annihilator ideal is prefixed.
\begin{theorem}\label{T 1I xaR expression}
Let $a, x \in \mathcal{R}$ and $\mathcal{S}$ be a right ideal of
$\mathcal{R}$. Then the following assertions are equivalent:
\begin{enumerate}
  \item $x \in a\{1\}$, $xa\mathcal{R}=\mathcal{S}$.
  \item $\varphi_{xa}=\rho_{\mathcal{S},{\rm rann}(a)}$.
  \item $a\{1\}\neq\emptyset$ and $x=\rho_{\mathcal{S},{\rm rann}(a)}(1)a^{(1)}+(1-a^{(1)}a)y(1-aa^{(1)})$ where $a^{(1)} \in a\{1\}$ and $y \in \mathcal{R}$.
\end{enumerate}
\end{theorem}
\begin{theorem}\label{T 1I ax0 expression}
Let $a, x \in \mathcal{R}$ and $\mathcal{T}$ be a right ideal of
$\mathcal{R}$. Then the following assertions are equivalent:
\begin{enumerate}
  \item $x \in a\{1\}$ and ${\rm rann}(ax)=\mathcal{T}$.
  \item $\varphi_{ax}=\rho_{a\mathcal{R},\mathcal{T}}$.
  \item $x=a^{(1)}\rho_{a\mathcal{R},\mathcal{T}}(1)+(1-a^{(1)}a)y(1-aa^{(1)})$ where $a^{(1)} \in a\{1\}$ and $y \in \mathcal{R}$.
\end{enumerate}
\end{theorem}
In the next two theorems, we consider the case of $\{1\}$-inverses
for which only a left principal or annihilator ideal is prefixed.
\begin{theorem}\label{T 1I Rax expression}
Let $a, x \in \mathcal{R}$ and $\mathcal{S}$ be a left ideal of
$\mathcal{R}$. Then the following assertions are equivalent:
\begin{enumerate}
  \item $x \in a\{1\}$ and $\mathcal{R}ax=\mathcal{S}$.
  \item $\,_{ax}\varphi=\rho_{\mathcal{S},{\rm lann}(a)}$.
  \item $a\{1\}\neq\emptyset$ and $x=a^{(1)}\rho_{\mathcal{S},{\rm lann}(a)}(1)+(1-a^{(1)}a)y(1-aa^{(1)})$ where $a^{(1)} \in a\{1\}$ and $y \in \mathcal{R}$.
\end{enumerate}
\end{theorem}
\begin{theorem}\label{T 1I 0xa expression}
Let $a, x \in \mathcal{R}$ and $\mathcal{T}$ be left a ideal of
$\mathcal{R}$. Then the following assertions are equivalent:
\begin{enumerate}
  \item $x \in a\{1\}$ and ${\rm lann}(xa)=\mathcal{T}$.
  \item $\,_{xa}\varphi=\rho_{\mathcal{R}a,\mathcal{T}}$.
  \item $x=\rho_{\mathcal{R}a,\mathcal{T}}(1)a^{(1)}+(1-a^{(1)}a)y(1-aa^{(1)})$ where $a^{(1)} \in a\{1\}$ and $y \in \mathcal{R}$.
\end{enumerate}
\end{theorem}
In the previous results, we can replace the hypothesis
$a\{1\}\neq\emptyset$ by the condition $\rho_{\mathcal{S},{\rm
rann}(a)}(1) \in a\mathcal{R}$ (resp. $\rho_{\mathcal{S},{\rm
lann}(a)}(1) \in \mathcal{R}a$).

Part (1) of Proposition~\ref{P 1I sets} generalizes \cite[Corollary
2.9]{Ben-Israel-Greville (2003)}.
\begin{proposition}\label{P 1I sets}
Let $a \in \mathcal{R}$ and $a^{(1)} \in a\{1\}$. Let
$\mathcal{A}=\{a^{(1)}+(1-a^{(1)}a)y(1-aa^{(1)}): y \in
\mathcal{R}\}$. The following assertions hold:
\begin{enumerate}
  \item Let $\mathcal{S}$, $\mathcal{T}$ be right ideals of $\mathcal{R}$ such
that $\mathcal{R}=a\mathcal{R}\oplus\mathcal{T}$ and
$\mathcal{R}=\mathcal{S}\oplus {\rm rann}(a)$. If
$a^{(1)}a\mathcal{R}=\mathcal{S}$ and ${\rm
rann}(aa^{(1)})=\mathcal{T}$, then $\mathcal{A}=\{x \in a\{1\} :
xa\mathcal{R}=\mathcal{S} \text{ and } {\rm
rann}(ax)=\mathcal{T}\}$.
  \item Let $\mathcal{S}$, $\mathcal{T}$ be left ideals of $\mathcal{R}$ such
that $\mathcal{R}=\mathcal{R}a\oplus\mathcal{T}$ and
$\mathcal{R}=\mathcal{S}\oplus {\rm lann}(a)$. If
$\mathcal{R}aa^{(1)}=\mathcal{S}$ and ${\rm
lann}(a^{(1)}a)=\mathcal{T}$, then $\mathcal{A}=\{x \in a\{1\}:
\mathcal{R}ax=\mathcal{S} \text{ and } {\rm
lann}(xa)=\mathcal{T}\}$.
  \item Let $\mathcal{S}$ be a right ideal of $\mathcal{R}$ and $\mathcal{S}'$ be a left ideal of $\mathcal{R}$ such
that $\mathcal{R}=\mathcal{S}\oplus {\rm rann}(a)$ and
$\mathcal{R}=\mathcal{S}'\oplus {\rm lann}(a)$. If
$a^{(1)}a\mathcal{R}=\mathcal{S}$ and
$\mathcal{R}aa^{(1)}=\mathcal{S}'$, then $\mathcal{A}=\{x \in
a\{1\}: xa\mathcal{R}=\mathcal{S} \text{ and }
\mathcal{R}ax=\mathcal{S}'\}$.
  \item Let $\mathcal{T}$ be a right ideal of
$\mathcal{R}$ and $\mathcal{T}'$ be a left ideal of $\mathcal{R}$
such that $\mathcal{R}=a\mathcal{R}\oplus\mathcal{T}$ and
$\mathcal{R}=\mathcal{R}a\oplus\mathcal{T}'$. If ${\rm
rann}(aa^{(1)})=\mathcal{T}$ and ${\rm
lann}(a^{(1)}a)=\mathcal{T}'$, then $\mathcal{A}=\{x \in a\{1\} :
{\rm rann}(ax)=\mathcal{T} \text{ and } {\rm
lann}(xa)=\mathcal{T}'\}$.
  \item Let $\mathcal{S}$ be a right ideal of $\mathcal{R}$ such
that $\mathcal{R}=\mathcal{S}\oplus {\rm rann}(a)$. If
$a^{(1)}a\mathcal{R}=\mathcal{S}$, then $\mathcal{A}=\{x \in a\{1\}
: xa\mathcal{R}=\mathcal{S}\}$.
  \item Let $\mathcal{S}$ be a left ideal of $\mathcal{R}$ such
that $\mathcal{R}=\mathcal{S}\oplus {\rm lann}(a)$. If
$\mathcal{R}aa^{(1)}=\mathcal{S}$, then $\mathcal{A}=\{x \in a\{1\}:
\mathcal{R}ax=\mathcal{S}\}$.
  \item Let $\mathcal{T}$ be a right ideal of
$\mathcal{R}$ such that $\mathcal{R}=a\mathcal{R}\oplus\mathcal{T}$.
If ${\rm rann}(aa^{(1)})=\mathcal{T}$, then $\mathcal{A}=\{x \in
a\{1\} : {\rm rann}(ax)=\mathcal{T}\}$.
  \item Let $\mathcal{T}'$ be a left ideal of
$\mathcal{R}$ such that
$\mathcal{R}=\mathcal{R}a\oplus\mathcal{T}'$. If ${\rm
lann}(a^{(1)}a)=\mathcal{T}'$, then $\mathcal{A}=\{x \in a\{1\} :
{\rm lann}(xa)=\mathcal{T}'\}$.
\end{enumerate}
\end{proposition}
\begin{proof}
(1): By Theorem~\ref{T 1I xaR ax0 expression}, $x \in a\{1\}$,
$xa\mathcal{R}=\mathcal{S}$, and ${\rm rann}(ax)=\mathcal{T}$ if and
only if $\varphi_{ax}=\rho_{a\mathcal{R},\mathcal{T}}$ and
$\varphi_{xa}=\rho_{\mathcal{S},{\rm rann}(a)}$. Then $ax=aa^{(1)}$
and $xa=a^{(1)}a$. Consequently,
$x=a^{(1)}+(1-a^{(1)}a)y(1-aa^{(1)})$ with $y=x-a^{(1)}$. Conversely
if there exists $y \in \mathcal{R}$ such that
$x=a^{(1)}+(1-a^{(1)}a)y(1-aa^{(1)})$, then $ax=aa^{(1)}$ and
$xa=a^{(1)}a$. This implies that $x \in a\{1\}$,
$xa\mathcal{R}=\mathcal{S}$, and ${\rm rann}(ax)=\mathcal{T}$. This
shows that (1) holds.

Using Theorems~\ref{T 1I Rax 0xa expression}-\ref{T 1I 0xa
expression}, parts (2)-(8) can be similarly proved.
\end{proof}

\section{$\{2\}$-inverses with prescribed principal and annihilator ideals}\label{S 2I with prescribed principal and annihilator ideals}

Taking into account that $x\in a\{2\}$ if and only if $a\in x\{1\}$,
applying the results of Section~\ref{S 1I with prescribed principal
and annihilator ideals} we obtain corresponding results for
$\{2\}$-inverses $x$ of $a$ such that $ax\mathcal{R}$, ${\rm
rann}(xa)$, $\mathcal{R}xa$, and/or ${\rm lann}(ax)$ are given. In
this section, we further study $\{2\}$-inverses with prescribed
principal and/or annihilator ideals. We begin with the following two
results about uniqueness and relations to projectors.
\begin{theorem}\label{T 2I xR x0 unique}
Let $a \in \mathcal{R}$, let $\mathcal{S}$, $\mathcal{T}$ be right
ideals of $\mathcal{R}$. If $a$ has a $\{2\}$-inverse $x$ such that
$x\mathcal{R}=\mathcal{S}$ and ${\rm rann}(x)=\mathcal{T}$, then
$\varphi_{ax}=\rho_{\varphi_{a}(\mathcal{S}),\mathcal{T}}$ and
$\varphi_{xa}=\rho_{\mathcal{S},\varphi_{a}^{-1}(\mathcal{T})}$. If
there exists, this $\{2\}$-inverse $x$ is unique and will be denoted
by $a^{(2)}_{{\rm rprin}=\mathcal{S},{\rm rann}=\mathcal{T}}$.
\end{theorem}
\begin{proof}
Assume that $x \in a\{2\}$, $x\mathcal{R}=\mathcal{S}$ and ${\rm
rann}(x)=\mathcal{T}$. We have
$ax\mathcal{R}=\varphi_{a}(\mathcal{S})$ and ${\rm
rann}(xa)=\varphi_{a}^{-1}(\mathcal{T})$. Thus, by Theorem~\ref{T 2I
projectors},
$\varphi_{ax}=\rho_{\varphi_{a}(\mathcal{S}),\mathcal{T}}$ and
$\varphi_{xa}=\rho_{\mathcal{S},\varphi_{a}^{-1}(\mathcal{T})}$.
Assume that $x_{1}, x_{2} \in a\{2\}$,
$x_{1}\mathcal{R}=x_{2}\mathcal{R}=\mathcal{S}$ and ${\rm
rann}(x_{1})={\rm rann}(x_{2})=\mathcal{T}$. Using Lemma~\ref{L ro a
= a and a ro = a}(1) we get,
$x_{1}=\rho_{\mathcal{S},\varphi_{a}^{-1}(\mathcal{T})}(x_{1})=x_{2}ax_{1}=x_{2}\rho_{\varphi_{a}(\mathcal{S}),\mathcal{T}}(1)=x_{2}$.
This shows that $x$ is unique.
\end{proof}
Using Theorem~\ref{T 2I projectors} and Lemma~\ref{L ro a = a and a
ro = a}(2), we analogously obtain:
\begin{theorem}\label{T 2I Rx 0x unique}
Let $a \in \mathcal{R}$ and let $\mathcal{S}$, $\mathcal{T}$ be left
ideals of $\mathcal{R}$. If $a$ has a $\{2\}$-inverse $x$ such that
$\mathcal{R}x=\mathcal{S}$, and ${\rm lann}(x)=\mathcal{T}$, then
$\,_{xa}\varphi=\rho_{\,_{a}\varphi(\mathcal{S}),\mathcal{T}}$ and
$\,_{ax}\varphi=\rho_{\mathcal{S},\,_{a}\varphi^{-1}(\mathcal{T})}$.
If there exists, this $\{2\}$-inverse $x$ is unique and will be
denoted by $a^{(2)}_{{\rm lprin}=\mathcal{S},{\rm
lann}=\mathcal{T}}$.
\end{theorem}
The next theorem gives necessary and sufficient conditions for the
existence of $a^{(2)}_{{\rm rprin}=\mathcal{S},{\rm
rann}=\mathcal{T}}$. The equivalence (1) $\Leftrightarrow$ (3) is a
generalization of \cite[Theorem 2.8]{Cao-Xue (2013)} in complex
Banach algebras and of \cite[Theorem 2.4]{Zhu (2018)} in rings. In
\cite[Theorem 2.8]{Cao-Xue (2013)}, $\mathcal{S}=p\mathcal{R}$ and
$\mathcal{T}=q\mathcal{R}$ with $p, q \in \mathcal{R}^{\bullet}$. In
\cite[Theorem 2.4]{Zhu (2018)}, $\mathcal{S}=b\mathcal{R}$ and
$\mathcal{T}={\rm rann}(c)$ with $b, c \in \mathcal{R}$.
\begin{theorem}\label{T 2I xR x0 existencia}
Let $a \in \mathcal{R}$ and let $\mathcal{S}$, $\mathcal{T}$ be
right ideals of $\mathcal{R}$. Then the following statements are
equivalent:
\begin{enumerate}
  \item $a^{(2)}_{{\rm rprin}=\mathcal{S},{\rm rann}=\mathcal{T}}$ exists.
  \item There exists $x \in \mathcal{S}$ such that
  $\varphi_{ax}=\rho_{a\mathcal{S},\mathcal{T}}$ and ${\rm rann}(a) \cap
  \mathcal{S}=\{0\}$.
  \item $\mathcal{R}=a\mathcal{S} \oplus
\mathcal{T}$ and ${\rm rann}(a) \cap \mathcal{S}=\{0\}$.
  \item There exists $x \in \mathcal{S}$ such that $xas=s$ for all $s \in \mathcal{S}$, $1-ax \in
  \mathcal{T}$, and $xt=0$ for all $t \in \mathcal{T}$.
\end{enumerate}
\end{theorem}
\begin{proof}
(2) $\Rightarrow$ (3) and (1) $\Rightarrow$ (4) are immediate.

(1) $\Rightarrow$ (2): Suppose that $a^{(2)}_{{\rm
rprin}=\mathcal{S},{\rm rann}=\mathcal{T}}$ exists and
$x=a^{(2)}_{{\rm rprin}=\mathcal{S},{\rm rann}=\mathcal{T}}$. Then
$x \in \mathcal{S}$ and, by Theorem~\ref{T 2I xR x0 unique},
$\varphi_{ax}=\rho_{a\mathcal{S},\mathcal{T}}$ and
$\varphi_{xa}=\rho_{\mathcal{S},\varphi_{a}^{-1}(\mathcal{T})}$. In
particular, $xas=s$ for all $s \in \mathcal{S}$. From this last
property, ${\rm rann}(a) \cap \mathcal{S}=\{0\}$.

(2) $\Rightarrow$ (1): Assume that there exists $x \in \mathcal{S}$
such that $\varphi_{ax}=\rho_{a\mathcal{S},\mathcal{T}}$ and ${\rm
rann}(a) \cap \mathcal{S}=\{0\}$. Then, $x\mathcal{R} \subseteq
\mathcal{S}$, ${\rm rann}(ax)=\mathcal{T}$ and $axax=ax$. Since $x
\in \mathcal{S}$, from the last equality we obtain, $xax-x \in {\rm
rann}(a) \cap \mathcal{S}$. Thus, $xax=x$.

By Theorem~\ref{T 2I xR x0 unique},
$\varphi_{xa}=\rho_{x\mathcal{R},\varphi_{a}^{-1}(\mathcal{T})}$.
Let $s \in \mathcal{S}$ and $r, y \in \mathcal{R}$ be such that
$s=xr+y$ is the unique decomposition of $s$ as a sum of an element
in $x\mathcal{R}$ and an element in $\varphi_{a}^{-1}(\mathcal{T})$.
Then $0=a(xr-s)+ay$ where $a(xr-s) \in a\mathcal{S}$ and $ay \in
\mathcal{T}$. Since $a\mathcal{S} \cap \mathcal{T}=\{0\}$, it
follows that $a(xr-s)=0$. From this last equality, using that $xr-s
\in \mathcal{S}$ and ${\rm rann}(a) \cap \mathcal{S}=\{0\}$, we
obtain $s=xr$. Thus, $\mathcal{S}\subseteq x\mathcal{R}$.

We have proved that $x=a^{(2)}_{{\rm rprin}=\mathcal{S},{\rm
rann}=\mathcal{T}}$.

(3) $\Rightarrow$ (2): Assume that $\mathcal{R}=a\mathcal{S} \oplus
\mathcal{T}$ and ${\rm rann}(a) \cap \mathcal{S}=\{0\}$. Then, there
exists a unique $x \in \mathcal{S}$ such that
$ax=\rho_{a\mathcal{S},\mathcal{T}}(1)$. Hence,
$\varphi_{ax}=\rho_{a\mathcal{S},\mathcal{T}}$.

(4) $\Rightarrow$ (1): Suppose that there exists $x \in \mathcal{S}$
such that $xas=s$ for all $s \in \mathcal{S}$, $1-ax \in
  \mathcal{T}$, and $xt=0$ for all $t \in \mathcal{T}$. Clearly, $x \in
  a\{2\}$. We also have, $x \in \mathcal{S} \Rightarrow x\mathcal{R} \subseteq \mathcal{S}$, $\{xas=s \text{ for all } s \in \mathcal{S}\} \Rightarrow \mathcal{S} \subseteq x\mathcal{R}$, $1-ax \in
  \mathcal{T} \Rightarrow (1-ax)\mathcal{R} \subseteq
  \mathcal{T}$, and $\{xt=0 \text{ for all } t \in \mathcal{T}\} \Rightarrow \mathcal{T} \subseteq
{\rm rann}(x)$. Since ${\rm rann}(x)={\rm
rann}(ax)=(1-ax)\mathcal{R}$, we conclude that (1) holds.
\end{proof}
Using Theorem~\ref{T 2I Rx 0x unique}, we obtain the next theorem
which is analogous to Theorem~\ref{T 2I xR x0 existencia}.
\begin{theorem}\label{T 2I Rx 0x existencia}
Let $a \in \mathcal{R}$ and let $\mathcal{S}$, $\mathcal{T}$ be left
ideals of $\mathcal{R}$. Then the following statements are
equivalent:
\begin{enumerate}
  \item $a^{(2)}_{{\rm lprin}=\mathcal{S},{\rm lann}=\mathcal{T}}$ exists.
  \item There exists $x \in \mathcal{S}$ such that $\,_{xa}\varphi=\rho_{\mathcal{S}a,\mathcal{T}}$ and ${\rm lann}(a) \cap
  \mathcal{S}=\{0\}$.
  \item $\mathcal{R}=\mathcal{S}a \oplus
  \mathcal{T}$ and ${\rm lann}(a) \cap \mathcal{S}=\{0\}$.
  \item There exists $x \in \mathcal{S}$ such that $sax=s$ for all $s \in \mathcal{S}$, $1-xa \in
  \mathcal{T}$, and $tx=0$ for all $t \in \mathcal{T}$.
\end{enumerate}
\end{theorem}
The next theorem is about the elements $x \in a\{2\}$ such that $x\mathcal{R}$
and $\mathcal{R}x$ are given. If $\mathcal{A} \subseteq \mathcal{R}$, we consider ${\rm rann}(\mathcal{A})=\{r \in \mathcal{R} : ar=0 \text{ for all } a \in \mathcal{A}\}$ and ${\rm lann}(\mathcal{A})=\{r \in \mathcal{R} : ra=0 \text{ for all } a \in \mathcal{A}\}$.
\begin{theorem}\label{T 2I xR Rx existencia}
Let $a \in \mathcal{R}$, $\mathcal{S}$ be a right ideal of
$\mathcal{R}$ and $\mathcal{S}'$ be a left ideal of $\mathcal{R}$.
Then the following statements are equivalent:
\begin{enumerate}
  \item There exists $x \in a\{2\}$ such that $x\mathcal{R}=\mathcal{S}$ and $\mathcal{R}x=\mathcal{S}'$.
  \item $\mathcal{R}=a\mathcal{S} \oplus
  {\rm rann}(\mathcal{S}')$ and $\mathcal{R}=\mathcal{S}'a
\oplus {\rm lann}(\mathcal{S})$.
  \item There exists $x \in \mathcal{S} \cap \mathcal{S}'$ such that $xas=s$ for all $s \in \mathcal{S}$ and $sax=s$ for all $s \in \mathcal{S}'$.
  \item There exists $x \in \mathcal{S} \cap \mathcal{S}'$ such that
  $\varphi_{ax}=\rho_{a\mathcal{S},{\rm rann}(\mathcal{S}')}$ and $\,_{xa}\varphi=\rho_{\mathcal{S}'a,{\rm lann}(\mathcal{S})}$.
\end{enumerate}
If the above $x$ exists, then it is unique. This generalized inverse
is equal to $a^{(2)}_{{\rm rprin}=\mathcal{S},{\rm rann}={\rm
rann}(\mathcal{S}')}$ (resp. $a^{(2)}_{{\rm lprin}=\mathcal{S}',{\rm
lann}={\rm lann}(\mathcal{S})}$) and will be denoted by
$a^{(2)}_{{\rm rprin}=\mathcal{S},{\rm lprin}=\mathcal{S}'}$.
\end{theorem}
\begin{proof}
(1) $\Rightarrow$ (2)(4): Since ${\rm rann}(x)={\rm
rann}(\mathcal{R}x)$ and ${\rm lann}(x)={\rm lann}(x\mathcal{R})$,
the implications follow from Theorems~\ref{T 2I xR x0 existencia}
and \ref{T 2I Rx 0x existencia}.

(2) $\Rightarrow$ (1): Suppose that (2) holds. Assume that $s \in
\mathcal{S}$ and $as=0$. Let $z \in \mathcal{S}'$ such that
$\,_{za}\varphi=\rho_{\mathcal{S}'a,{\rm lann}(\mathcal{S})}$. Then
$s=1s=\rho_{\mathcal{S}'a,{\rm lann}(\mathcal{S})}(1)s=zas=0$. This
shows that ${\rm rann}(a)\cap\mathcal{S}=\{0\}$. Similarly, ${\rm
lann}(a)\cap\mathcal{S}'=\{0\}$. By Theorems~\ref{T 2I xR x0
existencia} and \ref{T 2I Rx 0x existencia}, $a^{(2)}_{{\rm
rprin}=\mathcal{S},{\rm rann}={\rm rann}(\mathcal{S}')}$ and
$a^{(2)}_{{\rm lprin}=\mathcal{S}',{\rm lann}={\rm
lann}(\mathcal{S})}$ exist. Denote them by $x$ and $y$,
respectively. Since ${\rm rann}(x)={\rm rann}(\mathcal{S}')={\rm
rann}(\mathcal{R}y)={\rm rann}(y)$ and ${\rm lann}(y)={\rm
lann}(\mathcal{S})={\rm lann}(x\mathcal{R})={\rm lann}(x)$, it
follows that $y(ax)=y\rho_{a\mathcal{S},{\rm
rann}(\mathcal{S}')}(1)=y$, and $(ya)x=\rho_{\mathcal{S}'a,{\rm
lann}(\mathcal{S})}(1)x=x$. Hence, $x=y=a^{(2)}_{{\rm
rprin}=\mathcal{S},{\rm rann}={\rm
rann}(\mathcal{S}')}=a^{(2)}_{{\rm lprin}=\mathcal{S}',{\rm
lann}={\rm lann}(\mathcal{S})}$. In particular, $x \in a\{2\}$,
$x\mathcal{R}=\mathcal{S}$, $\mathcal{R}x=\mathcal{S}'$, and by
Theorem~\ref{T 2I xR x0 unique} (or Theorem~\ref{T 2I Rx 0x
unique}), this generalized inverse is unique.

(1) $\Leftrightarrow$ (3): It is immediate.

(4)  $\Rightarrow$ (1): Since $x \in \mathcal{S} \cap \mathcal{S}'$,
we have $x\mathcal{R} \subseteq \mathcal{S}$, $\mathcal{R}x
\subseteq \mathcal{S}'$, ${\rm lann}(\mathcal{S}) \subseteq {\rm
lann}(x)$, ${\rm rann}(\mathcal{S}') \subseteq {\rm rann}(x)$,
$xax=x\rho_{a\mathcal{S},{\rm
rann}(\mathcal{S}')}(1)=\rho_{\mathcal{S}'a,{\rm
lann}(\mathcal{S})}(1)x=x$. If $s \in \mathcal{S}$ and $s' \in
\mathcal{S}'$, then $xas=\rho_{\mathcal{S}'a,{\rm
lann}(\mathcal{S})}(1)s=s$ and $s'ax=s'\rho_{a\mathcal{S},{\rm
rann}(\mathcal{S}')}(1)=s'$. Thus, $\mathcal{S} \subseteq
x\mathcal{R}$, $\mathcal{S}' \subseteq \mathcal{R}x$. We conclude
that (1) holds.
\end{proof}
Assume that there exist $b, c \in \mathcal{R}$ such that
$\mathcal{S}=b\mathcal{R}$ and $\mathcal{S}'=\mathcal{R}c$. In this
case, Theorem~\ref{T 2I xR Rx existencia}(2) coincides with
\cite[Proposition 2.7(ii)]{Drazin (2012)}, whereas Theorem~\ref{T 2I
xR Rx existencia}(3) coincides with the definition of the $(b,c)$
inverse (see \cite[Definition 1.3]{Drazin (2012)} and \cite[page
103]{Rakic (2017)}).

The following theorem is about $\{2\}$-inverses $x$ such that ${\rm
rann}(x)$ and ${\rm lann}(x)$ are given.
\begin{theorem}\label{T 2I 0x x0 existencia}
Let $a, x \in \mathcal{R}$, $\mathcal{T}$ be a right ideal of
$\mathcal{R}$, and $\mathcal{T}'$ be a left ideal of $\mathcal{R}$.
Then the following statements are equivalent:
\begin{enumerate}
  \item $x \in a\{2\}$, ${\rm rann}(x)=\mathcal{T}$, and ${\rm lann}(x)=\mathcal{T}'$.
    \item $\varphi_{ax}=\rho_{ax\mathcal{R},\mathcal{T}}$, $\,_{xa}\varphi=\rho_{\mathcal{R}xa,\mathcal{T}'}$, and ${\rm rann}(a) \cap
  x\mathcal{R}=\{0\}$.
      \item $\varphi_{ax}=\rho_{ax\mathcal{R},\mathcal{T}}$, $\,_{xa}\varphi=\rho_{\mathcal{R}xa,\mathcal{T}'}$, and ${\rm lann}(a) \cap
  \mathcal{R}x=\{0\}$.
      \item $\varphi_{ax}=\rho_{ax\mathcal{R},\mathcal{T}}$, $\,_{xa}\varphi=\rho_{\mathcal{R}xa,\mathcal{T}'}$, and $\mathcal{T} \subseteq {\rm rann}(x)$.
      \item $\varphi_{ax}=\rho_{ax\mathcal{R},\mathcal{T}}$, $\,_{xa}\varphi=\rho_{\mathcal{R}xa,\mathcal{T}'}$, and $\mathcal{T}' \subseteq {\rm lann}(x)$.

  \item $1-ax \in \mathcal{T}$, $xt=0$ for all $t \in \mathcal{T}$, $1-xa
\in \mathcal{T}'$, and $tx=0$ for all $t \in \mathcal{T}'$.
\end{enumerate}
If the above $x$ exists, then it is unique and will be denoted by
$a^{(2)}_{{\rm lann}=\mathcal{T}',{\rm rann}=\mathcal{T}}$.
\end{theorem}
\begin{proof}
The implications (1) $\Rightarrow$ (2)(3) follow from
Theorems~\ref{T 2I xR x0 existencia} and \ref{T 2I Rx 0x
existencia}, the equivalences (1) $\Leftrightarrow$ (4) and (1)
$\Leftrightarrow$ (5) follow from Theorem~\ref{T 2I projectors}, and
the implication (1) $\Rightarrow$ (6) is immediate.

Assume that (2) holds. Then $xax-x \in {\rm rann}(a) \cap
  x\mathcal{R}$. Taking into account that ${\rm rann}(a) \cap
  x\mathcal{R}=\{0\}$, we get $x\in a\{2\}$. Applying Theorem~\ref{T 2I projectors}, we conclude that (1) holds. The proof of (3) $\Rightarrow$ (1) is similar.

The proof of (6) $\Rightarrow$ (1) is similar to the proof of (3)
$\Rightarrow$ (1) in Theorems~\ref{T 2I xR x0 existencia} and \ref{T
2I Rx 0x existencia}.

Let $x, y \in a\{2\}$ be such that ${\rm rann}(x)={\rm
rann}(y)=\mathcal{T}$ and ${\rm lann}(x)={\rm
lann}(y)=\mathcal{T}'$. Since $x\{1\}\neq\emptyset$ and
$y\{1\}\neq\emptyset$, from Lemma~\ref{L L 2.5 RDD (2014b)},
$x\mathcal{R}=y\mathcal{R}$ and $x\mathcal{R}=\mathcal{R}y$. By
Theorem~\ref{T 2I xR x0 unique} (or Theorem~\ref{T 2I Rx 0x
unique}), $x=y$. This proves the uniqueness.
\end{proof}
We note that \cite[Proposition 3.1]{Zhu (2018)}, in which
$\mathcal{T}={\rm rann}(c)$ and $\mathcal{T}'={\rm lann}(b)$ for
some $b, c \in \mathcal{R}$, follows from Theorem~\ref{T 2I 0x x0
existencia}(1)(6).

In \cite{Drazin (2012)}, the connection of $(b,c)$ inverses with the
Mitsch $\mathcal{M}$ partial order \cite{Mitsch (1986)} in a
semigroup is established. Here, we consider $\mathcal{M}$ in
relation to the $\{2\}$-inverses considered in this section
obtaining results similar to \cite[Lemmas 4.2 and 6.5, Theorems 4.3
and 6.6]{Drazin (2012)}. We recall that if $y, z \in \mathcal{R}$,
then $y\mathcal{M}z$ if there exists $v, w \in \mathcal{R}$ such
that $vz=vy=y=yw=zw$. Let $\mathcal{S}$, $\mathcal{T}$ (resp.
$\mathcal{S}'$, $\mathcal{T}'$) be right (resp. left) ideals of
$\mathcal{R}$. Consider the pair of sets

\centerline{$Y_{a,\mathcal{S},\mathcal{S}'}=\{y \in \mathcal{R}:
y\in a\{2\}, y\in\mathcal{S}\cap\mathcal{S}'\}$,}

\centerline{$Z_{a,\mathcal{S},\mathcal{S}'}=\{z \in \mathcal{R}:
z\in a\{2\}, \mathcal{S}\subseteq z\mathcal{R},
\mathcal{S}'\subseteq\mathcal{R}z\}$,}

\smallskip

\centerline{$Y_{a,\mathcal{S},\mathcal{T}}=\{y \in \mathcal{R}: y\in
a\{2\}, y\in\mathcal{S}, \mathcal{T} \subseteq {\rm rann}(y)\}$,}

\centerline{$Z_{a,\mathcal{S},\mathcal{T}}=\{z \in \mathcal{R}: z\in
a\{2\}, \mathcal{S}\subseteq z\mathcal{R}, {\rm rann}(z) \subseteq
\mathcal{T}\}$,}

\smallskip

\centerline{$Y_{a,\mathcal{S}',\mathcal{T}'}=\{y \in \mathcal{R}:
y\in a\{2\}, y\in\mathcal{S}', \mathcal{T}' \subseteq {\rm
lann}(y)\}$,}

\centerline{$Z_{a,\mathcal{S}',\mathcal{T}'}=\{z \in \mathcal{R}:
z\in a\{2\}, \mathcal{S}'\subseteq \mathcal{R}z, {\rm lann}(z)
\subseteq \mathcal{T}'\}$,}

\smallskip

\centerline{$Y_{a,\mathcal{T},\mathcal{T}'}=\{y \in \mathcal{R}:
y\in a\{2\}, \mathcal{T} \subseteq {\rm rann}(y), \mathcal{T}'
\subseteq {\rm lann}(y)\}$,}

\centerline{$Z_{a,\mathcal{T},\mathcal{T}'}=\{z \in \mathcal{R}:
z\in a\{2\}, {\rm rann}(z) \subseteq \mathcal{T}, {\rm lann}(z)
\subseteq \mathcal{T}'\}$.}

\begin{lemma}\label{L 2I M}
Let $a \in \mathcal{R}$, and $\mathcal{S}$, $\mathcal{T}$ (resp.
$\mathcal{S}'$, $\mathcal{T}'$) be right (resp. left) ideals of
$\mathcal{R}$. Then:
\begin{enumerate}
  \item $y\mathcal{M}z$ for each $y \in Y_{a,\mathcal{S},\mathcal{S}'}$ and $z \in Z_{a,\mathcal{S},\mathcal{S}'}$.
  \item $y\mathcal{M}z$ for each $y \in Y_{a,\mathcal{S},\mathcal{T}}$ and $z \in Z_{a,\mathcal{S},\mathcal{T}}$.
  \item $y\mathcal{M}z$ for each $y \in Y_{a,\mathcal{S}',\mathcal{T}'}$ and $z \in Z_{a,\mathcal{S}',\mathcal{T}'}$.
  \item $y\mathcal{M}z$ for each $y \in Y_{a,\mathcal{T},\mathcal{T}'}$ and $z \in Z_{a,\mathcal{T},\mathcal{T}'}$.
\end{enumerate}
\end{lemma}
\begin{proof}
(1): Let $y \in Y_{a,\mathcal{S},\mathcal{S}'}$ and $z \in
Z_{a,\mathcal{S},\mathcal{S}'}$. Then $y, z \in \{2\}$,
$\varphi_{za}=\rho_{z\mathcal{R},{\rm rann}(za)}$,
$\,_{az}\varphi=\rho_{\mathcal{R}z,{\rm lann}(az)}$, $y \in
\mathcal{S} \subseteq z\mathcal{R} = za\mathcal{R}$, and $y \in
\mathcal{S}' \subseteq \mathcal{R}z = \mathcal{R}az$. Therefore,
$zay=\rho_{z\mathcal{R},{\rm rann}(za)}(y)=y$ and
$yaz=\rho_{\mathcal{R}z,{\rm lann}(az)}(y)=y$. Now, as in the proof
of \cite[Lemma 4.2]{Drazin (2012)}, $(ya)z=(ya)y=y=y(ay)=z(ay)$.

The proofs of (2), (3), and (4) are similar to the proof of (1).
\end{proof}
\begin{theorem}
Let $a \in \mathcal{R}$, and $\mathcal{S}$, $\mathcal{T}$ (resp.
$\mathcal{S}, \mathcal{T}'$) be right (resp. left) ideals of
$\mathcal{R}$. Then:
\begin{enumerate}
  \item $x=a^{(2)}_{{\rm rprin}=\mathcal{S},{\rm lprin}=\mathcal{S}'} \Leftrightarrow x\in Y_{a,\mathcal{S},\mathcal{S}'} \cap Z_{a,\mathcal{S},\mathcal{S}'} \Leftrightarrow x={\rm max}_{\mathcal{M}}Y_{a,\mathcal{S},\mathcal{S}'} = {\rm min}_{\mathcal{M}}Z_{a,\mathcal{S},\mathcal{S}'}$.
  \item $x=a^{(2)}_{{\rm rprin}=\mathcal{S},{\rm rann}=\mathcal{T}} \Leftrightarrow x\in Y_{a,\mathcal{S},\mathcal{T}} \cap Z_{a,\mathcal{S},\mathcal{T}} \Leftrightarrow x={\rm max}_{\mathcal{M}}Y_{a,\mathcal{S},\mathcal{T}} = {\rm min}_{\mathcal{M}}Z_{a,\mathcal{S},\mathcal{T}}$.
  \item $x=a^{(2)}_{{\rm lprin}=\mathcal{S}',{\rm lann}=\mathcal{T}'} \Leftrightarrow x\in Y_{a,\mathcal{S}',\mathcal{T}'} \cap Z_{a,\mathcal{S}',\mathcal{T}'} \Leftrightarrow x={\rm max}_{\mathcal{M}}Y_{a,\mathcal{S}',\mathcal{T}'} = {\rm min}_{\mathcal{M}}Z_{a,\mathcal{S}',\mathcal{T}'}$.
  \item $x=a^{(2)}_{{\rm rann}=\mathcal{T},{\rm lann}=\mathcal{T}'} \Leftrightarrow x\in Y_{a,\mathcal{T},\mathcal{T}'} \cap Z_{a,\mathcal{T},\mathcal{T}'} \Leftrightarrow x={\rm max}_{\mathcal{M}}Y_{a,\mathcal{T},\mathcal{T}'} = {\rm min}_{\mathcal{M}}Z_{a,\mathcal{T},\mathcal{T}'}$.
\end{enumerate}
\end{theorem}
\begin{proof}
(1): The equivalence $x=a^{(2)}_{{\rm rprin}=\mathcal{S},{\rm
lprin}=\mathcal{S}'} \Leftrightarrow x\in
Y_{a,\mathcal{S},\mathcal{S}'} \cap Z_{a,\mathcal{S},\mathcal{S}'}$
and the implication $x={\rm
max}_{\mathcal{M}}Y_{a,\mathcal{S},\mathcal{S}'} = {\rm
min}_{\mathcal{M}}Z_{a,\mathcal{S},\mathcal{S}'} \Rightarrow x\in
Y_{a,\mathcal{S},\mathcal{S}'} \cap Z_{a,\mathcal{S},\mathcal{S}'}$
are immediate. The implication $x\in Y_{a,\mathcal{S},\mathcal{S}'}
\cap Z_{a,\mathcal{S},\mathcal{S}'} \Rightarrow x={\rm
max}_{\mathcal{M}}Y_{a,\mathcal{S},\mathcal{S}'} = {\rm
min}_{\mathcal{M}}Z_{a,\mathcal{S},\mathcal{S}'}$ follows from
Lemma~\ref{L 2I M}(1).

The proofs of (2), (3), and (4) are similar to the proof of (1).
\end{proof}

\section{$\{1,2\}$-inverses with prescribed principal and annihilator ideals}\label{S 12I with prescribed principal and annihilator ideals}

Let $\mathcal{S}$, $\mathcal{T}$ be right ideals of $\mathcal{R}$.
If  $x=a^{(2)}_{{\rm rprin}=\mathcal{S},{\rm rann}=\mathcal{T}} \in
a\{1\}$, then we write $x=a^{(1,2)}_{{\rm rprin}=\mathcal{S},{\rm
rann}=\mathcal{T}}$. Similar meaning will have $a^{(1,2)}_{{\rm
lprin}=\mathcal{S}',{\rm lann}=\mathcal{T}'}$, $a^{(1,2)}_{{\rm
rprin}=\mathcal{S},{\rm lprin}=\mathcal{S}'}$, and $a^{(1,2)}_{{\rm
lann}=\mathcal{T}',{\rm rann}=\mathcal{T}}$, where $\mathcal{S}'$ and
$\mathcal{T}'$ are left ideals of $\mathcal{R}$.

Theorem~\ref{T 12I xR x0 expression} below characterizes
$a^{(1,2)}_{{\rm rprin}=\mathcal{S},{\rm rann}=\mathcal{T}}$ for
right ideals $\mathcal{S}$, $\mathcal{T}$ of
  $\mathcal{R}$. The equivalences (1) $\Leftrightarrow$ (4) $\Leftrightarrow$ (8) are a
generalization of \cite[Theorem 2.12(c) and Ex.
2.37]{Ben-Israel-Greville (2003)} for finite complex matrices.
\begin{theorem}\label{T 12I xR x0 expression}
Let $a, x \in \mathcal{R}$ and $\mathcal{S}$, $\mathcal{T}$ be right
ideals of $\mathcal{R}$. Then the following assertions are
equivalent:
\begin{enumerate}
  \item $x=a^{(1,2)}_{{\rm rprin}=\mathcal{S},{\rm rann}=\mathcal{T}}$.
  \item $\varphi_{ax}=\rho_{a\mathcal{R},\mathcal{T}}$, $\varphi_{xa}=\rho_{\mathcal{S},{\rm rann}(a)}$, and
  $x \in \mathcal{S}$.
  \item $\varphi_{ax}=\rho_{a\mathcal{R},\mathcal{T}}$, $\varphi_{xa}=\rho_{\mathcal{S},{\rm rann}(a)}$, and
  ${\rm lann}(\mathcal{S}) \subseteq {\rm lann}(x)$.
  \item $\varphi_{ax}=\rho_{a\mathcal{R},\mathcal{T}}$, $\varphi_{xa}=\rho_{\mathcal{S},{\rm rann}(a)}$, and
  $\mathcal{T} \subseteq {\rm rann}(x)$.
  \item $x \in a\{1\}$, $xa\mathcal{R}=\mathcal{S}$, ${\rm rann}(ax)=\mathcal{T}$, and $x \in \mathcal{S}$.
  \item $x \in a\{1\}$, $xa\mathcal{R}=\mathcal{S}$, ${\rm rann}(ax)=\mathcal{T}$, and ${\rm lann}(\mathcal{S}) \subseteq {\rm lann}(x)$.
  \item $x \in a\{1\}$, $xa\mathcal{R}=\mathcal{S}$, ${\rm rann}(ax)=\mathcal{T}$, and $\mathcal{T} \subseteq {\rm rann}(x)$.
  \item $x=\rho_{\mathcal{S},{\rm rann}(a)}(1)a^{(1)}\rho_{a\mathcal{R},\mathcal{T}}(1)$ where $a^{(1)} \in a\{1\}$.
\end{enumerate}
\end{theorem}
\begin{proof}
The equivalences (1) $\Leftrightarrow$ (2) and (1) $\Leftrightarrow$
(4) follow from Theorem~\ref{T 12I projectors}. Note that if
$\varphi_{xa}=\rho_{\mathcal{S},{\rm rann}(a)}$, then $1-xa \in {\rm
lann}(\mathcal{S})$. Hence, by Theorem~\ref{T 12I projectors}, (1)
$\Leftrightarrow$ (3).

The implications (1) $\Rightarrow$ (5)(6)(7) are immediate. The
implications (5) $\Rightarrow$ (2), (6) $\Rightarrow$ (3), and (7)
$\Rightarrow$ (4) follow from Theorem~\ref{T 1I projectors}.

(1) $\Rightarrow$ (8): By Theorem~\ref{T 1I xaR ax0 expression},
$x=\rho_{\mathcal{S},{\rm
rann}(a)}(1)a^{(1)}\rho_{a\mathcal{R},\mathcal{T}}(1)+(1-a^{(1)}a)y(1-aa^{(1)})$ where $a^{(1)} \in a\{1\}$ and $y \in \mathcal{R}$.
Then $ax=aa^{(1)}\rho_{a\mathcal{R},\mathcal{T}}(1)$ and
$xa=\rho_{\mathcal{S},{\rm rann}(a)}(1)a^{(1)}a$. Let $r \in
\mathcal{R}$ be such that
$\rho_{a\mathcal{R},\mathcal{T}}(1)=ar$. Since $x \in a\{2\}$, we get $x=x(ax)=(xa)a^{(1)}\rho_{a\mathcal{R},\mathcal{T}}(1)=\rho_{\mathcal{S},{\rm rann}(a)}(1)a^{(1)}aa^{(1)}ar
=\rho_{\mathcal{S},{\rm rann}(a)}(1)a^{(1)}ar=\rho_{\mathcal{S},{\rm
rann}(a)}(1)a^{(1)}\rho_{a\mathcal{R},\mathcal{T}}(1)$. This shows
that (8) holds.

(8) $\Rightarrow$ (1): By Theorem~\ref{T 1I xaR ax0 expression}, $x
\in a\{1\}$, $xa\mathcal{R}=\mathcal{S}$ and ${\rm
rann}(ax)=\mathcal{T}$. As in the proof of (1) $\Rightarrow$ (8),
$xax=\rho_{\mathcal{S},{\rm
rann}(a)}(a^{(1)}\rho_{a\mathcal{R},\mathcal{T}}(1))=x$. Hence, (1)
holds.
\end{proof}
Using Theorems~\ref{T 1I projectors}, \ref{T 12I projectors}, and
\ref{T 1I Rax 0xa expression}, we analogously obtain:
\begin{theorem}\label{T 12I Rx 0x expression}
Let $a, x \in \mathcal{R}$ and $\mathcal{S}$, $\mathcal{T}$ be left
ideals of $\mathcal{R}$. Then the following assertions are
equivalent:
\begin{enumerate}
  \item $x = a^{(1,2)}_{{\rm lprin}=\mathcal{S}',{\rm lann}=\mathcal{T}'}$.
  \item $\,_{ax}\varphi=\rho_{\mathcal{S},{\rm lann}(a)}$, $\,_{xa}\varphi=\rho_{\mathcal{R}a,\mathcal{T}}$ and
  $x \in \mathcal{S}$.
  \item $\,_{ax}\varphi=\rho_{\mathcal{S},{\rm lann}(a)}$, $\,_{xa}\varphi=\rho_{\mathcal{R}a,\mathcal{T}}$ and
  ${\rm rann}(\mathcal{S}) \subseteq {\rm rann}(x)$.
  \item $\,_{ax}\varphi=\rho_{\mathcal{S},{\rm lann}(a)}$, $\,_{xa}\varphi=\rho_{\mathcal{R}a,\mathcal{T}}$ and
  $\mathcal{T} \subseteq {\rm lann}(x)$.
  \item $x \in a\{1\}$, $\mathcal{R}ax=\mathcal{S}$, ${\rm lann}(ax)=\mathcal{T}$, $x \in \mathcal{S}$.
  \item $x \in a\{1\}$, $\mathcal{R}ax=\mathcal{S}$, ${\rm lann}(ax)=\mathcal{T}$, ${\rm rann}(\mathcal{S}) \subseteq {\rm rann}(x)$.
  \item $x \in a\{1\}$, $\mathcal{R}ax=\mathcal{S}$, ${\rm lann}(ax)=\mathcal{T}$, $\mathcal{T} \subseteq {\rm lann}(x)$.
  \item $x=\rho_{\mathcal{R}a,\mathcal{T}}(1)a^{(1)}\rho_{\mathcal{S},{\rm lann}(a)}(1)$ for some $a^{(1)} \in a\{1\}$.
\end{enumerate}
\end{theorem}
Analogously we have the following six theorems with proofs similar
to the proof of Theorem~\ref{T 12I xR x0 expression}.  As in
Section~\ref{S 1I with prescribed principal and annihilator ideals},
in the next results, we can replace the hypothesis
$a\{1\}\neq\emptyset$ by the condition $\rho_{\mathcal{S},{\rm
rann}(a)}(1) \in a\mathcal{R}$ (resp. $\rho_{\mathcal{S},{\rm
lann}(a)}(1) \in \mathcal{R}a$).
\begin{theorem}\label{T 12I xR Rx expression}
Let $a, x \in \mathcal{R}$, $\mathcal{S}$ be a right ideal of
$\mathcal{R}$, and $\mathcal{S}'$ be left a ideal of $\mathcal{R}$.
Then the following assertions are equivalent:
\begin{enumerate}
  \item $x=a^{(1,2)}_{{\rm rprin}=\mathcal{S},{\rm lprin}=\mathcal{S}'}$.
  \item $\varphi_{xa}=\rho_{\mathcal{S},{\rm rann}(a)}$, $\,_{ax}\varphi=\rho_{\mathcal{S}',{\rm lann}(a)}$ and
  $x \in \mathcal{S} \cup \mathcal{S}'$.
  \item $\varphi_{xa}=\rho_{\mathcal{S},{\rm rann}(a)}$, $\,_{ax}\varphi=\rho_{\mathcal{S}',{\rm lann}(a)}$ and
  ${\rm lann}(\mathcal{S}) \subseteq {\rm lann}(x)$.
  \item $\varphi_{xa}=\rho_{\mathcal{S},{\rm rann}(a)}$, $\,_{ax}\varphi=\rho_{\mathcal{S}',{\rm lann}(a)}$ and
  ${\rm rann}(\mathcal{S}') \subseteq {\rm rann}(x)$.
  \item $x \in a\{1\}$, $xa\mathcal{R}=\mathcal{S}$, $\mathcal{R}ax=\mathcal{S}'$, $x \in \mathcal{S} \cup \mathcal{S}'$.
  \item $x \in a\{1\}$, $xa\mathcal{R}=\mathcal{S}$, $\mathcal{R}ax=\mathcal{S}'$, ${\rm lann}(\mathcal{S}) \subseteq {\rm lann}(x)$.
  \item $x \in a\{1\}$, $xa\mathcal{R}=\mathcal{S}$, $\mathcal{R}ax=\mathcal{S}'$, ${\rm rann}(\mathcal{S}') \subseteq {\rm rann}(x)$.
  \item $a\{1\}\neq\emptyset$ and $x=\rho_{\mathcal{S},{\rm rann}(a)}(1)a^{(1)}\rho_{\mathcal{S}',{\rm lann}(a)}(1)$ where $a^{(1)} \in a\{1\}$.
\end{enumerate}
\end{theorem}
\begin{theorem}\label{T 12I 0x x0 expression}
Let $a, x \in \mathcal{R}$, $\mathcal{T}$ be a right ideal of
$\mathcal{R}$, and $\mathcal{T}'$ be left a ideal of $\mathcal{R}$.
Then the following assertions are equivalent:
\begin{enumerate}
  \item $x=a^{(1,2)}_{{\rm lann}=\mathcal{T}',{\rm rann}=\mathcal{T}}$.
  \item $\varphi_{ax}=\rho_{a\mathcal{R},\mathcal{T}}$, $\,_{xa}\varphi=\rho_{\mathcal{R}a,\mathcal{T}'}$ and
  $\mathcal{T} \subseteq {\rm rann}(x)$.
  \item $\varphi_{ax}=\rho_{a\mathcal{R},\mathcal{T}}$, $\,_{xa}\varphi=\rho_{\mathcal{R}a,\mathcal{T}'}$ and
  $\mathcal{T}' \subseteq {\rm lann}(x)$.
  \item $x \in a\{1\}$, ${\rm rann}(ax)=\mathcal{T}$, ${\rm lann}(ax)=\mathcal{T}'$, $\mathcal{T} \subseteq {\rm rann}(x)$.
  \item $x \in a\{1\}$, ${\rm rann}(ax)=\mathcal{T}$, ${\rm lann}(ax)=\mathcal{T}'$, $\mathcal{T}' \subseteq {\rm lann}(x)$.
  \item $x=\rho_{\mathcal{R}a,\mathcal{T}'}(1)a^{(1)}\rho_{a\mathcal{R},\mathcal{T}}(1)$ where $a^{(1)} \in a\{1\}$.
\end{enumerate}
\end{theorem}
\begin{theorem}\label{T 12I xR expression}
Let $a, x \in \mathcal{R}$ and $\mathcal{S}$ be a right ideal of
$\mathcal{R}$. Then the following assertions are equivalent:
\begin{enumerate}
  \item $x \in a\{1,2\}$ and $x\mathcal{R}=\mathcal{S}$.
  \item $\varphi_{xa}=\rho_{\mathcal{S},{\rm rann}(a)}$ and
  $x \in \mathcal{S}$.
  \item $\varphi_{xa}=\rho_{\mathcal{S},{\rm rann}(a)}$ and
  ${\rm lann}(\mathcal{S}) \subseteq {\rm lann}(x)$.
  \item $x \in a\{1\}$, $xa\mathcal{R}=\mathcal{S}$, and $x \in \mathcal{S}$.
  \item $x \in a\{1\}$, $xa\mathcal{R}=\mathcal{S}$, and ${\rm lann}(\mathcal{S}) \subseteq {\rm lann}(x)$.
  \item $a\{1\}\neq\emptyset$ and $x=\rho_{\mathcal{S},{\rm rann}(a)}(1)a^{(1)}$ where $a^{(1)} \in a\{1\}$.
\end{enumerate}
\end{theorem}
\begin{theorem}\label{T 12I x0 expression}
Let $a, x \in \mathcal{R}$ and $\mathcal{T}$ be a right ideal of
$\mathcal{R}$. Then the following assertions are equivalent:
\begin{enumerate}
  \item $x \in a\{1,2\}$ and ${\rm rann}(x)=\mathcal{T}$.
  \item $\varphi_{ax}=\rho_{a\mathcal{R},\mathcal{T}}$ and
  $\mathcal{T} \subseteq {\rm rann}(x)$.
  \item $x \in a\{1\}$, ${\rm rann}(ax)=\mathcal{T}$, and $\mathcal{T} \subseteq {\rm rann}(x)$.
  \item $x=a^{(1)}\rho_{a\mathcal{R},\mathcal{T}}(1)$ where $a^{(1)} \in a\{1\}$.
\end{enumerate}
\end{theorem}
\begin{theorem}\label{T 12I Rx expression}
Let $a, x \in \mathcal{R}$ and $\mathcal{S}$ be a left ideal of
$\mathcal{R}$. Then the following assertions are equivalent:
\begin{enumerate}
  \item $x \in a\{1,2\}$ and $\mathcal{R}x=\mathcal{S}$.
  \item $\,_{ax}\varphi=\rho_{\mathcal{S},{\rm lann}(a)}$ and
  $x \in \mathcal{S}$.
  \item $\,_{ax}\varphi=\rho_{\mathcal{S},{\rm lann}(a)}$ and
  ${\rm rann}(\mathcal{S}) \subseteq {\rm rann}(x)$.
  \item $x \in a\{1\}$, $\mathcal{R}ax=\mathcal{S}$, $x \in \mathcal{S}$.
  \item $x \in a\{1\}$, $\mathcal{R}ax=\mathcal{S}$, and ${\rm rann}(\mathcal{S}) \subseteq {\rm rann}(x)$.
  \item $a\{1\}\neq\emptyset$ and $x=a^{(1)}\rho_{\mathcal{S},{\rm lann}(a)}(1)$ for some $a^{(1)} \in a\{1\}$.
\end{enumerate}
\end{theorem}
\begin{theorem}\label{T 12I 0x expression}
Let $a, x \in \mathcal{R}$ and $\mathcal{T}$ be left a ideal of
$\mathcal{R}$. Then the following assertions are equivalent:
\begin{enumerate}
  \item $x \in a\{1,2\}$ and ${\rm lann}(x) = \mathcal{T}$.
  \item $\,_{xa}\varphi=\rho_{\mathcal{R}a,\mathcal{T}}$ and
  $\mathcal{T} \subseteq {\rm lann}(x)$.
  \item $x \in a\{1\}$, ${\rm lann}(ax)=\mathcal{T}$, $\mathcal{T} \subseteq {\rm lann}(x)$.
  \item $x=\rho_{\mathcal{R}a,\mathcal{T}}(1)a^{(1)}$ for some $a^{(1)} \in a\{1\}$.
\end{enumerate}
\end{theorem}
If the $\{1,2\}$-inverses characterized in Theorems~\ref{T 12I xR
expression}-\ref{T 12I 0x expression} exist, then they are not
necessarily unique. To see this note that if $a \in \mathcal{R}^{\#}
\cap \mathcal{R}^{\dag} \cap \mathcal{R}^{\core} \cap
\mathcal{R}_{\core}$, then $a^{\#}, a^{\dag}, a^{\core}, a_{\core}
\in a\{1,2\}$,
$a^{\#}\mathcal{R}=a^{\core}\mathcal{R}=a\mathcal{R}$,
$a^{\dag}\mathcal{R}=a_{\core}\mathcal{R}=a^{\ast}\mathcal{R}$,
${\rm rann}(a^{\#})={\rm rann}(a_{\core})={\rm rann}(a)$, ${\rm
rann}(a^{\dag})={\rm rann}(a^{\core})={\rm rann}(a^{\ast})$,
$\mathcal{R}a^{\#}=\mathcal{R}a_{\core}=\mathcal{R}a$,
$\mathcal{R}a^{\dag}=\mathcal{R}a^{\core}=\mathcal{R}a^{\ast}$,
${\rm lann}(a^{\#})={\rm lann}(a^{\core})={\rm lann}(a)$, and ${\rm
lann}(a^{\dag})={\rm lann}(a_{\core})={\rm lann}(a^{\ast})$.

Theorems~\ref{T 12 xR x0 inverse group isomorphism} and \ref{T 12 Rx
0x inverse group isomorphism} give other characterizations of
$a^{(1,2)}_{{\rm rprin}=\mathcal{S},{\rm rann}=\mathcal{T}}$ and
$a^{(1,2)}_{{\rm lprin}=\mathcal{S},{\rm lann}=\mathcal{T}}$,
respectively.
\begin{theorem}\label{T 12 xR x0 inverse group isomorphism}
Let $a \in \mathcal{R}$. Let $\mathcal{S}$, $\mathcal{T}$ be right
ideals of $\mathcal{R}$ such that $\mathcal{R}=\mathcal{S}\oplus
{\rm rann}(a)$ and $\mathcal{R}=a\mathcal{R}\oplus\mathcal{T}$. Let
$(\varphi_{a})_{|\mathcal{S}}$ be the group isomorphism obtained by
the restriction of $\varphi_{a}$ from $\mathcal{S}$ to
$a\mathcal{R}$. Let $\psi: \mathcal{R} \rightarrow \mathcal{R}$ be
the group endomorphism given by
$\psi(r)=((\varphi_{a})_{|\mathcal{S}})^{-1}(\rho_{a\mathcal{R},\mathcal{T}}(r))$
for each $r\in\mathcal{R}$. Then the following assertions are
equivalent:
\begin{enumerate}
  \item $b$ is the unique element in $\mathcal{S}$ such that
$ab=\rho_{a\mathcal{R},\mathcal{T}}(1)$.
  \item $\psi=\varphi_{b}$.
  \item $b=a^{(1,2)}_{{\rm rprin}=\mathcal{S},{\rm rann}=\mathcal{T}}$.
\end{enumerate}
\end{theorem}
\begin{proof}
Note that ${\rm ker}(\psi)=\mathcal{T}$ and ${\rm
im}(\psi)=\mathcal{S}$. Since $\varphi_{a}(\rho_{\mathcal{S},{\rm
rann}(a)}(1))=\varphi_{a}(1)=\rho_{a\mathcal{R},\mathcal{T}}(a)$, we
obtain $\psi(a)=\rho_{\mathcal{S},{\rm rann}(a)}(1)$.

(1) $\Rightarrow$ (2): Let $r, s \in \mathcal{R}$ be such that
$\psi(r)=s$. We have $\psi(r)=s$ if and only if $s \in \mathcal{S}$
and
$(\varphi_{a})_{|\mathcal{S}}(s)=\rho_{a\mathcal{R},\mathcal{T}}(r)$.
Thus,
$as=(\varphi_{a})_{|\mathcal{S}}(s)=\rho_{a\mathcal{R},\mathcal{T}}(r)=\rho_{a\mathcal{R},\mathcal{T}}(1)r=abr$.
Then $s=br$. From here, $\psi=\varphi_{b}$.

(2) $\Rightarrow$ (3): Suppose that $\psi=\varphi_{b}$. Then
$b\mathcal{R}={\rm im}(\psi)=\mathcal{S}$, ${\rm rann}(b)={\rm
ker}(\psi)=\mathcal{T}$,
$ba=\varphi_{b}(a)=\psi(a)=\rho_{\mathcal{S},{\rm rann}(a)}(1)$,
$aba=a\rho_{\mathcal{S},{\rm rann}(a)}(1)=a$, and
$bab=\rho_{\mathcal{S},{\rm rann}(a)}(b)=b$. Therefore, (3) holds.

(3) $\Rightarrow$ (1): Assume that (3) holds. Then, $b \in
b\mathcal{R}=\mathcal{S}$ and, by Theorem~\ref{T 12I projectors},
$\varphi_{ab}=\rho_{a\mathcal{R},\mathcal{T}}$. In particular,
$ab=\rho_{a\mathcal{R},\mathcal{T}}(1)$. Since $\mathcal{S}\cap {\rm
rann}(a)=\{0\}$, there is a unique element that satisfies (1).
\end{proof}
Theorem~\ref{T 12 xR x0 inverse group isomorphism} can be viewed as
a generalization of \cite[Theorem 6.2.1]{Campbell-Meyer (2009)}
about the equivalence of three definitions of $\{1,2\}$-inverses of
linear transformations with prescribed range and null spaces.
Analogously to Theorem~\ref{T 12 xR x0 inverse group isomorphism},
we get:
\begin{theorem}\label{T 12 Rx 0x inverse group isomorphism}
Let $a \in \mathcal{R}$. Let $\mathcal{S}$, $\mathcal{T}$ be left
ideals of $\mathcal{R}$ such that $\mathcal{R}=\mathcal{S}\oplus
{\rm lann}(a)$ and $\mathcal{R}=\mathcal{R}a\oplus\mathcal{T}$. Let
$(\,_{a}\varphi)_{|\mathcal{S}}$ be the group isomorphism obtained
as the restriction of $\,_{a}\varphi$ from $\mathcal{S}$ to
$\mathcal{R}a$. Let $\psi: \mathcal{R} \rightarrow \mathcal{R}$ be
the group endomorphism given by
$\psi(r)=((\,_{a}\varphi)_{|\mathcal{S}})^{-1}(\rho_{\mathcal{R}a,\mathcal{T}}(r))$
for each $r\in\mathcal{R}$.  Then the following assertions are
equivalent:
\begin{enumerate}
  \item $b$ is the unique element in $\mathcal{S}$ such that
$ba=\rho_{\mathcal{R}a,\mathcal{T}}(1)$.
  \item $\psi=\,_{b}\varphi$.
  \item $b=a^{(1,2)}_{{\rm lprin}=\mathcal{S},{\rm lann}=\mathcal{T}}$.
\end{enumerate}
\end{theorem}
The next theorem generalizes \cite[Theorem 3.1]{Yu-Wang (2007)} for
finite matrices over an associative ring and \cite[Theorem
3.3]{Cao-Xue (2013)} in complex Banach algebras with
$\mathcal{S}=p\mathcal{R}$, $\mathcal{T}=q\mathcal{R}$, and $p, q
\in \mathcal{R}^{\bullet}$. The proof of (2) $\Rightarrow$ (3)
$\Rightarrow$ (1) is similar to the one presented for Theorem 3.3 in
\cite{Cao-Xue (2013)}. Here we use Theorem~\ref{T 2I xR x0
existencia} to prove (3) $\Rightarrow$ (1).
\begin{theorem}\label{T 12I xR x0 existencia}
Let $a \in \mathcal{R}$ and let $\mathcal{S}$, $\mathcal{T}$ be
right ideals of $\mathcal{R}$. Then the following statements are
equivalent:
\begin{enumerate}
  \item $a^{(1,2)}_{{\rm rprin}=\mathcal{S},{\rm rann}=\mathcal{T}}$ exists.
  \item $\mathcal{R}=a\mathcal{R}\oplus\mathcal{T}$ and $\mathcal{R}=\mathcal{S} \oplus {\rm rann}(a)$.
  \item $\mathcal{R}=a\mathcal{S}\oplus\mathcal{T}$, ${\rm rann}(a) \cap \mathcal{S}  =\{0\}$, and $a\mathcal{R}\cap\mathcal{T} =\{0\}$.
  \item There exists $x \in \mathcal{S}$ such that
  $\varphi_{ax}=\rho_{a\mathcal{S},\mathcal{T}}$, ${\rm rann}(a) \cap
  \mathcal{S}=\{0\}$, and $a\mathcal{R}\cap\mathcal{T} =\{0\}$.
\end{enumerate}
\end{theorem}
\begin{proof}
(1) $\Rightarrow$ (2): It follows from the definition of
$a^{(1,2)}_{{\rm rprin}=\mathcal{S},{\rm rann}=\mathcal{T}}$ and
Theorem~\ref{T 12I projectors}.

(2) $\Rightarrow$ (3): Suppose that
$\mathcal{R}=a\mathcal{R}\oplus\mathcal{T}$ and
$\mathcal{R}=\mathcal{S} \oplus {\rm rann}(a)$. Then
$a\mathcal{R}\cap\mathcal{T} =\{0\}$ and ${\rm rann}(a) \cap
\mathcal{S} =\{0\}$. Clearly, $a\mathcal{S}\subseteq a\mathcal{R}$.
Let $t \in a\mathcal{R}$. Then there exists $s \in \mathcal{R}$ such
that $t=as=a\rho_{\mathcal{S},{\rm rann}(a)}(s)$. Hence,
$a\mathcal{R}\subseteq a\mathcal{S}$.

(3) $\Rightarrow$ (1): Assume that
$\mathcal{R}=a\mathcal{S}\oplus\mathcal{T}$, ${\rm rann}(a) \cap
\mathcal{S} =\{0\}$, and $a\mathcal{R}\cap\mathcal{T} =\{0\}$. By
Theorem~\ref{T 2I xR x0 existencia}, $a$ has a $\{2\}$-inverse $x$
such that $x\mathcal{R}=\mathcal{S}$ and ${\rm
rann}(x)=\mathcal{T}$. Then $axa-a \in a\mathcal{R} \cap
\mathcal{T}$, and consequently $axa=a$. Hence, $x=a^{(1,2)}_{{\rm
rprin}=\mathcal{S},{\rm rann}=\mathcal{T}}$.

(3) $\Leftrightarrow$ (4): It follows from Theorem~\ref{T 2I xR x0
existencia}.
\end{proof}
Using Theorems~\ref{T 12I projectors} and \ref{T 2I Rx 0x
existencia}, we obtain the following result.
\begin{theorem}\label{T 12I Rx 0x existencia}
Let $a \in \mathcal{R}$ and let $\mathcal{S}$, $\mathcal{T}$ be left
ideals of $\mathcal{R}$. Then the following statements are
equivalent:
\begin{enumerate}
  \item $a^{(1,2)}_{{\rm lprin}=\mathcal{S},{\rm lann}=\mathcal{T}}$ exists.
  \item $\mathcal{R}=\mathcal{R}a\oplus\mathcal{T}$ and $\mathcal{R}=\mathcal{S} \oplus {\rm lann}(a)$.
  \item $\mathcal{R}=\mathcal{S}a\oplus\mathcal{T}$, ${\rm lann}(a) \cap \mathcal{S}  =\{0\}$, and $\mathcal{R}a\cap\mathcal{T} =\{0\}$.
  \item There exists $x \in \mathcal{S}$ such that $\,_{xa}\varphi=\rho_{\mathcal{S}a,\mathcal{T}}$, ${\rm lann}(a) \cap
  \mathcal{S}=\{0\}$, and $\mathcal{R}a\cap\mathcal{T} =\{0\}$.
\end{enumerate}
\end{theorem}
We note that (2) $\Rightarrow$ (1) in Theorem~\ref{T 12I xR x0
existencia} (resp. Theorem~\ref{T 12I Rx 0x existencia}) follows
also from Theorem~\ref{T 12 xR x0 inverse group isomorphism} (resp.
Theorem~\ref{T 12 Rx 0x inverse group isomorphism}). Now we obtain a
theorem that gives necessary and sufficient conditions for the
existence of $\{1,2\}$-inverses with given right and left principal
ideals.
\begin{theorem}\label{T 12I xR Rx existencia}
Let $a \in \mathcal{R}$, let $\mathcal{S}$ be a right ideal of
$\mathcal{R}$, and $\mathcal{S}'$ be a left ideal of $\mathcal{R}$.
Then the following statements are equivalent:
\begin{enumerate}
  \item $a^{(1,2)}_{{\rm rprin}=\mathcal{S},{\rm lprin}=\mathcal{S}'}$ exists.
  \item $\mathcal{R}=a\mathcal{R}\oplus{\rm rann}(\mathcal{S}')$, $\mathcal{R}=\mathcal{S} \oplus {\rm rann}(a)$, $\mathcal{R}=\mathcal{R}a\oplus{\rm lann}(\mathcal{S})$, and $\mathcal{R}=\mathcal{S}' \oplus {\rm lann}(a)$.
  \item $\mathcal{R}=a\mathcal{S}\oplus{\rm rann}(\mathcal{S}')$, $a\mathcal{R}\cap{\rm rann}(\mathcal{S}') =\{0\}$, $\mathcal{R}=\mathcal{S}a\oplus{\rm lann}(\mathcal{S})$, and $\mathcal{R}a\cap{\rm lann}(\mathcal{S}) =\{0\}$.
  \item There exists $x \in \mathcal{S}\cap\mathcal{S}'$ such that
  $\varphi_{ax}=\rho_{a\mathcal{S},{\rm rann}(\mathcal{S}')}$, $a\mathcal{R}\cap{\rm rann}(\mathcal{S}') =\{0\}$, $\,_{xa}\varphi=\rho_{\mathcal{S}'a,{\rm lann}(\mathcal{S})}$, and $\mathcal{R}a\cap{\rm lann}(\mathcal{S}) =\{0\}$.
\end{enumerate}
\end{theorem}
\begin{proof}
From Theorems~\ref{T 12I xR x0 existencia} and \ref{T 12I Rx 0x
existencia}, we get (1) $\Rightarrow$ (2) $\Rightarrow$ (3). The
implication (1) $\Rightarrow$ (4) follows from Theorem~\ref{T 12I
projectors}. The proof of (3) $\Rightarrow$ (1) is similar to the
proof of (2) $\Rightarrow$ (1) in Theorem~\ref{T 2I xR Rx
existencia} and is based on Theorems~\ref{T 12I xR x0 existencia}
and \ref{T 12I Rx 0x existencia}. The implication (4) $\Rightarrow$
(3) is immediate.
\end{proof}

From previous results, we derive the next sufficient conditions for
right/left ideal of $\mathcal{R}$ to be principal/annhililator
ideals of idempotent elements of $\mathcal{R}$.
\begin{corollary}\label{C 2I 12I ideales p q}
Let $a \in \mathcal{R}$, $\mathcal{S}$, $\mathcal{T}$ be right
ideals of $\mathcal{R}$ and $\mathcal{S}'$, $\mathcal{T}'$ be left
ideals of $\mathcal{R}$. The following assertions hold:
\begin{enumerate}
  \item If any of the equivalent statements (1)-(4)
of Theorem~\ref{T 2I xR x0 existencia} (or Theorem~\ref{T 12I xR x0
existencia}) holds, then there exist $p, q \in
\mathcal{R}^{\bullet}$ such that $\mathcal{S}=p\mathcal{R}$ and
$\mathcal{T}={\rm rann}(q)$.
  \item If any of the equivalent statements (1)-(4)
of Theorem~\ref{T 2I Rx 0x existencia} (or Theorem~\ref{T 12I Rx 0x
existencia}) holds, then there exist $p, q \in
\mathcal{R}^{\bullet}$ such that $\mathcal{S}=\mathcal{R}p$ and
$\mathcal{T}={\rm lann}(q)$.
  \item If any of the equivalent statements (1)-(4) of Theorem~\ref{T 2I xR
Rx existencia} (or Theorem~\ref{T 12I xR Rx existencia}) holds, then
there exist $p, q \in \mathcal{R}^{\bullet}$ such that
$\mathcal{S}=p\mathcal{R}$ and $\mathcal{S}'=\mathcal{R}q$.
  \item If any of the equivalent statements (1)-(6) of Theorem~\ref{T 2I 0x x0 existencia} holds, then
there exist $p, q \in \mathcal{R}^{\bullet}$ such that
$\mathcal{T}={\rm rann}(p)$ and $\mathcal{T}'={\rm lann}(q)$.
\end{enumerate}
\end{corollary}
\begin{proof}
(1): By Theorem~\ref{T 2I xR x0 existencia} (resp. Theorem~\ref{T
12I xR x0 existencia}), $x=a^{(2)}_{{\rm rprin}=\mathcal{S},{\rm
rann}=\mathcal{T}}$ (resp. $x=a^{(1,2)}_{{\rm
rprin}=\mathcal{S},{\rm rann}=\mathcal{T}}$) exists and the
conclusion follows taking $p=xa$ and $q=ax$. The rest of the proof
is similar.
\end{proof}

\section{Particular classes of $\{1\}$, $\{2\}$, and $\{1,2\}$-inverses}\label{S Particular classes of 1 2 and 12 inverses}

In this section, we apply previous results to study particular
classes of $\{1\}$, $\{2\}$, and $\{1,2\}$-inverses. We also give an
illustrative example with a matrix over a field.

\subsection{$\{1,3\}$, $\{1,4\}$, $\{1,3,4\}$, $\{1,3,6\}$, $\{1,4,8\}$, $\{1,3,7\}$, and $\{1,4,9\}$-inverses}\label{S 13I 14I}

For $\{1,3\}$-inverses we have:
\begin{theorem}\label{T 13 orthogonal projector}
Let $\mathcal{R}$ be a $\ast$-ring and $a, x \in \mathcal{R}$. Then
the following assertions are equivalent:
\begin{enumerate}
  \item $x \in
a\{1,3\}$.
  \item $\varphi_{ax} =
\rho_{a\mathcal{R},{\rm rann}(a^{\ast})}$.
  \item $\,_{ax}\varphi = \rho_{\mathcal{R}a^{\ast},{\rm lann}(a)}$.
\end{enumerate}
\end{theorem}
\begin{proof}
(1) $\Rightarrow$ (2): It follows from Theorem~\ref{T 1I projectors}
and the equality ${\rm rann}(ax)={\rm rann}(a^{\ast})$.

(2) $\Rightarrow$ (1): If $\varphi_{ax} = \rho_{a\mathcal{R},{\rm
rann}(a^{\ast})}$, then $x \in a\{1\}$, $ax \in R^{\bullet}$,
$ax\mathcal{R}=a\mathcal{R}$, and ${\rm rann}(ax)={\rm
rann}(a^{\ast})$. By Lemma~\ref{L aR a0 ortogonales}, $ax \in
R^{{\rm sym}}$, i.e, $x \in a\{3\}$.

(1) $\Leftrightarrow$ (3): It is analogous to the proof of (1)
$\Leftrightarrow$ (2).
\end{proof}
Similar to Theorem~\ref{T 13 orthogonal projector}, we obtain:
\begin{theorem}\label{T 14 orthogonal projector}
Let $\mathcal{R}$ be a $\ast$-ring and $a, x \in \mathcal{R}$. Then
the following assertions are equivalent:
\begin{enumerate}
  \item $x \in a\{1,4\}$.
  \item $\varphi_{xa} =
\rho_{a^{\ast}\mathcal{R},{\rm rann}(a)}$.
  \item $\,_{xa}\varphi = \rho_{\mathcal{R}a,{\rm lann}(a^{\ast})}$.
\end{enumerate}
\end{theorem}
As a consequence of Theorems~\ref{T 13 orthogonal projector} and
\ref{T 14 orthogonal projector}, we obtain the next theorem.
\begin{theorem}\label{T 134 orthogonal projector}
Let $\mathcal{R}$ be a $\ast$-ring and $a, x \in \mathcal{R}$. Then
the following assertions are equivalent:
\begin{enumerate}
  \item $x \in a\{1,3,4\}$.
  \item $\varphi_{ax} = \rho_{a\mathcal{R},{\rm rann}(a^{\ast})}$ and $\varphi_{xa} = \rho_{a^{\ast}\mathcal{R},{\rm rann}(a)}$.
  \item $\varphi_{ax} = \rho_{a\mathcal{R},{\rm rann}(a^{\ast})}$ and $\,_{xa}\varphi = \rho_{\mathcal{R}a,{\rm lann}(a^{\ast})}$.
  \item $\,_{ax}\varphi = \rho_{\mathcal{R}a^{\ast},{\rm lann}(a)}$ and $\varphi_{xa} = \rho_{a^{\ast}\mathcal{R},{\rm rann}(a)}$.
  \item $\,_{ax}\varphi = \rho_{\mathcal{R}a^{\ast},{\rm lann}(a)}$ and $\,_{xa}\varphi = \rho_{\mathcal{R}a,{\rm lann}(a^{\ast})}$.
\end{enumerate}
\end{theorem}
In the next two theorems, we give the projectors associated with
$\{1,3,6\}$ and $\{1,4,8\}$-inverses.
\begin{theorem}\label{T 136 orthogonal projector}
Let $\mathcal{R}$ be a $\ast$-ring and $a, x \in \mathcal{R}$. Each
of the assertions
\begin{enumerate}
  \item $\varphi_{ax} = \rho_{a\mathcal{R},{\rm rann}(a^{\ast})}$ and $\varphi_{xa} = \rho_{a\mathcal{R},{\rm rann}(a)}$,
  \item $\varphi_{ax} = \rho_{a\mathcal{R},{\rm rann}(a^{\ast})}$ and $\,_{xa}\varphi = \rho_{\mathcal{R}a,{\rm lann}(a)}$,
  \item $\,_{ax}\varphi = \rho_{\mathcal{R}a^{\ast},{\rm lann}(a)}$ and $\varphi_{xa} = \rho_{a\mathcal{R},{\rm rann}(a)}$,
  \item $\,_{ax}\varphi = \rho_{\mathcal{R}a^{\ast},{\rm lann}(a)}$ and $\,_{xa}\varphi = \rho_{\mathcal{R}a,{\rm lann}(a)}$,
\end{enumerate}
implies $x \in a\{1,3,6\}$.
\end{theorem}
\begin{proof}
Suppose that (1) holds. By Theorem~\ref{T 13 orthogonal projector},
$x \in a\{1,3\}$. Since $xa^{2}=\rho_{a\mathcal{R},{\rm
rann}(a)}(a)=a$, we have $x \in a\{6\}$. The remainder of the
implications can be similarly proved.
\end{proof}
We analogously have:
\begin{theorem}\label{T 148 orthogonal projector}
Let $\mathcal{R}$ be a $\ast$-ring and $a, x \in \mathcal{R}$. Each
of the assertions
\begin{enumerate}
  \item $\varphi_{ax} = \rho_{a\mathcal{R}, {\rm rann}(a)}$ and $\varphi_{xa} = \rho_{a^{\ast}\mathcal{R}, {\rm rann}(a)}$,
  \item $\varphi_{ax} = \rho_{a\mathcal{R}, {\rm rann}(a)}$ and $\,_{xa}\varphi = \rho_{\mathcal{R}a,{\rm lann}(a^{\ast})}$,
  \item $\,_{ax}\varphi = \rho_{\mathcal{R}a,{\rm lann}(a)}$ and $\varphi_{xa} = \rho_{a^{\ast}\mathcal{R}, {\rm rann}(a)}$,
  \item $\,_{ax}\varphi = \rho_{\mathcal{R}a,{\rm lann}(a)}$ and $\,_{xa}\varphi = \rho_{\mathcal{R}a,{\rm lann}(a^{\ast})}$,
\end{enumerate}
implies $x \in a\{1,4,8\}$.
\end{theorem}
From Theorems~\ref{T 13 orthogonal projector} and \ref{T 14
orthogonal projector}, we obtain the next two theorems.
\begin{theorem}\label{T 137 orthogonal projector}
Let $\mathcal{R}$ be a $\ast$-ring and $a, x \in \mathcal{R}$. Then
the following assertions are equivalent:
\begin{enumerate}
  \item $x \in a\{1,3,7\}$.
  \item $\varphi_{ax} = \rho_{a\mathcal{R},{\rm rann}(a^{\ast})}$ and $x\in a\mathcal{R}$.
  \item $\,_{ax}\varphi = \rho_{\mathcal{R}a^{\ast},{\rm lann}(a)}$ and ${\rm lann}(a)\subseteq{\rm lann}(x)$.
\end{enumerate}
\end{theorem}
We note that the elements of $a\{1,3,7\}$ are the \emph{right core
inverses} of $a$ which are a particular case of right $(b,c)$
inverse of $a$ (see \cite{Drazin (2016), Wang-Mosic (2021),
Wang-Mosic-Gao (2019)}, in particular, \cite[Theorem 5.1]{Wang-Mosic
(2021)}).
\begin{theorem}\label{T 149 orthogonal projector}
Let $\mathcal{R}$ be a $\ast$-ring and $a, x \in \mathcal{R}$. Then
the following assertions are equivalent:
\begin{enumerate}
  \item $x \in a\{1,4,9\}$.
  \item $\varphi_{xa} = \rho_{a^{\ast}\mathcal{R},{\rm rann}(a)}$ and ${\rm rann}(a) \subseteq {\rm rann}(x)$.
  \item $\,_{xa}\varphi = \rho_{\mathcal{R}a,{\rm lann}(a^{\ast})}$ and $x\in\mathcal{R}a$.
\end{enumerate}
\end{theorem}

In Section~\ref{S 1I with prescribed principal and annihilator
ideals}, we studied $\{1\}$-inverses with given principal and
annihilator ideals. Now, we consider some examples of sets of these
$\{1\}$-inverses. Let $a \in \mathcal{R}$. From Theorems~\ref{T 1I
projectors} and \ref{T 15 projector},
\begin{align*}
  a\{1,5\} &= \{x \in a\{1\} : xa\mathcal{R}=a\mathcal{R} \text{ and }
{\rm rann}(ax)={\rm rann}(a)\} \\
   &= \{x \in a\{1\}: \mathcal{R}ax=\mathcal{R}a \text{ and }
{\rm lann}(xa)={\rm lann}(a)\}.
\end{align*}
Let $\mathcal{R}$ be a $\ast$-ring and $a \in \mathcal{R}$. By Theorems~\ref{T 1I ax0 expression}, \ref{T 1I Rax expression}
and \ref{T 13 orthogonal projector},
\[a\{1,3\}=\{x \in a\{1\}: {\rm rann}(ax)={\rm rann}(a^{\ast})\}=\{x \in a\{1\}: \mathcal{R}ax=\mathcal{R}a^{\ast}\},\] and by Theorems~\ref{T 1I xaR expression}, \ref{T 1I 0xa expression} and \ref{T 14 orthogonal projector},
\[a\{1,4\} = \{x \in a\{1\}: xa\mathcal{R}=a^{\ast}\mathcal{R}\}=\{x \in a\{1\}: {\rm lann}(xa)={\rm lann}(a^{\ast})\}.\] From the above equalities, we get
\begin{align*}
a\{1,3,4\} &= \{x \in a\{1\}: xa\mathcal{R}=a^{\ast}\mathcal{R} \text{ and } {\rm rann}(ax)={\rm rann}(a^{\ast})\}\\
&= \{x \in a\{1\}: {\rm lann}(xa)={\rm lann}(a^{\ast}) \text{ and } {\rm rann}(ax)={\rm rann}(a^{\ast})\} \\
&= \{x \in a\{1\}: xa\mathcal{R}=a^{\ast}\mathcal{R} \text{ and }
\mathcal{R}ax=\mathcal{R}a^{\ast}\} = \{x \in a\{1\}:
\mathcal{R}ax=\mathcal{R}a^{\ast} \text{ and } {\rm lann}(xa)={\rm
lann}(a^{\ast})\}.
\end{align*}
From Theorems~\ref{T 1I xaR ax0 expression} and \ref{T 136
orthogonal projector},
\begin{align}\label{E 136 inverse xaR ax0}
a\{1,3,6\} &\supseteq \{x \in a\{1\}: xa\mathcal{R}=a\mathcal{R} \text{ and } {\rm rann}(ax)={\rm rann}(a^{\ast})\}\nonumber\\
&= \{x \in a\{1\}: {\rm lann}(xa)={\rm lann}(a) \text{ and } {\rm
rann}(ax)={\rm rann}(a^{\ast})\}\nonumber\\ &= \{x \in a\{1\}:
xa\mathcal{R}=a\mathcal{R} \text{ and }
\mathcal{R}ax=\mathcal{R}a^{\ast}\} = \{x \in a\{1\}:
\mathcal{R}ax=\mathcal{R}a^{\ast} \text{ and } {\rm lann}(xa)={\rm
lann}(a))\}
\end{align} and from Theorems~\ref{T 1I Rax 0xa expression} and \ref{T 148 orthogonal projector},
\begin{align*}
a\{1,4,8\} &\supseteq \{x \in a\{1\}: xa\mathcal{R}=a^{\ast}\mathcal{R} \text{ and } {\rm rann}(ax)={\rm rann}(a)\}\\
&= \{x \in a\{1\}: {\rm lann}(xa)={\rm lann}(a^{\ast}) \text{ and } {\rm rann}(ax)={\rm rann}(a)\} \\
&= \{x \in a\{1\}: xa\mathcal{R}=a^{\ast}\mathcal{R} \text{ and }
\mathcal{R}ax=\mathcal{R}a\} = \{x \in a\{1\}:
\mathcal{R}ax=\mathcal{R}a \text{ and } {\rm lann}(xa)={\rm
lann}(a^{\ast})\}.
\end{align*}
We comment on the previous inclusions. Let $\mathcal{R}$ be a
$\ast$-ring and $a, x \in \mathcal{R}$ be such that $x\in
a\{1,3,6\}$. Then $a\mathcal{R} \subseteq xa\mathcal{R}$ and ${\rm
rann}(ax)={\rm rann}(x^{\ast}a^{\ast})={\rm rann}(a^{\ast})$.
Depending on the ring $\mathcal{R}$, we can always have
$a\mathcal{R} = xa\mathcal{R}$ or not; consequently, the equality in
(\ref{E 136 inverse xaR ax0}) is always satisfied or not. We next
consider two examples. Let $\mathcal{R}=\mathbb{C}^{n \times n}$ and
$A, X \in \mathbb{C}^{n \times n}$ be such that $X \in A\{1,3,6\}$.
Then ${\rm R}(A) \subseteq {\rm R}(XA)$ and ${\rm dim}({\rm
R}(XA))=n-{\rm dim}({\rm N}(XA))=n-{\rm dim}({\rm N}(A))={\rm
dim}({\rm R}(A))$. Thus, ${\rm R}(XA) = {\rm R}(A)$ and the equality
in (\ref{E 136 inverse xaR ax0}) holds for $A$. Let now
$\ell^{2}(\mathbb{N})$ be the Hilbert space of the complex sequences
$x=(x_{i})_{i \in \mathbb{N}}$ such that
$\sum_{i=1}^{\infty}|x_{i}|^{2} < \infty$ with inner product
$\langle x, y \rangle=\sum_{i=1}^{\infty}x_{i}\overline{y_{i}}$. Let
$\mathcal{R}=\mathcal{B}(\ell^{2}(\mathbb{N}))$ be the ring of all
bounded linear operators from $\ell^{2}(\mathbb{N})$ to
$\ell^{2}(\mathbb{N})$. Let $A, X \in
\mathcal{B}(\ell^{2}(\mathbb{N}))$ defined by $A(x_{1}, x_{2},
\ldots)=(0, x_{1}, x_{2}, \ldots)$ and $X(x_{1}, x_{2},
\ldots)=(x_{2}, x_{3}, \ldots)$. These operators were considered in
\cite[Remark 3.1]{Rakic-Dincic-Djordjevic (2014a)} and satisfy $X
\in A\{1,3,6\}$ and ${\rm R}(A) \varsubsetneq {\rm
R}(XA)=\ell^{2}(\mathbb{N})$. Therefore, the strict inclusion in
(\ref{E 136 inverse xaR ax0}) holds for $A$. Similar considerations
are valid for the other inclusion.

By Theorems~\ref{T 1I ax0
expression}, \ref{T 1I Rax expression} and \ref{T 137 orthogonal projector},
\begin{align*}
a\{1,3,7\} &= \{x \in a\{1\}: {\rm rann}(ax)={\rm rann}(a^{\ast}) \text{ and } x \in a\mathcal{R}\}\\
&= \{x \in a\{1\}: \mathcal{R}ax=\mathcal{R}a^{\ast} \text{ and }
{\rm lann}(a)\subseteq{\rm lann}(x)\}
\end{align*} and by Theorems~\ref{T 1I xaR expression}, \ref{T 1I 0xa expression} and \ref{T 149 orthogonal projector},
\begin{align*}
a\{1,4,9\} &= \{x \in a\{1\}: xa\mathcal{R}=a^{\ast}\mathcal{R} \text{ and } {\rm rann}(a)\subseteq{\rm rann}(x)\}\\
&= \{x \in a\{1\}: {\rm lann}(xa)={\rm lann}(a^{\ast}) \text{ and }
x\in\mathcal{R}a\}.
\end{align*}

\subsection{Generalizations of Moore-Penrose, core, and dual core inverses}\label{S Moore-Penrose core and dual core inverses}

In this section, we consider some generalizations of
Moore-Penrose, core, and dual core inverses. As in Section~\ref{S
13I 14I},  the focus is on their relation to projectors.

\subsubsection{The $(e,f)$ Moore-Penrose inverse}\label{S ef MPI}

Let $\mathcal{R}$ be a $\ast$-ring. Let $a \in \mathcal{R}$ and $e,
f \in \mathcal{R}^{-1} \cap \mathcal{R}^{\rm sym}$. In
\cite{Mosic-Djordjevic (2011)}, if  $x \in a\{1,2\}$,
$(eax)^{\ast}=eax$ and $(fxa)^{\ast}=fxa$, then $x \in \mathcal{R}$
is called the \emph{(weighted) $(e,f)$ Moore-Penrose inverse} of
$a$. This generalized inverse is denoted by $a^{\dag}_{e,f}$ and
$a^{\dag}=a^{\dag}_{1,1}$. If $x=a^{\dag}_{e,f}$, then $xe^{-1} \in
(ea)\{1\}$ and $fx \in (af^{-1})\{1\}$.  More details can be found
in, e.g., \cite{Mosic-Djordjevic (2011), Song-Zhu (2020), Zhu-Wang
(2021)}. If $a^{\dag}_{e,f}$ exists, then
\[a^{\dag}_{e,f}=a^{(1,2)}_{{\rm
rprin}=f^{-1}a^{\ast}\mathcal{R},{\rm rann}={\rm
rann}(a^{\ast}e)}=a^{(1,2)}_{{\rm lprin}=\mathcal{R}a^{\ast}e,{\rm
lann}={\rm lann}(f^{-1}a^{\ast})}=a^{(1,2)}_{{\rm
rprin}=f^{-1}a^{\ast}\mathcal{R},{\rm
lprin}=\mathcal{R}a^{\ast}e}=a^{(1,2)}_{{\rm lann}={\rm
lann}(f^{-1}a^{\ast}),{\rm rann}={\rm rann}(a^{\ast}e)}.\] Consider
the conditions

\begin{flushleft}
\setlength{\abovedisplayskip}{0pt}
\setlength{\belowdisplayskip}{0pt}
\begin{subequations}
\begin{minipage}{.54\textwidth}
\begin{align}
&\varphi_{ax} = \rho_{a\mathcal{R}, {\rm rann}(a^{\ast}e)}, \varphi_{xa} = \rho_{f^{-1}a^{\ast}\mathcal{R}, {\rm rann}(a)}.\label{E efMPI 1p}\\
&\,_{ax}\varphi = \rho_{\mathcal{R}a^{\ast}e,{\rm lann}(a)}, \,_{xa}\varphi = \rho_{\mathcal{R}a,{\rm lann}(f^{-1}a^{\ast})}.\\
&\varphi_{ax} = \rho_{a\mathcal{R}, {\rm rann}(a^{\ast}e)}, \,_{xa}\varphi = \rho_{\mathcal{R}a,{\rm lann}(f^{-1}a^{\ast})}.\\
&\,_{ax}\varphi =
\rho_{\mathcal{R}a^{\ast}e,{\rm lann}(a)}, \varphi_{xa} =
\rho_{f^{-1}a^{\ast}\mathcal{R}, {\rm rann}(a)}.\label{E efMPI 4p}
\end{align}
\end{minipage}
\begin{minipage}{.01\textwidth}
\end{minipage}
\begin{minipage}{.45\textwidth}
\begin{align}
&x\mathcal{R} \subseteq f^{-1}a^{\ast}\mathcal{R}.\label{E efMPI 1a}\\
&{\rm lann}(f^{-1}a^{\ast}) \subseteq {\rm lann}(x).\\
&\mathcal{R}x \subseteq \mathcal{R}a^{\ast}e.\\
&{\rm rann}(a^{\ast}e) \subseteq {\rm rann}(x).\label{E efMPI 4a}
\end{align}
\end{minipage}
\end{subequations}
\end{flushleft}

From Theorems~\ref{T 12I xR x0 expression}-\ref{T 12I 0x x0
expression}, $a^{\dag}_{e,f}$ exists and $x=a^{\dag}_{e,f}$ if and
only if one of the conditions (\ref{E efMPI 1p})-(\ref{E efMPI 4p})
holds and one of the conditions (\ref{E efMPI 1a})-(\ref{E efMPI
4a}) holds. Let now $e=f=1$. For Hilbert space operators with closed
range, conditions (\ref{E efMPI 1p}) and (\ref{E efMPI 1a}) coincide
with the conditions (iii) of \cite[Theorem 1]{Petryshyn (1967)}.
Conditions (\ref{E efMPI 1p}) and (\ref{E efMPI 4a}) are a
generalization of the conditions of \cite[Ex.
2.58]{Ben-Israel-Greville (2003)}. The relations of generalized
inverses to projectors can be used to extend to any ring results of
matrices and operators. In particular, the characterizations of the
Moore-Penrose inverse using orthogonal projectors given in
Section~\ref{S ef MPI} can be used to generalize results of
\cite[Section 2]{Morillas (2023)} to $\ast$-rings.

\subsubsection{The $e$-core and the $f$-dual core inverses}

Let $\mathcal{R}$ be a $\ast$-ring. Let $a, x \in \mathcal{R}$ and
$e, f \in \mathcal{R}^{-1} \cap \mathcal{R}^{\rm sym}$. Then $x$ is
the \emph{$e$-core inverse} of $a$ if $x\in a\{1\}$,
$x\mathcal{R}=a\mathcal{R}$, and
$\mathcal{R}x=\mathcal{R}a^{\ast}e$, whereas $x$ is the
\emph{$f$-dual core inverse} of $a$ if $x\in a\{1\}$,
$x\mathcal{R}=f^{-1}a^{\ast}\mathcal{R}$, and
$\mathcal{R}x=\mathcal{R}a$. These generalized inverses were defined
and studied in \cite{Mosic-Deng-Ma (2018)} (see also, e.g.,
\cite{Zhu-Wang (2021)} for more properties). If they exist, then
they are unique, and we denote them with $a^{\core,e}$ and
$a_{\core,f}$, respectively. We have $a^{\core,1}=a^{\core}$ and
$a_{\core,1}=a_{\core}$. By \cite[Theorems 2.1 and
2.2]{Mosic-Deng-Ma (2018)}, $a^{\core,e}, a_{\core,f} \in a\{1,2\}$.

If $a^{\core,e}$ exists, then \[a^{\core,e}=a^{(1,2)}_{{\rm
rprin}=a\mathcal{R},{\rm
lprin}=\mathcal{R}a^{\ast}e}=a^{(1,2)}_{{\rm
rprin}=a\mathcal{R},{\rm rann}={\rm
rann}(a^{\ast}e)}=a^{(1,2)}_{{\rm lprin}=\mathcal{R}a^{\ast}e,{\rm
lann}={\rm lann}(a)}=a^{(1,2)}_{{\rm lann}={\rm lann}(a),{\rm
rann}={\rm rann}(a^{\ast}e)}.\] If $a_{\core,f}$ exists, then
\[a_{\core,f}=a^{(1,2)}_{{\rm rprin}=f^{-1}a^{\ast}\mathcal{R},{\rm
lprin}=\mathcal{R}a}=a^{(1,2)}_{{\rm
rprin}=f^{-1}a^{\ast}\mathcal{R},{\rm rann}={\rm
rann}(a)}=a^{(1,2)}_{{\rm lprin}=\mathcal{R}a,{\rm lann}={\rm
lann}(f^{-1}a^{\ast})}=a^{(1,2)}_{{\rm lann}={\rm
lann}(f^{-1}a^{\ast}),{\rm rann}={\rm rann}(a)}.\] As a consequence
of Theorems~\ref{T 12I xR x0 expression}-\ref{T 12I 0x x0
expression}, $a^{\core,e}$ exists and $x=a^{\core,e}$ if and only if
one of the conditions (\ref{E eCI 1p})-(\ref{E eCI 4p}) holds and
one of the conditions (\ref{E eCI 1a})-(\ref{E eCI 4a}) holds, where
the conditions are
\begin{flushleft}
\setlength{\abovedisplayskip}{0pt}
\setlength{\belowdisplayskip}{0pt}
\begin{subequations}
\begin{minipage}{.54\textwidth}
\begin{align}
&\varphi_{ax} = \rho_{a\mathcal{R}, {\rm rann}(a^{\ast}e)}, \varphi_{xa} = \rho_{a\mathcal{R}, {\rm rann}(a)}.\label{E eCI 1p}\\
&\,_{ax}\varphi = \rho_{\mathcal{R}a^{\ast}e,{\rm lann}(a)}, \,_{xa}\varphi = \rho_{\mathcal{R}a,{\rm lann}(a)}.\\
&\varphi_{ax} = \rho_{a\mathcal{R}, {\rm rann}(a^{\ast}e)}, \,_{xa}\varphi = \rho_{\mathcal{R}a,{\rm lann}(a)}.\\
&\,_{ax}\varphi =
\rho_{\mathcal{R}a^{\ast}e,{\rm lann}(a)}, \varphi_{xa} =
\rho_{a\mathcal{R}, {\rm rann}(a)}.\label{E eCI 4p}
\end{align}
\end{minipage}
\begin{minipage}{.01\textwidth}
\end{minipage}
\begin{minipage}{.45\textwidth}
\begin{align}
&x\mathcal{R} \subseteq a\mathcal{R}.\label{E eCI 1a}\\
&{\rm lann}(a) \subseteq {\rm lann}(x).\\
&\mathcal{R}x \subseteq \mathcal{R}a^{\ast}e.\\
&{\rm rann}(a^{\ast}e) \subseteq {\rm rann}(x).\label{E eCI 4a}
\end{align}
\end{minipage}
\end{subequations}
\end{flushleft}
Similarly, $a_{\core,f}$ exists and $x=a_{\core,f}$ if and only if
one of the conditions (\ref{E fDCI 1p})-(\ref{E fDCI 4p}) holds and
one of the conditions (\ref{E fDCI 1a})-(\ref{E fDCI 4a}) holds,
where the conditions are
\begin{flushleft}
\setlength{\abovedisplayskip}{0pt}
\setlength{\belowdisplayskip}{0pt}
\begin{subequations}
\begin{minipage}{0.54\textwidth}
\begin{align}
&\varphi_{ax} = \rho_{a\mathcal{R}, {\rm rann}(a)}, \varphi_{xa} = \rho_{f^{-1}a^{\ast}\mathcal{R}, {\rm rann}(a)},\label{E fDCI 1p}\\
&\,_{ax}\varphi = \rho_{\mathcal{R}a,{\rm lann}(a)}, \,_{xa}\varphi = \rho_{\mathcal{R}a,{\rm lann}(f^{-1}a^{\ast})},\\
&\varphi_{ax} = \rho_{a\mathcal{R}, {\rm rann}(a)}, \,_{xa}\varphi = \rho_{\mathcal{R}a,{\rm lann}(f^{-1}a^{\ast})},\\
&\,_{ax}\varphi =
\rho_{\mathcal{R}a,{\rm lann}(a)}, \varphi_{xa} =
\rho_{f^{-1}a^{\ast}\mathcal{R}, {\rm rann}(a)},\label{E fDCI 4p}
\end{align}
\end{minipage}
\begin{minipage}{0.01\textwidth}
\end{minipage}
\begin{minipage}{0.45\textwidth}
\begin{align}
&x\mathcal{R} \subseteq f^{-1}a^{\ast}\mathcal{R},\label{E fDCI 1a}\\
&{\rm lann}(f^{-1}a^{\ast}) \subseteq {\rm lann}(x),\\
&\mathcal{R}x \subseteq \mathcal{R}a,\\
&{\rm rann}(a) \subseteq {\rm rann}(x).\label{E fDCI 4a}
\end{align}
\end{minipage}
\end{subequations}
\end{flushleft}

\subsubsection{The $w$-core and the dual $v$-core inverses}

Let $\mathcal{R}$ be a $\ast$-ring and $a, x, w, v \in \mathcal{R}$.
Then $x$ is the \emph{$w$-core inverse} of $a$ if
$(awx)^{\ast}=awx$, $xawa=a$, and $awx^{2}=x$. Similarly, $x$ is the
\emph{ dual $v$-core inverse} of $a$ if $(xva)^{\ast}=xva$, $avax=a$,
and $x^{2}va=x$. The $w$-core and the dual $v$-core inverses of $a$
are unique if they exist, and are denoted by $a^{\core,w}$ and
$a_{\core,v}$, respectively. We have $a^{\core}_{1}=a^{\core}$ and
$a_{\core,1}=a_{\core}$. We refer the reader to, e.g.
\cite{Jin-Zhu-Wu (2023), Zhu-Wu-Chen (2023)} for more details. We
note that the equivalence (1) $\Leftrightarrow$ (2) of the next
proposition was given in \cite[Theorem 2.10]{Zhu-Wu-Chen (2023)}.
Here, we present a proof based on Theorem~\ref{T 12I projectors}.
\begin{proposition}\label{P wCI CI}
Let $\mathcal{R}$ be a $\ast$-ring and $a, w, x \in \mathcal{R}$.
The following assertions are equivalent:
\begin{enumerate}
  \item $x=a^{\core,w}$.
  \item $x=(aw)^{\core}$ and $a\mathcal{R} \subseteq aw\mathcal{R}$.
  \item $x=(aw)^{\core}$ and ${\rm lann}(aw) \subseteq {\rm lann}(a)$.
\end{enumerate}
\end{proposition}
\begin{proof}
(1) $\Rightarrow$ (2): If $x=a^{\core,w}$, then $x=(aw)^{\core}$
and, by Theorem~\ref{T 12I projectors},
$a=xawa=\rho_{aw\mathcal{R},{\rm rann}(aw)}(a)$. This last equality
implies that $a\mathcal{R} \subseteq aw\mathcal{R}$.

(2) $\Rightarrow$ (3) is immediate.

(3) $\Rightarrow$ (1): If $x=(aw)^{\core}$ and ${\rm lann}(aw)
\subseteq {\rm lann}(a)$, then $(awx)^{\ast}=awx$, $awx^{2}=x$ and,
by Theorem~\ref{T 12I projectors}, $xawa=\rho_{\mathcal{R}aw,{\rm
lann}(aw)}(1)a=a$.
\end{proof}

For the dual $v$-core inverse we analogously have:
\begin{proposition}\label{P vDCI DCI}
Let $\mathcal{R}$ be a $\ast$-ring and $a, v, x \in \mathcal{R}$.
The following assertions are equivalent:
\begin{enumerate}
  \item $x=a_{\core,v}$.
  \item $x=(va)_{\core}$ and $\mathcal{R}a \subseteq \mathcal{R}va$.
  \item $x=(va)_{\core}$ and ${\rm rann}(va) \subseteq {\rm rann}(a)$.
\end{enumerate}
\end{proposition}
By Proposition~\ref{P wCI CI}, if $a^{\core,w}$ exists, then
\[a^{\core,w}\!=\!a^{(1,2)}_{{\rm rprin}=aw\mathcal{R},{\rm rann}={\rm
rann}((aw)^{\ast})}\!=\!a^{(1,2)}_{{\rm
lprin}=\mathcal{R}(aw)^{\ast},{\rm lann}={\rm
lann}(aw)}\!=\!a^{(1,2)}_{{\rm rprin}=aw\mathcal{R},{\rm
lprin}=\mathcal{R}(aw)^{\ast}}\!=\!a^{(1,2)}_{{\rm lann}={\rm
lann}(aw),{\rm rann}={\rm rann}((aw)^{\ast})}.\] If $a_{\core,v}$ exists, then
\[a_{\core,v}\!=\!a^{(1,2)}_{{\rm rprin}=(va)^{\ast}\mathcal{R},{\rm
rann}={\rm rann}(va)}\!=\!a^{(1,2)}_{{\rm lprin}=\mathcal{R}va,{\rm
lann}={\rm lann}((va)^{\ast})}\!=\!a^{(1,2)}_{{\rm
rprin}=(va)^{\ast}\mathcal{R},{\rm
lprin}=\mathcal{R}va}=a^{(1,2)}_{{\rm lann}={\rm
lann}((va)^{\ast}),{\rm rann}={\rm rann}(va)}.\]

Let $b=aw$ and $c=va$. By Theorems~\ref{T 12I xR x0
expression}-\ref{T 12I 0x x0 expression} and Proposition~\ref{P wCI
CI} (resp. Proposition~\ref{P vDCI DCI}), $a^{\core,w}$ exists and
$x=a^{\core,w}$ if and only if one of the conditions (\ref{E wCI
1p})-(\ref{E wCI 4p}) holds, one of the conditions (\ref{E wCI
1a})-(\ref{E wCI 4a}) holds, and one of the conditions (\ref{E wCI
1b})-(\ref{E wCI 2b}) holds, where the conditions are
\begin{flushleft}
\setlength{\abovedisplayskip}{0pt}
\setlength{\belowdisplayskip}{0pt}
\begin{subequations}
\begin{minipage}{0.42\textwidth}
\begin{align}
&\varphi_{bx} = \rho_{b\mathcal{R},{\rm rann}(b^{\ast})}, \varphi_{xa} = \rho_{b\mathcal{R},{\rm rann}(b)},\label{E wCI 1p}\\
&\,_{bx}\varphi = \rho_{\mathcal{R}b^{\ast},{\rm lann}(b)}, \,_{xa}\varphi = \rho_{\mathcal{R}b,{\rm lann}(b)},\\
&\varphi_{bx} = \rho_{b\mathcal{R}, {\rm rann}(b^{\ast})}, \,_{xa}\varphi = \rho_{\mathcal{R}b,{\rm lann}(b)},\\
&\,_{bx}\varphi = \rho_{\mathcal{R}b^{\ast},{\rm lann}(b)},
\varphi_{xa} = \rho_{b\mathcal{R}, {\rm rann}(b)},\label{E wCI 4p}
\end{align}
\end{minipage}
\begin{minipage}{0.29\textwidth}
\begin{align}
&x\mathcal{R} \subseteq b\mathcal{R},\label{E wCI 1a}\\
&{\rm lann}(b) \subseteq {\rm lann}(x),\\
&\mathcal{R}x \subseteq \mathcal{R}b^{\ast},\\
&{\rm rann}(b^{\ast}) \subseteq {\rm rann}(x),\label{E wCI 4a}
\end{align}
\end{minipage}
\begin{minipage}{0.25\textwidth}
\begin{align}
&\!\!\!\!\!a\mathcal{R} \subseteq b\mathcal{R},\label{E wCI 1b}\\
&\!\!\!\!\!{\rm lann}(b) \subseteq {\rm lann}(a).\label{E wCI 2b}
\end{align}
\end{minipage}
\end{subequations}
\end{flushleft}
We also have $a_{\core,v}$ exists and $x=a_{\core,v}$ if and only if
one of the conditions (\ref{E vDCI 1p})-(\ref{E vDCI 4p}) holds, one
of the conditions (\ref{E vDCI 1a})-(\ref{E vDCI 4a}) holds, and one
of the conditions (\ref{E vDCI 1b})-(\ref{E vDCI 2b}) holds, where
the conditions are
\begin{flushleft}
\setlength{\abovedisplayskip}{0pt}
\setlength{\belowdisplayskip}{0pt}
\begin{subequations}
\begin{minipage}{0.43\textwidth}
\begin{align}
&\varphi_{cx} = \rho_{c\mathcal{R}, {\rm rann}(c)}, \varphi_{xc} = \rho_{c^{\ast}\mathcal{R}, {\rm rann}(c)},\label{E vDCI 1p}\\
&\,_{cx}\varphi = \rho_{\mathcal{R}c,{\rm lann}(c)}, \,_{xc}\varphi = \rho_{\mathcal{R}c,{\rm lann}(c^{\ast})},\\
&\varphi_{cx} = \rho_{c\mathcal{R}, {\rm rann}(c)}, \,_{xc}\varphi = \rho_{\mathcal{R}c,{\rm lann}(c^{\ast})},\\
&\,_{cx}\varphi = \rho_{\mathcal{R}c,{\rm lann}(c)}, \varphi_{xc} =
\rho_{c^{\ast}\mathcal{R}, {\rm rann}(c)},\label{E vDCI 4p}
\end{align}
\end{minipage}
\begin{minipage}{0.28\textwidth}
\begin{align}
&x\mathcal{R} \subseteq c^{\ast}\mathcal{R},\label{E vDCI 1a}\\
&{\rm lann}(c^{\ast}) \subseteq {\rm lann}(x),\\
&\mathcal{R}x \subseteq \mathcal{R}c,\\
&{\rm rann}(c) \subseteq {\rm rann}(x),\label{E vDCI 4a}
\end{align}
\end{minipage}
\begin{minipage}{0.27\textwidth}
\begin{align}
&\!\!\!\!\!\mathcal{R}a \subseteq \mathcal{R}c,\label{E vDCI 1b}\\
&\!\!\!\!\!{\rm rann}(c) \subseteq {\rm rann}(a).\label{E vDCI 2b}
\end{align}
\end{minipage}
\end{subequations}
\end{flushleft}
Let $w=1$. The conditions $\varphi_{ax} = \rho_{a\mathcal{R},{\rm
rann}(a^{\ast})}$, $\varphi_{xa} = \rho_{a\mathcal{R},{\rm
rann}(a)}$ and $x\mathcal{R} \subseteq a\mathcal{R}$ are stronger
than the conditions (\ref{E D CI matrix}) of the definition of the
core inverse for finite complex matrices. Let $\mathcal{H}$ be an
arbitrary Hilbert space and $L(\mathcal{H})$ be the ring of all
bounded linear operators from $\mathcal{H}$ to $\mathcal{H}$. In
\cite[Remark 3.1]{Rakic-Dincic-Djordjevic (2014a)}, it is shown that
the conditions $A, X \in L(\mathcal{H})$, $AX=P_{{\rm R}(A)}$ and
${\rm R}(X) \subseteq {\rm R}(A)$, do not imply that $X=A^{\core}$.
Similar considerations can be made for the dual core inverse.

\subsubsection{The right $w$-core and left dual $v$-core inverses}

The \emph{right $w$-core inverse} $x$ of $a$ is defined by the
equations $awxa=a$, $(awx)^{\ast}=awx$, and $awx^{2}=x$. We
analogously have the \emph{left dual $v$-core inverse} $x$ of $a$
defined by the equations $axva=a$, $(xva)^{\ast}=xva$, and
$x^{2}va=x$. These generalized inverses were defined in
\cite{Zhu-Wu-Mosic (2023)}. From Theorem~\ref{T 137 orthogonal
projector} we obtain:
\begin{proposition}\label{P rwCI 137I}
Let $\mathcal{R}$ be a $\ast$-ring and $a, w, x \in \mathcal{R}$.
The following assertions are equivalent:
\begin{enumerate}
  \item $x$ is a right $w$-core inverse of $a$.
  \item $x\in(aw)\{1,3,7\}$ and $a\mathcal{R} \subseteq aw\mathcal{R}$.
  \item $x\in(aw)\{1,3,7\}$ and ${\rm lann}(aw) \subseteq {\rm lann}(a)$.
\end{enumerate}
\end{proposition}
We note that the equivalence (1) $\Leftrightarrow$ (2) of
Proposition~\ref{P rwCI 137I} appears in \cite[Theorem
2.14]{Zhu-Wu-Mosic (2023)}. Analogously, using Theorem~\ref{T 149
orthogonal projector} we get:
\begin{proposition}\label{P lwDCI 149I}
Let $\mathcal{R}$ be a $\ast$-ring and $a, v, x \in \mathcal{R}$.
The following assertions are equivalent:
\begin{enumerate}
  \item $x$ is a left dual $v$-core inverse of $a$.
  \item $x\in(va)\{1,4,9\}$ and $\mathcal{R}a \subseteq \mathcal{R}va$.
  \item $x\in(va)\{1,4,9\}$ and ${\rm rann}(va) \subseteq {\rm rann}(a)$.
\end{enumerate}
\end{proposition}
From Proposition~\ref{P rwCI 137I} and Theorem~\ref{T 137 orthogonal
projector}, $x$ is a right $w$-core inverse of $a$ if and only if
one of the conditions (\ref{E rwCI 1p})-(\ref{E rwCI 2p}) holds and
one of the conditions (\ref{E rwCI 1a})-(\ref{E rwCI 2a}) holds,
where the conditions are
\begin{flushleft}
\setlength{\abovedisplayskip}{0pt}
\setlength{\belowdisplayskip}{0pt}
\begin{subequations}
\begin{minipage}{0.54\textwidth}
\begin{align}
&\varphi_{awx} = \rho_{aw\mathcal{R},{\rm rann}((aw)^{\ast})},\,\, x\in aw\mathcal{R},\label{E rwCI 1p}\\
&\,_{ax}\varphi = \rho_{\mathcal{R}a^{\ast},{\rm lann}(a)},\,\, {\rm
lann}(aw)\subseteq{\rm lann}(x),\label{E rwCI 2p}
\end{align}
\end{minipage}
\begin{minipage}{0.01\textwidth}
\end{minipage}
\begin{minipage}{0.45\textwidth}
\begin{align}
&a\mathcal{R} \subseteq aw\mathcal{R},\label{E rwCI 1a}\\
&{\rm lann}(aw) \subseteq {\rm lann}(a).\label{E rwCI 2a}
\end{align}
\end{minipage}
\end{subequations}
\end{flushleft}
By Proposition~\ref{P lwDCI 149I} and Theorem~\ref{T 149 orthogonal
projector}, $x$ is a left dual $v$-core inverse of $a$ if and only
if one of the one of the conditions (\ref{E lvDCI 1p})-(\ref{E lvDCI
2p}) holds and one of the conditions (\ref{E lvDCI 1a})-(\ref{E
lvDCI 2a}) holds, where the conditions are
\begin{flushleft}
\setlength{\abovedisplayskip}{0pt}
\setlength{\belowdisplayskip}{0pt}
\begin{subequations}
\begin{minipage}{0.54\textwidth}
\begin{align}
&\varphi_{xva} = \rho_{(va)^{\ast}\mathcal{R},{\rm rann}(va)},\,\, x\in\mathcal{R}va,\label{E lvDCI 1p}\\
&\,_{xa}\varphi = \rho_{\mathcal{R}a,{\rm lann}(a^{\ast})},\,\, {\rm
rann}(va) \subseteq {\rm rann}(x),\label{E lvDCI 2p}
\end{align}
\end{minipage}
\begin{minipage}{0.01\textwidth}
\end{minipage}
\begin{minipage}{0.45\textwidth}
\begin{align}
&\mathcal{R}a \subseteq \mathcal{R}va,\label{E lvDCI 1a}\\
&{\rm rann}(va) \subseteq {\rm rann}(a).\label{E lvDCI 2a}
\end{align}
\end{minipage}
\end{subequations}
\end{flushleft}

\subsection{$(b,c)$-inverses}\label{S bc 2 inverses}

Let $a, x, b, c, d \in \mathcal{R}$. In this section, we consider
the $(b,c)$ inverses defined by Drazin in \cite{Drazin (2012)}:
$a^{(2)}_{{\rm rprin}=b\mathcal{R},{\rm lprin}=\mathcal{R}c}$ is the
\emph{$(b,c)$ inverse} of $a$, $a^{(2)}_{{\rm
rprin}=b\mathcal{R},{\rm rann}={\rm rann}(c)}$ is the \emph{right
hybrid $(b,c)$ inverse} of $a$, $a^{(2)}_{{\rm
rprin}=\mathcal{R}b,{\rm rann}={\rm lann}(c)}$ is the \emph{left
hybrid $(b,c)$ inverse} of $a$, and $a^{(2)}_{{\rm lann}={\rm
lann}(b),{\rm rann}={\rm rann}(c)}$ is the \emph{annihilator $(b,c)$
inverse} of $a$. See, e.g., \cite{Benitez-Boasso (2016), Drazin (2012), Drazin (2019), KantunMontiel (2014), Ke-CvetkovicIlic-Chen-Visnjic (2018), Mary (2011), Mary (2021), Rakic (2017), Zhu (2018)}. Indeed in \cite{Drazin (2012)}, Drazin
calls $x$ the $(b,c)$-inverse of $a$ if $x \in b\mathcal{R}x \cap
x\mathcal{R}c$, $xab=b$, and $cax=c$. In \cite[page 1922]{Drazin
(2012)}, Drazin shows that this definition is equivalent to the
given here for the $(b,c)$-inverse. The $(d,d)$-inverse of $a$ is
the \emph{inverse of $a$ along $d$} defined by Mary in \cite{Mary
(2011)}. Conversely, each $(b,c)$-inverse is a $(d,d)$-inverse
\cite{Mary (2021)}.  In \cite{Drazin (2019)}, the previous inverses
were generalized as follows. Let $v, w \in \mathcal{R}$. Then,
$(vaw)^{(2)}_{{\rm rprin}=b\mathcal{R},{\rm lprin}=\mathcal{R}c}$ is
the \emph{$(w,v)$-weighted $(b,c)$-inverse} of $a$,
$(vaw)^{(2)}_{{\rm rprin}=b\mathcal{R},{\rm rann}={\rm rann}(c)}$ is
the \emph{right hybrid $(w,v)$-weighted $(b,c)$-inverse} of $a$,
$(vaw)^{(2)}_{{\rm rprin}=\mathcal{R}b,{\rm rann}={\rm lann}(c)}$ is
the \emph{left hybrid $(w,v)$-weighted $(b,c)$-inverse} of $a$, and
$(vaw)^{(2)}_{{\rm lann}={\rm lann}(b),{\rm rann}={\rm rann}(c)}$ is
the \emph{annihilator $(w,v)$-weighted $(b,c)$-inverse} of $a$.

Theorem~\ref{T 12 inverse b cab1 c}(1)-(3) below can be seen as a
generalization of \cite[Theorem 2.13]{Ben-Israel-Greville (2003)}.
\begin{theorem}\label{T 12 inverse b cab1 c}
Let $a, b, c \in \mathcal{R}$ be such that $(cab)\{1\} \neq
\emptyset$. Let $(cab)^{(1)} \in (cab)\{1\}$ and $x=b(cab)^{(1)}c$.
Then:
\begin{enumerate}
  \item $x \in a\{1\} \Leftrightarrow \{ab\mathcal{R}=a\mathcal{R}\! \text{ and } {\rm rann}(cab)={\rm rann}(ab)\} \Leftrightarrow \{ab\mathcal{R}=a\mathcal{R}\! \text{ and } \!\mathcal{R}cab=\mathcal{R}ab\}$.
  \item $\{x \in a\{2\} \text{ and } x\mathcal{R}=b\mathcal{R}\} \Leftrightarrow {\rm rann}(cab)={\rm rann}(b) \Leftrightarrow \mathcal{R}cab=\mathcal{R}b$.
  \item $\{x \in a\{2\} \text{ and } {\rm rann}(x)={\rm rann}(c)\} \Leftrightarrow cab\mathcal{R}=c\mathcal{R}  \Leftrightarrow {\rm lann}(cab)={\rm lann}(c)$.
  \item $\{x \in a\{2\} \text{ and } \mathcal{R}x=\mathcal{R}c\} \Leftrightarrow {\rm lann}(cab)={\rm lann}(c) \Leftrightarrow
  cab\mathcal{R}=c\mathcal{R}$.
  \item $\{x \in a\{2\} \text{ and } {\rm lann}(x)={\rm lann}(b)\} \Leftrightarrow \mathcal{R}cab=\mathcal{R}b \Leftrightarrow
{\rm rann}(cab)={\rm rann}(b)$.
\end{enumerate}
\end{theorem}
\begin{proof}
We first observe that by Lemma~\ref{L ab ab1 a y b ab1 ab},
$cab\mathcal{R}=ca\mathcal{R}$ (resp. $cab\mathcal{R}=c\mathcal{R}$)
if and only if ${\rm lann}(cab)={\rm lann}(ca)$ (resp. ${\rm
lann}(cab)={\rm lann}(c)$), and $\mathcal{R}cab=\mathcal{R}ab$
(resp. $\mathcal{R}cab=\mathcal{R}b$) if and only if ${\rm
rann}(cab)={\rm rann}(ab)$ (resp. ${\rm rann}(cab)={\rm rann}(b)$).

(1): If $x \in a\{1\}$, then $a=ab(cab)^{(1)}ca$ and
$ab=ab(cab)^{(1)}(cab)$. By the last equalities and Lemma~\ref{L ab
ab1 a y b ab1 ab}(2), $ab\mathcal{R}=a\mathcal{R}$ and ${\rm
rann}(cab)={\rm rann}(ab)$. Conversely, if
$ab\mathcal{R}=a\mathcal{R}$ and ${\rm rann}(cab)={\rm rann}(ab)$,
then there exists $r \in \mathcal{R}$ such that $a=abr$ and, using
again Lemma~\ref{L ab ab1 a y b ab1 ab}(2),
$axa=a(b(cab)^{(1)}c)a=(ab)(cab)^{(1)}(cab)r=abr=a$. Hence, $x \in
a\{1\}$.

(2): We have, $x \in a\{2\}$ and $x\mathcal{R}=b\mathcal{R}$ if and
only if $x=xax$ and there exists $r \in \mathcal{R}$ such that
$b=xr$. By the expression of $x$, we obtain
$b=xr=xaxr=b(cab)^{(1)}cab$. By Lemma~\ref{L ab ab1 a y b ab1
ab}(2), this last equality is equivalent to ${\rm rann}(cab)={\rm
rann}(b)$. Conversely, by Lemma~\ref{L ab ab1 a y b ab1 ab}(2), if
${\rm rann}(cab)={\rm rann}(b)$, then $xab=b(cab)^{(1)}cab=b$. From
the last equality, $x=b(cab)^{(1)}c=xab(cab)^{(1)}c=xax$ and
$x\mathcal{R}=b\mathcal{R}$.

(3): By Theorem~\ref{T 2I projectors}, if $x \in a\{2\}$ and ${\rm
rann}(x)={\rm rann}(c)$, then $\mathcal{R}=
ab(cab)^{(1)}c\mathcal{R} \oplus {\rm rann}(c)$. Therefore,
$c\mathcal{R} \subseteq cab(cab)^{(1)}c\mathcal{R} \subseteq cab
\mathcal{R}$. This implies that $c\mathcal{R}= cab \mathcal{R}$.
Conversely, using Lemma~\ref{L ab ab1 a y b ab1 ab}(1), we obtain
that if $cab\mathcal{R}=c\mathcal{R}$, then $cax=cab(cab)^{(1)}c =
c$. From here, ${\rm rann}(x)={\rm rann}(c)$ and
$xax=b(cab)^{(1)}cax=b(cab)^{(1)}c=x$.

The proofs of (4) and (5) are similar to the proofs of (2) and (3),
respectively.
\end{proof}
We get the next result for right hybrid $(b,c)$ inverses.
\begin{theorem}\label{T 2I bR c0 existence}
Let $a, b, c \in \mathcal{R}$. If any of the conditions
\begin{enumerate}
  \item ${\rm rann}(ab)=\{0\}$, $c\mathcal{R}=\mathcal{R}$ and
  $\mathcal{R}=ab\mathcal{R}\oplus{\rm rann}(c)$, or
  \item ${\rm rann}(b)=\{0\}$, $ca\mathcal{R}=\mathcal{R}$ and
  $\mathcal{R}=b\mathcal{R}\oplus\varphi_{a}^{-1}({\rm rann}(c))$
\end{enumerate}
holds, then $cab \in \mathcal{R}^{-1}$ and
$b(cab)^{-1}c=a^{(2)}_{{\rm rprin}=b\mathcal{R},{\rm rann}={\rm
rann}(c)}$.
\end{theorem}
\begin{proof}
We only prove (1) since the proof of (2) is similar. Assume that
${\rm rann}(ab)=\{0\}$, $c\mathcal{R}=\mathcal{R}$ and
  $\mathcal{R}=ab\mathcal{R}\oplus{\rm rann}(c)$. Then $cab\mathcal{R}=c\mathcal{R}=\mathcal{R}$. Let $r
\in \mathcal{R}$. If $cabr=0$, then $abr \in ab\mathcal{R} \cap {\rm
rann}(c)$. Hence, $abr=0$. This shows that ${\rm rann}(cab)={\rm
rann}(ab)={\rm rann}(b)=\{0\}$. By Lemma~\ref{L invertible}, $cab
\in \mathcal{R}^{-1}$, and by Theorem~\ref{T 12 inverse b cab1
c}(2)(3), if $x=b(cab)^{-1}c$, then $x\in a\{2\}$,
$x\mathcal{R}=b\mathcal{R}$ and ${\rm rann}(x)={\rm rann}(c)$.
\end{proof}
Theorems~\ref{T 2I xR x0 unique} and \ref{T 2I bR c0 existence} can
be seen as generalizations of \cite[Theorem
2.14]{Ben-Israel-Greville (2003)}. Let $\mathbb{C}^{m \times n}_{r}$
denote the class of complex matrices of rank $r$. Let $A \in
\mathbb{C}^{m \times n}_{r}$, $s \leq r$, $U \in \mathbb{C}^{n
\times s}_{s}$ and $V \in \mathbb{C}^{s \times n}$. Then, ${\rm
rank}(V)=s \Leftrightarrow {\rm R}(V)=\mathbb{C}^{s}$ and equality
\cite[(2.62)]{Ben-Israel-Greville (2003)} is equivalent to ${\rm
N}(AU)=\{0\}$. This shows that the hypotheses ${\rm rann}(ab)=\{0\}$
and $c\mathcal{R}=\mathcal{R}$ in Theorem~\ref{T 2I bR c0 existence}
are natural generalizations of the hypotheses of \cite[Theorem
2.14]{Ben-Israel-Greville (2003)}.

Using Lemma~\ref{L invertible} and Theorem~\ref{T 12 inverse b cab1
c}(4)(5), we analogously obtain:
\begin{theorem}\label{T 2I Rc 0b existence}
Let $a, b, c \in \mathcal{R}$. If any of the conditions
\begin{enumerate}
  \item ${\rm lann}(ca)=\{0\}$, $\mathcal{R}b=\mathcal{R}$ and
  $\mathcal{R}=\mathcal{R}ca \oplus {\rm lann}(b)$, or
  \item ${\rm lann}(c)=\{0\}$, $\mathcal{R}ab=\mathcal{R}$ and
  $\mathcal{R}=\mathcal{R}c\oplus\,_{a}\varphi^{-1}({\rm lann}(b))$
\end{enumerate}
holds, then $cab \in \mathcal{R}^{-1}$ and
$b(cab)^{-1}c=a^{(2)}_{{\rm lprin}=\mathcal{R}c,{\rm lann}={\rm
lann}(b)}$.
\end{theorem}

The next theorem gives relations between the different types of $(b,c)$-inverses.
\begin{theorem}\label{T 2I bc igualdad} Let $a, b, c, x \in
\mathcal{R}$. The following assertions are equivalent:
\begin{enumerate}
  \item $x=a^{(2)}_{{\rm rprin}=b\mathcal{R},{\rm rann}={\rm rann}(c)}$ and $x \in \mathcal{R}c$ (or $c\{1\}\neq\emptyset$).
  \item $x=a^{(2)}_{{\rm rprin}=b\mathcal{R},{\rm rann}={\rm rann}(c)}$ and $(cab)\{1\}\neq\emptyset$.
  \item $x=a^{(2)}_{{\rm lprin}=\mathcal{R}c,{\rm lann}={\rm lann}(b)}$ and $x \in b\mathcal{R}$ (or $b\{1\}\neq\emptyset$).
  \item $x=a^{(2)}_{{\rm lprin}=\mathcal{R}c,{\rm lann}={\rm lann}(b)}$ and $(cab)\{1\}\neq\emptyset$.
  \item $x=a^{(2)}_{{\rm rprin}=b\mathcal{R},{\rm lprin}=\mathcal{R}c}$.
  \item $x=a^{(2)}_{{\rm lann}={\rm lann}(b),{\rm rann}={\rm rann}(c)}$ and $x \in b\mathcal{R}$ (or $b\{1\}\neq\emptyset$) and $x \in \mathcal{R}c$ (or $c\{1\}\neq\emptyset$).
  \item $x=a^{(2)}_{{\rm lann}={\rm lann}(b),{\rm rann}={\rm rann}(c)}$ and $(cab)\{1\}\neq\emptyset$.
  \item $(cab)\{1\}\neq\emptyset$, $\mathcal{R}cab=\mathcal{R}b$ (or
  ${\rm rann}(cab)={\rm rann}(b)$), $cab\mathcal{R}=c\mathcal{R}$ (or ${\rm lann}(cab)={\rm lann}(c)$), and $x=b(cab)^{(1)}c$ for any $(cab)^{(1)} \in (cab)\{1\}$.
\end{enumerate}
\end{theorem}
\begin{proof}
If $x\in\{2\}$, then $a \in x\{1\}$. Hence, from Lemma~\ref{L
a1neqemptyset}, in (1) and (6): $x \in \mathcal{R}c \Leftrightarrow
c\{1\}\neq\emptyset$, and in (3) and (6): $x \in
b\mathcal{R}\Leftrightarrow b\{1\}\neq\emptyset$.

As a consequence of \cite[Theorem 2.4]{Zhu (2018)} (or
Theorem~\ref{T 2I xR x0 existencia}) and \cite[Corollary 2.6]{Zhu
(2018)}, we get (1)  $\Leftrightarrow$ (2) $\Leftrightarrow$ (5).
From Theorem~\ref{T 2I Rx 0x existencia} and \cite[Corollary
2.6]{Zhu (2018)}, we obtain (3) $\Leftrightarrow$ (4)
$\Leftrightarrow$ (5).

(5) $\Rightarrow$ (6): It is immediate.

(6) $\Rightarrow$ (5): Suppose that $x=a^{(2)}_{{\rm lann}={\rm
lann}(b),{\rm rann}={\rm rann}(c)}$, $x \in b\mathcal{R}$, and $x
\in \mathcal{R}c$. By Theorem~\ref{T 2I 0x x0 existencia}(6),
$c=cax$ and $b=xab$. Hence, $\mathcal{R}c \subseteq \mathcal{R}x$
and $b\mathcal{R} \subseteq x\mathcal{R}$. We conclude that
$x=a^{(2)}_{{\rm rprin}=b\mathcal{R},{\rm lprin}=\mathcal{R}c}$.

(5) $\Rightarrow$ (8): Suppose that $x=a^{(2)}_{{\rm
rprin}=b\mathcal{R},{\rm lprin}=\mathcal{R}c}$. Then
$x=a^{(2)}_{{\rm rprin}=b\mathcal{R},{\rm lprin}={\rm rann}(c)}$.
Thus, (8) follows from \cite[Proposition 2.5]{Zhu (2018)} and
Theorem~\ref{T 12 inverse b cab1 c}(2)(4).

(8) $\Rightarrow$ (5): It follows from Theorem~\ref{T 12 inverse b
cab1 c}(2)(4).

(7) $\Leftrightarrow$ (8): It follows from Theorem~\ref{T 12 inverse
b cab1 c}(3)(5).
\end{proof}
From Theorems~\ref{T 2I xR x0
existencia}(3), \ref{T 2I Rx 0x existencia}(3) and \ref{T 2I bc igualdad}(5)(1)(3), we get
\cite[Proposition 2.7]{Ke-CvetkovicIlic-Chen-Visnjic (2018)}.

\subsection{The $(p,q)$ inverse}

We now consider $\{2\}$-inverses defined using prefixed idempotent
elements in $\mathcal{R}$. Let $a \in \mathcal{R}$ and $p, q \in
\mathcal{R}^{\bullet}$. Then $a^{(2)}_{{\rm rprin}=p\mathcal{R},{\rm
rann}=q\mathcal{R}}$ is the \emph{image-kernel $(p,q)$ inverse} of
$a$ (see \cite[Definition 3.1]{KantunMontiel (2014)}). The
image-kernel $(p,q)$ inverse is the Cao-Xue $(p,q,l)$ inverse (see
\cite[Definition 2.10]{Cao-Xue (2013)}).

Let $a \in \mathcal{R}$ and $p \in \mathcal{R}^{\bullet}$. Then $a$
is called \emph{Bott-Duffin invertible} if $1-p+ap \in
\mathcal{R}^{-1}$, and in this case, the \emph{Bott-Duffin $p$
inverse} of $a$ is $p(1-p+ap)^{-1}$ (see \cite[Definitions
(c)]{Bott-Duffin (1953)}). If $x \in \mathcal{R}$ is such that
$x=px=xq$, $xap=p$, and $qax=q$, then $x$ is called the
\emph{Bott-Duffin $(p,q)$ inverse} of $a$ (see \cite[Definition
3.2]{Drazin (2012)}). The Bott-Duffin $(p,p)$ inverse of $a$ is the
Bott-Duffin $p$ inverse of $a$ (see \cite[Proposition 3.1]{Drazin
(2012)}). By \cite[Proposition 3.4]{KantunMontiel (2014)},
$x=a^{(2)}_{{\rm rprin}=p\mathcal{R},{\rm rann}=q\mathcal{R}}$ if
and only if $x$ is the Bott-Duffin $(p,1-q)$ inverse of $a$.

Let $a \in \mathcal{R}$ and $p, q \in \mathcal{R}^{\bullet}$. Then
$x \in \mathcal{R}$ is called the \emph{Djordjevi\'{c}-Wei $(p,q)$
inverse} of $a$ if $x \in a\{2\}$, $xa=p$ and $ax=1-q$ (see
\cite[Definition 2.1]{Djordjevic-Wei (2005)}). We note that
$a^{(2)}_{p,q}=a^{(2)}_{{\rm rprin}=p\mathcal{R},{\rm
rann}=q\mathcal{R}}$ with ${\rm rann}(a^{(2)}_{p,q}a)={\rm rann}(p)$
and $aa^{(2)}_{p,q}\mathcal{R}={\rm rann}(q)$. We also have,
$a^{(2)}_{p,q} = a^{(2)}_{{\rm lprin}={\rm lann}(q),{\rm lann}={\rm
lann}(p)}$ with $\mathcal{R}a^{(2)}_{p,q}a=\mathcal{R}p$ and ${\rm
lann}(aa^{(2)}_{p,q})=\mathcal{R}q$. Since $(1-q)a=(1-q)ap
\Leftrightarrow a(1-p)=qa(1-p) \Rightarrow a(1-p)\mathcal{R}
\subseteq q\mathcal{R}$, the conditions of part (2) of the following
theorem are weaker than the conditions of \cite[Theorem
2.1(2)]{Djordjevic-Wei (2005)} that include the equality $px=x$. The
conditions of part (4) are with inclusions instead of with
equalities as in \cite[Theorem 2.4(2)]{Cao-Xue (2013)} in a complex
Banach algebra.
\begin{theorem}\label{T GI 2uv}
Let $a \in \mathcal{R}$ and $p, q \in \mathcal{R}^{\bullet}$. Then
the following statements are equivalent:
\begin{enumerate}
  \item $x \in \mathcal{R}$ is the
Djordjevi\'{c}-Wei $(p,q)$ inverse of $a$.
  \item $a(1-p)\mathcal{R} \subseteq q\mathcal{R}$, $xap=p$, $1-q=ax$, and $xq=0$.
  \item $\mathcal{R}qa \subseteq \mathcal{R}(1-p)$, $p=xa$, $px=x$, and $(1-q)ax=1-q$.
  \item $x \in a\{2\}$, $xa\mathcal{R} \subseteq
  p\mathcal{R}, {\rm rann}(xa) \subseteq
  {\rm rann}(p)$ (or $p\mathcal{R} \subseteq xa\mathcal{R}, {\rm rann}(p) \subseteq {\rm rann}(xa)$)
  and $ax\mathcal{R} \subseteq {\rm rann}(q), {\rm rann}(ax) \subseteq
  q\mathcal{R}$ (or ${\rm rann}(q) \subseteq ax\mathcal{R}, q\mathcal{R} \subseteq {\rm rann}(ax)$).
\end{enumerate}
If $a^{(2)}_{p,q}$ exists, then ${\rm
rann}(p)=\varphi_{a}^{-1}(q\mathcal{R})$ and $\mathcal{R}q =
\,_{a}\varphi^{-1}({\rm lann}(p))$.
\end{theorem}
\begin{proof}
(1) $\Rightarrow$ (2)(3): It follows from the definition of
$a^{(2)}_{p,q}$ that $xap=p$, $1-q=ax$, $xq=0$, $p=xa$, $px=x$, and
$(1-q)ax=1-q$.

By Theorems~\ref{T 2I xR x0 unique} and \ref{T 2I Rx 0x
unique}, ${\rm rann}(p)=(1-p)\mathcal{R}={\rm
rann}(xa)=\varphi_{a}^{-1}(q\mathcal{R})$ and $\mathcal{R}q = {\rm
lann}(1-q) = \,_{a}\varphi^{-1}(\mathcal{R}(1-p))$. Then
$a(1-p)\mathcal{R} \subseteq q\mathcal{R}$ and $\mathcal{R}qa
\subseteq \mathcal{R}(1-p)$.

(2) $\Rightarrow$ (1): Assume that $a(1-p)\mathcal{R} \subseteq
q\mathcal{R}$ and there exists $x \in \mathcal{R}$ such that
$xap=p$, $1-q=ax$, and $xq=0$. Then $xax=x$ and $xa(1-p)=0$. Since
$xap=p$ and $xa(1-p)=0$, we have $xa=p$.

The proof of (3) $\Rightarrow$ (1) is similar to the proof of (2)
$\Rightarrow$ (1) and the proof of (1) $\Rightarrow$ (4) is
immediate.

If $x$ is a $\{2\}$-inverse of $a$, then $ax, xa \in
\mathcal{R}^{\bullet}$. Hence, (4) $\Rightarrow$ (1) follows from
Lemma~\ref{L igualdad idempotentes}(3).
\end{proof}
In \cite{KantunMontiel (2014)}, it is noted that if $x=a^{(2)}_{{\rm
rprin}=p\mathcal{R},{\rm rann}=q\mathcal{R}}$, then $x$ is the
Djordjevi\'{c}-Wei $(pxa,q(1-ax))$ inverse of $a$. Using Theorems~\ref{T 2I xR x0 unique} and \ref{T 2I Rx 0x unique},
we obtain the following proposition.
\begin{proposition}
Let $a \in \mathcal{R}$ and $p, q \in \mathcal{R}^{\bullet}$. Then
the following statements are equivalent:
\begin{enumerate}
  \item $x$ is the Djordjevi\'{c}-Wei
$(p,q)$ inverse of $a$.
  \item $x=a^{(2)}_{{\rm rprin}=p\mathcal{R},{\rm rann}=q\mathcal{R}}$, $ap\mathcal{R}=(1-q)\mathcal{R}$, and $\varphi_{a}^{-1}(q\mathcal{R})=(1-p)\mathcal{R}$.
  \item $x=a^{(2)}_{{\rm lprin}={\rm lann}(q),{\rm lann}={\rm lann}(p)}$, $\mathcal{R}(1-q)a=\mathcal{R}p$, and $\,_{a}\varphi^{-1}(\mathcal{R}(1-p))=\mathcal{R}q$.
\end{enumerate}
\end{proposition}
Let $a \in \mathcal{R}$ and $p, q \in \mathcal{R}^{\bullet}$. If
$a^{(2)}_{p,q}$ exists and $a^{(2)}_{p,q}\in a\{1\}$, then
$a^{(2)}_{p,q}$ is the \emph{Djordjevi\'{c}-Wei $(p,q)$-reflexive
generalized inverse} of $a$ and it is denoted by $a^{(1,2)}_{p,q}$
(see \cite[page 3054]{Djordjevic-Wei (2005)}).
\begin{example}
If $a \in \mathcal{R}^{D}$ with ${\rm ind}(a) \leq l$, then
$a^{D}=a^{(2)}_{aa^{D},1-aa^{D}}$. Let $\mathcal{R}$ be a
$\ast$-ring. If $a \in \mathcal{R}^{\dag}$, then
$a^{\dag}=a^{(2)}_{a^{\dag}a,1-aa^{\dag}}=a^{(1,2)}_{a^{\dag}a,1-aa^{\dag}}$.
If $a \in \mathcal{R}^{\core}$, then
$a^{\core}=a^{(2)}_{a^{\core}a,1-aa^{\core}}=a^{(1,2)}_{a^{\core}a,1-aa^{\core}}$.
If $a \in \mathcal{R}_{\core}$, then
$a_{\core}=a^{(2)}_{a_{\core}a,1-aa_{\core}}=a^{(1,2)}_{a_{\core}a,1-aa_{\core}}$.
\end{example}
The equivalence (1) $\Leftrightarrow$ (2) of the following
proposition coincides with \cite[Proposition 3.1]{Cao-Xue (2013)}
for complex Banach algebras.
\begin{proposition}
Let $a \in \mathcal{R}$. Then the following statements are
equivalent:
\begin{enumerate}
  \item $a\{1,2\} \neq \emptyset$.
  \item There exist $p, q \in
\mathcal{R}^{\bullet}$ such that ${\rm rann}(a)={\rm rann}(p)$ and
$a\mathcal{R}=q\mathcal{R}$.
  \item There exist $p, q \in
\mathcal{R}^{\bullet}$ such that $\mathcal{R}a=\mathcal{R}p$ and
${\rm lann}(a)={\rm lann}(q)$.
\item There exist $p, q \in
\mathcal{R}^{\bullet}$ such that $\mathcal{R}a=\mathcal{R}p$ and
$a\mathcal{R}=q\mathcal{R}$.
\end{enumerate}
\end{proposition}
\begin{proof}
(1) $\Rightarrow$ (2)(3): Let $x \in a\{1,2\}$. Setting $p=xa$ and
$q=ax$, these implications follow from Theorem~\ref{T 12I
projectors}.

(2) $\Rightarrow$ (1): From the hypotheses,
$\mathcal{R}=p\mathcal{R} \oplus {\rm rann}(p) = p\mathcal{R} \oplus
{\rm rann}(a)$ and $\mathcal{R}=q\mathcal{R} \oplus {\rm rann}(q) =
q\mathcal{R} \oplus a\mathcal{R}$. By Theorem~\ref{T 12I xR x0
existencia}, $a^{(1,2)}_{{\rm rprin}=p\mathcal{R},{\rm
rann}=q\mathcal{R}}$ exists. Thus, $a\{1,2\} \neq \emptyset$.

Using Theorems~\ref{T 12I Rx 0x existencia} and \ref{T 12I xR Rx
existencia}, the proofs of the remainder implications are similar to
the proof of (2) $\Rightarrow$ (1).
\end{proof}

\subsection{An example}\label{S An example}

Let $\mathbb{F}$ be a field with ${\rm char}(\mathbb{F})\neq 2$ and
$E_{i,j}=e_{i}e_{j}^{t} \in \mathbb{F}^{2\times2}$ for each $i, j
\in \{1,2\}$ where $e_{1}=(1,0)^{t}$ and $e_{2}=(0,1)^{t}$. Let
$A=E_{1,2}$. Then $A^{2}=0$, $A\mathbb{F}^{2\times2}={\rm rann}(A)$,
$\mathbb{F}^{2\times2}A={\rm lann}(A)$, $A\mathbb{F}^{2\times2}=\{(x_{i,j}) \in \mathbb{F}^{2\times2}:
x_{2,1}=x_{2,2}=0\}$, and $\mathbb{F}^{2\times2}A=\{(x_{i,j}) \in
\mathbb{F}^{2\times2}: x_{1,1}=x_{2,1}=0\}$.

By a direct computation, we get $A\{1\}=\{(x_{i,j}) \in
\mathbb{F}^{2\times2}: x_{2,1}=1\}$ and \[A\{2\}\!=\!\{(x_{i,j}) \in
\mathbb{F}^{2\times2}: x_{1,1}\!=\!x_{1,1}x_{2,1},
\,x_{1,2}\!=\!x_{1,1}x_{2,2}, \,x_{2,1}\!=\!x_{2,1}x_{2,1},
x_{2,2}\!=\!x_{2,1}x_{2,2}\}.\] Hence, $A\{1,2\}=\{(x_{i,j}) \in
\mathbb{F}^{2\times2}: x_{2,1}=1 \text{ and }
x_{1,2}=x_{1,1}x_{2,2}\}$.

Let $\mathcal{S}$ and $\mathcal{S}'$ be the right and the left
ideals of $\mathbb{F}^{2\times2}$ such that
$\mathbb{F}^{2\times2}=A\mathbb{F}^{2\times2}\oplus\mathcal{S}$ and
$\mathbb{F}^{2\times2}=\mathbb{F}^{2\times2}A\oplus\mathcal{S}'$.
Then $\mathcal{S}=\{(x_{i,j}) \in \mathbb{F}^{2\times2}:
x_{1,1}=x_{1,2}=0\}$ and $\mathcal{S}'=\{(x_{i,j}) \in
\mathbb{F}^{2\times2}: x_{1,2}=x_{2,2}=0\}$. We set
$\mathcal{T}=\mathcal{S}$ and $\mathcal{T}'=\mathcal{S}'$. We have
$\rho_{A\mathbb{F}^{2\times2},\mathcal{T}}(I)=\rho_{\mathcal{S}',{\rm
lann}(A)}(I)=E_{1,1}$ and $\rho_{\mathcal{S},{\rm
rann}(A)}(I)=\rho_{\mathbb{F}^{2\times2}A,\mathcal{T}'}(I)=E_{2,2}$.

Let $Z\in A\{1\}$ and $Y=(y_{i,j}) \in \mathbb{F}^{2\times2}$. Then
$\rho_{\mathcal{S},{\rm
rann}(A)}(I)Z=\rho_{\mathbb{F}^{2\times2}A,\mathcal{T}'}(I)Z=E_{2,1}+z_{2,2}E_{2,2}$,
$Z\rho_{A\mathbb{F}^{2\times2},\mathcal{T}}(I)=Z\rho_{\mathcal{S}',{\rm
lann}(A)}(I)=z_{1,1}E_{1,1}+E_{2,1}$, $\rho_{\mathcal{S},{\rm
rann}(A)}(I)Z\rho_{A\mathbb{F}^{2\times2},\mathcal{T}}(I)=\rho_{\mathbb{F}^{2\times2}A,\mathcal{T}'}(I)Z\rho_{\mathcal{S}',{\rm
lann}(A)}(I)=\rho_{\mathcal{S},{\rm
rann}(A)}(I)Z\rho_{\mathcal{S}',{\rm lann}(A)}(I)=E_{2,1}$, and
$(I-ZA)Y(I-AZ)=(y_{1,2}-z_{1,1}y_{2,2}-z_{2,2}y_{1,1}-z_{1,1}z_{2,2}y_{2,1})E_{1,2}$.

By Theorems~\ref{T 1I xaR ax0 expression}(3)-\ref{T 1I 0xa ax0
expression}(3), 
\begin{eqnarray*}
  \{aE_{1,2}+E_{2,1}: a \in \mathbb{F}\} &=& \{X\in
A\{1\} : XA\mathbb{F}^{2\times2}=\mathcal{S} \text{ and } {\rm
rann}(AX)=\mathcal{T}\} \\
&=& \{X\in A\{1\} :
\mathbb{F}^{2\times2}AX=\mathcal{S}' \text{ and } {\rm
lann}(XA)=\mathcal{T}'\} \\
   &=& \{X\in A\{1\} : {\rm
lann}(XA)=\mathcal{T}' \text{ and } {\rm rann}(AX)=\mathcal{T}\} \\
 &=& \{X\in A\{1\} : XA\mathbb{F}^{2\times2}=\mathcal{S} \text{ and }
\mathbb{F}^{2\times2}AX=\mathcal{S}'\}.
\end{eqnarray*}
As a consequence of
Theorems~\ref{T 1I xaR expression}(3) and \ref{T 1I 0xa
expression}(3), we obtain \[\{aE_{2,2}+E_{2,1}+bE_{1,2}: a, b \in
\mathbb{F}\}=\{X\in A\{1\} : XA\mathbb{F}^{2\times2}=\mathcal{S}\} =
\{X\in A\{1\} : {\rm lann}(XA)=\mathcal{T}'\},\] and from
Theorems~\ref{T 1I ax0 expression}(3) and \ref{T 1I Rax
expression}(3) we get, \[\{aE_{1,2}\,+\,E_{2,1}\,+\,bE_{1,1}: a, b
\in \mathbb{F}\}=\{X\in A\{1\} : {\rm rann}(AX)=\mathcal{T}\} =
\{X\in A\{1\} : \mathbb{F}^{2\times2}AX=\mathcal{S}'\}.\] Let
$\mathcal{Z}_{r}=\{Z\in\mathbb{F}^{2\times2}:AZ=\rho_{A\mathbb{F}^{2\times2},\mathcal{T}}(I)\}$
and
$\mathcal{Z}_{l}=\{Z\in\mathbb{F}^{2\times2}:ZA=\rho_{\mathbb{F}^{2\times2}A,\mathcal{T}'}(I)\}$.
Then $\mathcal{Z}_{r} \cap \mathcal{Z}_{l} \subseteq A\{1\}$,
$\mathcal{Z}_{r}=\{(z_{i,j}) \in \mathbb{F}^{2\times2}: z_{2,1}=1
\text{ and } z_{2,2}=0\}$ and $\mathcal{Z}_{l}=\{(z_{i,j}) \in
\mathbb{F}^{2\times2}: z_{1,1}=0 \text{ and } z_{2,1}=1\}$. Since
$\mathbb{F}^{2\times2}=A\mathbb{F}^{2\times2}\oplus\mathcal{S}$ and
$A^{2}=0$, we have $A\mathbb{F}^{2\times2}=A\mathcal{S}$. Similarly,
$\mathbb{F}^{2\times2}A=\mathcal{S}'A$. We also have ${\rm
rann}(\mathcal{S}')=\mathcal{T}$ and ${\rm
lann}(\mathcal{S})=\mathcal{T}'$. Hence, by Theorem~\ref{T 2I xR x0
existencia}(2), $A^{(2)}_{{\rm rprin}=\mathcal{S},{\rm
rann}=\mathcal{T}} \in \mathcal{S} \cap \mathcal{Z}_{r}$, by
Theorem~\ref{T 2I Rx 0x existencia}(2), $A^{(2)}_{{\rm
lprin}=\mathcal{S}',{\rm lann}=\mathcal{T}'} \in \mathcal{S}' \cap
\mathcal{Z}_{l}$, and by Theorem~\ref{T 2I xR Rx existencia}(4),
$A^{(2)}_{{\rm rprin}=\mathcal{S},{\rm lprin}=\mathcal{S}'} \in
\mathcal{S} \cap \mathcal{S}' \cap \mathcal{Z}_{r} \cap
\mathcal{Z}_{l}$. From here, \[E_{2,1}=A^{(2)}_{{\rm rprin}=\mathcal{S},{\rm
rann}=\mathcal{T}}=A^{(2)}_{{\rm lprin}=\mathcal{S}',{\rm
lann}=\mathcal{T}'}=A^{(2)}_{{\rm rprin}=\mathcal{S},{\rm
lprin}=\mathcal{S}'}.\] Applying Theorem~\ref{T 2I 0x x0
existencia}(6), we obtain $E_{2,1}=A^{(2)}_{{\rm lann}=\mathcal{T}',{\rm
rann}=\mathcal{T}}$.

By parts (8) of Theorems~\ref{T 12I xR x0 expression}-\ref{T 12I 0x
x0 expression}, \[E_{2,1}=A^{(1,2)}_{{\rm rprin}=\mathcal{S},{\rm
rann}=\mathcal{T}}=A^{(1,2)}_{{\rm lprin}=\mathcal{S}',{\rm
lann}=\mathcal{T}'}=A^{(1,2)}_{{\rm rprin}=\mathcal{S},{\rm
lprin}=\mathcal{S}'}=A^{(1,2)}_{{\rm lann}=\mathcal{T}',{\rm
rann}=\mathcal{T}}.\] By Theorems~\ref{T 12I xR expression}(6) and
\ref{T 12I 0x expression}(4), \[\{E_{2,1}+aE_{2,2}: a \in
\mathbb{F}\}=\{X\in A\{1,2\} : XA\mathbb{F}^{2\times2}=\mathcal{S}\}
= \{X\in A\{1,2\} : {\rm lann}(XA)=\mathcal{T}'\}.\] By
Theorems~\ref{T 12I x0 expression}(6) and \ref{T 12I Rx
expression}(4), \[\{aE_{1,1}+E_{2,1}: a \in \mathbb{F}\}=\{X\in
A\{1,2\} : {\rm rann}(AX)=\mathcal{T}\} = \{X\in A\{1,2\} :
\mathbb{F}^{2\times2}AX=\mathcal{S}'\}.\]

It is easy to see that
\[\{B\in\mathbb{F}^{2\times2}:\mathcal{S}=B\mathbb{F}^{2\times2}\}=\{B\in\mathbb{F}^{2\times2}:\mathcal{T}'={\rm
lann}(B)\}=\{(b_{i,j})\in\mathbb{F}^{2\times2}:b_{1,1}\!=\!b_{1,2}\!=\!0
\text{ and } (b_{2,1},b_{2,2}) \neq 0\}\] and
\[\{C\in\mathbb{F}^{2\times2}:\mathcal{T}={\rm
rann}(C)\}=\{C\in\mathbb{F}^{2\times2}:\mathcal{S}'=\mathbb{F}^{2\times2}C\}=\{(c_{i,j})\in\mathbb{F}^{2\times2}:c_{1,2}\!=\!c_{2,2}\!=\!0
\text{ and } (c_{1,1},c_{2,1}) \neq 0\}.\]

We note that we have obtained the unique $\{2\}$-inverse
corresponding to the right ideals $\mathcal{S}$ and $\mathcal{T}$
such that
$\mathbb{F}^{2\times2}=A\mathbb{F}^{2\times2}\oplus\mathcal{T}$ and
$\mathbb{F}^{2\times2}=\mathcal{S}\oplus{\rm rann}(A)$ (resp. left
ideals $\mathcal{S}'$ and $\mathcal{T}'$ such that
$\mathbb{F}^{2\times2}=\mathbb{F}^{2\times2}A\oplus\mathcal{T}'$ and
$\mathbb{F}^{2\times2}=\mathcal{S}'\oplus{\rm lann}(A)$). There are
other $\{2\}$-inverses with other principal/annihilator ideals. For
example, as in \cite[Example 2.5]{Drazin (2012)}, we can consider
$B=(\lambda,1)^{t}(\alpha,\beta)$ and $C=(\gamma,\delta)^{t}(1,\mu)$
with $(\alpha,\beta), (\gamma,\delta) \in
\mathbb{F}^{2}\setminus\{0\}$ and $\lambda, \mu \in \mathbb{F}$. The
conditions that any pair of ideals must satisfy are given in
Theorems~\ref{T 2I xR x0 existencia}-\ref{T 2I 0x x0 existencia}.

\bigskip
\textbf{Acknowledgements.} The author thanks the reviewer for the useful observations. This research has been supported by
Grants PIP 11220220100112CO (CONICET) and PROIPRO 03-2823 (UNSL).

\end{document}